\documentclass[11pt]{amsart}
\usepackage[latin1]{inputenc}
\usepackage[active]{srcltx}
\usepackage{bbm}
\usepackage{fullpage}
\usepackage{amsfonts,enumerate}
\usepackage[mathscr]{euscript}
\usepackage{epsfig}
\usepackage{latexsym}
\usepackage{leftindex}
\usepackage[all]{xy}
\usepackage{amssymb}
\usepackage{amsthm}
\usepackage{float}
\usepackage{amsmath}
\usepackage{amsxtra}
\usepackage{xcolor}
\usepackage[colorlinks]{hyperref}
\usepackage{fancyvrb}
\usepackage{tikz}
\hypersetup{citecolor=blue}
\usepackage{algorithmic}
\usepackage{subcaption}
\usepackage{graphicx}
\usepackage[textwidth=2.5cm]{todonotes}
\usepackage{multirow}

.tfm
.tfm

\theoremstyle{plain}
\newtheorem{proposition}{Proposition}[section]
\newtheorem{theorem}[proposition]{Theorem}
\newtheorem{corollary}[proposition]{Corollary}

\newtheorem{lemma}[proposition]{Lemma}
\newtheorem{definition}[proposition]{Definition}
\newtheorem{alg}[proposition]{Algorithm}
\newtheorem{conjecture}[proposition]{Conjecture}

\newtheorem{ass}{Assumption}
\newtheorem*{thm*}{Theorem}
\newtheorem*{lemma*}{Lemma}
\newtheorem{thmA}{Theorem}
\newtheorem{thmB}[thmA]{Theorem}
\newtheorem{thmC}[thmA]{Theorem}

\newtheorem*{THM1}{Theorem~\ref{thm:extension2}}
\newtheorem*{prop*}{Proposition}

\theoremstyle{definition}
\newtheorem{example}{Example}
\newtheorem*{example*}{Example}

\theoremstyle{remark}

\newtheorem{remark}[proposition]{Remark}
\newtheorem*{remarkintro}{Remark}

\numberwithin{table}{section}

\DeclareMathOperator{\Id}{Id}
\DeclareMathOperator{\den}{den}
\DeclareMathOperator{\Perm}{Perm}
\DeclareMathOperator{\Real}{Re}

\DeclareMathOperator{\denominator}{den}

\DeclareMathOperator{\Frob}{Frob}
\DeclareMathOperator{\lcm}{lcm}

\DeclareMathOperator{\Gal}{Gal}
\DeclareMathOperator{\Aut}{Aut}
\DeclareMathOperator{\Ind}{Ind}

\DeclareMathOperator{\Jac}{Jac}
\DeclareMathOperator{\Disc}{disc}
\DeclareMathOperator{\Het}{H_{et,\chi}^{1}(\Cnsp, \QQ_{\ell})}
\DeclareMathOperator{\HetL}{H_{et,\chi}^{1}(\Cnsp, F_\lambda)}
\DeclareMathOperator{\JacMot}{\mathbf J}
\DeclareMathOperator{\len}{len}
\DeclareMathOperator{\Teich}{Teich}
\DeclareMathOperator{\car}{char}
\DeclareMathOperator{\Sp}{Sp}

\DeclareMathOperator{\ptame}{pt}
\DeclareMathOperator{\pwild}{pw}

\newcommand{\calO}{\mathcal O}
\newcommand{\D}{\mathcal D}

\newcommand{\F}{\mathbb F}
\newcommand{\bfC}{\mathcal C}

\newcommand{\JJ}{\mathscr J}
\newcommand{\T}{\mathcal T}
\newcommand{\B}{\mathcal B}
\newcommand{\Ab}{\mathcal A}
\newcommand{\Om}{{\mathscr{O}}}

\newcommand{\GL}{{\rm GL}}

\newcommand{\Cnsp}{\widehat{\C}}

\makeatletter
\newcommand{\verbatimfont}[1]{\renewcommand{\verbatim@font}{\ttfamily#1}}
\makeatother
\newcommand{\HGM}{\mathcal H}
\newcommand{\hgm}{\HGM((a,b),(c,d)|z)}

\newcommand{\hgms}{\HGM((a,b),(c,d)|\z-spec)}
\newcommand{\hgmps}{\HGM((a',b'),(c',d')|\z-spec)}
\newcommand{\Euler}{H^1_{\chi_0}(\Cnsp,F)^{\text{new}}}

\newcommand{\orbit}{{\mathcal O}}

\newcommand{\Gammap}{{\mathbf \Gamma}}
\newcommand{\psip}{{\pmb \psi}}
\newcommand{\gaussp}{{\pmb g}}
\newcommand{\hgmpadic}{{\pmb H}}
\newcommand{\chip}{\chi_{\id{p}}}
\newcommand{\veca}{\pmb \alpha}
\newcommand{\vecb}{\pmb \beta}
\newcommand{\vecg}{\pmb \gamma}
\newcommand{\vecd}{\pmb \delta}
\newcommand{\vecx}{\pmb x}
\newcommand{\vecy}{\pmb y}
\newcommand{\twist}{\varkappa}
\newcommand{\Count}{\mathcal {N}}

\def\Char{\varphi}

\def\irr{(\bf{Irr})}

\def\ZZ{\mathbb Z}

\def\FF{\mathbb F}
\def\QQ{\mathbb Q}
\def\PP{\mathbb P}

\def\PP{\mathbb P}

\def\CC{\mathbb C}

\def\gen{\varpi}
\def\gent{\chi_{\id{p}}}
\def\K{\mathfrak R}
\def\Yst{\mathcal{Y}^{{\rm st}}}
\def\C{\mathcal C}

\def\z-spec{\xi}
\def\chil{\chi_{0,\lambda}}
\newcommand{\hyp}[3]{{{\mathcal H}({#1},{#2}{\,|\,}{#3})}}
\newcommand{\hgmvec}{\hyp{\veca}{\vecb}{z}}
\newcommand{\hgmvecz}{\hyp{\veca}{\vecb}{\z-spec}}

\newcommand{\mubb}{\mbox{$\raisebox{-0.59ex}
  {$l$}\hspace{-0.18em}\mu\hspace{-0.88em}\raisebox{-0.98ex}{\scalebox{2}
  {$\color{white}.$}}\hspace{-0.416em}\raisebox{+0.88ex}
  {$\color{white}.$}\hspace{0.46em}$}{}}

\def\<#1>{{\left\langle{#1}\right\rangle}}

\def\calH{{\mathcal H}}
\def\calP{{\mathcal P}}
\def\Z{{\mathbb Z}}             
\def\Q{{\mathbb Q}}             

\def\id#1{{\mathfrak{#1}}}      
\def\idL#1{\hat{\mathfrak{#1}}}
\def\normid#1{{\norm{\id{#1}}}}
\DeclareMathOperator{\norm}{{\mathscr N}}

\DeclareMathOperator{\trace}{{\mathrm{Tr}}}

\let\kro\dkro

\usetikzlibrary{calc}

\newcommand\FundGroup[1][1cm]{%

\begin{tikzpicture}[every node={circle,thick, inner sep=1.5pt}]

\filldraw [black] (0,0) circle (2pt);
\node at (0,0) (A) [below] {$z_0$};
\filldraw [black] (3,0) circle (2pt);
\node at (3,0) (B) [above] {$1$};
\filldraw [black] (-3,0) circle (2pt);
\node at (-3,0) (C) [above] {$0$};

\node at (1,-.3) (D) [below] {$\gamma_1$};
\node at (-1,-.3) (E) [below] {$\gamma_0$};
\node at (0,-1) (F) [below] {$\gamma_{\infty}$};

\draw [->](0,0) .. controls (3,-1) and (4,-1) .. (4,0);
\draw (4,0) .. controls (4,1) and (3,1) .. (0,0);

\draw (0,0) .. controls (-3,-1) and (-4,-1) .. (-4,0);
\draw [<-](-4,0) .. controls (-4,1) and (-3,1) .. (0,0);

\draw (0,0) .. controls (0,1) and (2,1.5) .. (3,1.5);
\draw (0,0) .. controls (0,1) and (-2,1.5) .. (-3,1.5);
\draw (3,1.5) .. controls (4,1.5) and (5,1) .. (5,0);
\draw (-3,1.5) .. controls (-4,1.5) and (-5,1) .. (-5,0);
\draw [->](5,0) .. controls (5,-1) and (3,-1.5) .. (0,-1.5);
\draw (-5,0) .. controls (-5,-1) and (-3,-1.5) .. (0,-1.5);
\end{tikzpicture}
}

\usetikzlibrary{snakes,fit,positioning,calc}

\definecolor{amethyst}{rgb}{0.6, 0.4, 0.8}
\definecolor{atomictangerine}{rgb}{1.0, 0.6, 0.4}
\definecolor{deeppeach}{rgb}{1.0, 0.8, 0.64}
\definecolor{eggshell}{rgb}{0.94, 0.92, 0.84}
\definecolor{lightapricot}{rgb}{0.99, 0.84, 0.69}
\definecolor{lemonchiffon}{rgb}{1.0, 0.98, 0.8}
\definecolor{roundabout}{rgb}{1.0, 0.91, 0.75}
\definecolor{atomictangerine}{rgb}{1.0, 0.6, 0.4}

\def\rootsep{0.03}               
\def\clustersep{0.06}            
\def\cnamescale{0.4}             
\def\cdepthscale{0.4}            
\def\cltopskip{1pt}              
\def\clbottomskip{1pt}           

\def\rootscale{0.5}   \def\rootcolor{amethyst}

\usepackage{pbox}
\tikzset{
  root/.style = {circle,scale=\rootscale,ball color=\rootcolor},
    rc/.style 2 args = {right=#1*1.5*\clustersep of {#2.east|-first},root}, rr/.style = {right=\rootsep of {#1.east|-first},root},
  roott/.style = {circle,inner sep=-2pt,minimum size=5pt,black,font=\ttfamily\footnotesize},
    rct/.style 2 args = {right=#1*1.5*\clustersep of {#2.east|-first},roott}, rrt/.style = {right=\rootsep of {#1.east|-first},roott},
  rootA/.style = {circle,scale=\rootscaleA,ball color=\rootcolorA},
    rcA/.style 2 args = {right=#1*1.5*\clustersep of {#2.east|-first},rootA}, rrA/.style = {right=\rootsep of {#1.east|-first},rootA},
  rootB/.style = {circle,scale=\rootscaleB,ball color=\rootcolorB},
    rcB/.style 2 args = {right=#1*1.5*\clustersep of {#2.east|-first},rootB}, rrB/.style = {right=\rootsep of {#1.east|-first},rootB},
  rootC/.style = {circle,scale=\rootscaleC,ball color=\rootcolorC},
    rcC/.style 2 args = {right=#1*1.5*\clustersep of {#2.east|-first},rootC}, rrC/.style = {right=\rootsep of {#1.east|-first},rootC},
  rootD/.style = {circle,scale=\rootscaleD,ball color=\rootcolorD},
    rcD/.style 2 args = {right=#1*1.5*\clustersep of {#2.east|-first},rootD}, rrD/.style = {right=\rootsep of {#1.east|-first},rootD},
  cluster/.style = {draw=blue!70,thick,rounded corners,inner sep=22*\clustersep,outer xsep=22*\clustersep,fit=#1},
  clabel/.style  = {anchor=west,scale=\cdepthscale,black,inner sep=0,outer xsep=1,outer ysep=0},
  clabelL/.style = {above right=-\clustersep of #1t.north east,clabel},
  clabelD/.style = {below right=-\clustersep of #1t.south east,clabel},
  clouter/.style = {inner sep=0,outer sep=0,fit=#1}
}

\def\Cluster #1 = #2;{\node[cluster=#2] (#1) {};}
\def\ClusterL #1[#2] = #3;{
  \node[cluster=#3] (#1t) {}; \node[clabelL=#1] (#1l) {$#2$}; \node[clouter=(#1t)(#1l)] (#1) {};}
\def\ClusterD #1[#2] = #3;{
  \node[cluster=#3] (#1t) {}; \node[clabelD=#1] (#1d) {$#2$}; \node[clouter=(#1t)(#1d)] (#1) {};}
\def\ClusterLD #1[#2][#3] = #4;{
  \node[cluster=#4] (#1t) {}; \node[clabelL=#1] (#1l) {$#2$}; 
  \node[clabelD=#1] (#1d) {$#3$}; \node[clouter=(#1t)(#1l)(#1d)] (#1) {};}
\def\ClusterLDName #1[#2][#3][#4] = #5;{
  \node[cluster=#5] (#1t) {}; \node[clabelL=#1] (#1l) {$#2$}; 
  \node[clabelD=#1] (#1d) {$#3$}; 
  \node[scale=\cnamescale,above=\clustersep/3 of #1t,inner sep=0, outer sep=0] (#1n) {$#4$}; 
  \node[clouter=(#1l)(#1d)(#1t)] (#1) {};}

\newcommand{\Root}[4][]{
  \ifx\relax#2\relax\node[rr#1=#3] (#4) {};\else\node[rc#1={#2}{#3}] (#4) {};\fi}
\newcommand{\RootT}[5][]{
  \ifx\relax#2\relax\node[rrt#1=#3] (#4) {#5};\else\node[rct#1={#2}{#3}] (#4) {#5};\fi}

\long\def\clusterpicture#1\endclusterpicture{\pb{\vbox to \cltopskip{\vfill}\\%
  \begin{tikzpicture}\node[coordinate] (first) {};#1\end{tikzpicture}\\[-11pt]\vbox to \clbottomskip{\vfill}}}   
\long\def\clusterpictureopt#1#2\endclusterpicture{\pb{\vbox to \cltopskip{\vfill}\\%
  \begin{tikzpicture}[#1]\node[coordinate] (first) {};#2\end{tikzpicture}\\[-11pt]\vbox to \clbottomskip{\vfill}}}   

\def\pb#1{\pbox[c]{\textwidth}{\hfil #1\hfil}}

\begin{document}
	
\title{On rank $2$ hypergeometric motives}
	
\author{Franco Golfieri Madriaga}
\address[Golfieri]{FCEFN Department of Mathematics and Chemistry, Universidad Nacional de San Juan. CP: 5400, San Juan, Argentina}
\email{francogolfieri87@gmail.com}
\author{Ariel Pacetti}

\address[Pacetti]{Center for Research and Development in Mathematics and
	Applications (CIDMA), Department of Mathematics, University of
	Aveiro, 3810-193 Aveiro, Portugal} \email{apacetti@ua.pt}
\thanks{The second author was supported by CIDMA under the Portuguese Foundation for Science and Technology 
(FCT, https://ror.org/00snfqn58) Multi-Annual Financing Program for R\&D Units, grants UID/4106/2025 and UID/PRR/4106/2025.}

\author{Fernando Rodriguez-Villegas}
\address[Villegas]{The Abdus Salam International Centre for Theoretical Physics, Strada Costiera 11, 34151 Trieste, Italy}
\email{villegas@ictp.it}

\address[Lorenzo-García]{Aix Marseille Univ, CNRS, I2M, Marseille, France}
\email{elisa.lorenzo-garcia@univ-amu.fr}

\dedicatory{with an Appendix by Elisa Lorenzo Garc\'ia and Ariel Pacetti}
      
\begin{abstract} Hypergeometric motives are family of motives
  associated to hypergeometric local systems. Their special features,
  in particular their rigidity, makes them more tractable than general
  motives. In the present article we prove most of the properties that
  they are expected to satisfy in the rank $2$ case.
\end{abstract}
	
\maketitle
{
  	\hypersetup{linkcolor=black}
	\setcounter{tocdepth}{1}
        \tableofcontents
}

\section{Introduction}

The theory of hypergeometric motives has its origins in the study of
the complex hypergeometric series 
\[
  {}_2F_1(a,b;c|z)=\sum_{n=0}^\infty \frac{(a)_n(b)_n}{(c)_nn!}z^n,
\]
for $z$ a complex number satisfying $|z|<1$ (as done by Gauss
\cite{MR0616131}, Kummer \cite{MR1578088}, Goursat \cite{MR1508709} et
al). Riemann (in \cite{MR221900}, ``Contribution a la th\'eorie des
fonctions repr\'esentables par la s\'erie de Gauss $F(a, b, c, x)$'')
studied the differential equation satisfied by this series. This
differential equation yields a \emph{rank two local system} on
$\PP^1\setminus\{0,1,\infty\}$
 which in turn determines a
$2$-dimensional  \emph{monodromy
  representation}
\[
\rho: \pi_1(\PP^1\setminus\{0,1,\infty\},z_0) \to \GL_2(\CC),
\]
describing how a chosen basis of local solutions at a regular point
$z_0$ changes while looping around appropriately around the missing
points $0,1, \infty$. Riemann realized that the monodromy
representation is a natural way to encode information of the
differential equation and its solutions.

We denote by $M_0$, $M_1$ and $M_\infty$ the local monodromy matrices
around $0$, $1$ and $\infty$ respectively such that
$$
M_0M_1M_\infty=I_2,
$$
where $I_2$ is the identity matrix.

We are interested in algebraic incarnations of the hypergeometric
series, so for the rest of the introduction we let $a,b,c$ be rational
numbers and let $N$ be their common denominator. Under this
assumption, a result of Levelt implies that we may choose a basis of
local solutions such that the monodromy
takes values in the number field $F:=\Q(\zeta_N)$, with $\zeta_N$ a
primitive $N$-th root of unity.  Given a prime ideal $\id{p}$ of
$\Z[\zeta_N]$ one can extend the monodromy representation to a
\emph{geometric} continuous representation
\begin{equation}
  \label{eq:geom-rep}
\rho_{\id{p}}:\Gal(\overline{\Q(z)}/\overline{\Q}(z)) \to \GL_2(F_{\id{p}}),  
\end{equation}
where $F_{\id{p}}$ denotes the $\id{p}$-adic completion of
$F$ at $\id{p}$ (as explained in \S 2) and $z$ is a
variable.  

We expect the representation~\eqref{eq:geom-rep} to be \emph{motivic},
namely to match the restriction of a representation of
$\Gal(\overline{\Q(z)}/F(z))$ to
$\Gal(\overline{\Q(z)}/\overline{\Q}(z))$ arising from a motive
defined over $F(z)$. In particular and more precisely, for every
specialization $z=\z-spec \in L$ of the parameter in a finite
extension $L/F$ there should exist a pure motive
$\HGM(a,b;c\,|\,\z-spec)$ defined over $L$ related to Gauss's
hypergeometric series. Concretely, (see
Conjecture~\ref{conjecture:field-of-defi}) we should be able to
determine from the defining hypergeometric data the characteristic
polynomial of Frobenius, under the compatible system of Galois
representations determined by $\HGM(a,b;c\,|\,\z-spec)$, for primes of
good reduction, as well as obtain information on the action of inertia.

In more detail, let $H$ denote the set of elements
$i \in (\Z/N)^\times$ fixing $c$ and the set $\{a,b\}$ modulo~$\Z$;
i.e., $ic\equiv c \bmod \Z$ and $i\{a,b\}\equiv\{a,b\} \bmod \Z$. Let
$K=F^H$, after identifying $H$ with a subgroup of $\Gal(F/\Q)$ in a
natural way. Then the following should hold (see Corollary~\ref{coro:conjecture}):

{\it
\begin{enumerate}[(i)]
\item
  \label{prop-1}
  The motive $\HGM(a,b;c\,|\,z)$ has a model defined over its
  field of moduli~$K(z)$.

\item For generic values $\z-spec \in K$ the motive $\HGM(a,b;c\,|\,\z-spec)$
  defined over $K$ resulting from specializing~(i) has coefficient field~$K$.

\noindent 
For $\z-spec \neq 0,1 \in F$ we have the following.

\item
\label{prop-2}
 The prime ideals $\id{q}$ of $F$ not diving $N, \z-spec, \z-spec^{-1}$ and $\z-spec-1$
are primes of good reduction for $\HGM(a,b;c\,|\,\z-spec)$.
  
\item
  \label{prop-3}
  Let $\id{q}$ be a prime ideal of $F$ of good reduction for
  $\HGM(a,b;c\,|\,\z-spec)$. Then the trace of $\Frob_{\id{q}}$ on
  $\HGM(a,b;c\,|\,\z-spec)$ matches the finite field analogue of the function
  ${}_2F_1(a,b;c\,|\,\z-spec)$ (defined in~ \S4).
  
\item
  \label{prop-4}
  The primes $\id{q}$ of $F$ not dividing $N$ but dividing
  $\z-spec, \z-spec^{-1}$ or $\z-spec-1$ are (at worst) \emph{tame} primes
  for~$\rho_{\id{p}}$. Moreover, the conjugacy class of a generator of
  the image of inertia~$\rho_{\id{p}}(I_{\id{q}})$ (for any prime
  ideal $\id{p}$ whose residual characteristic is prime to that of
  $\id{q}$) is that of $M_s^{k_s}$, where $s=0, 1, \infty$
  and
$$
k_0:=v_{\id{q}}(\z-spec), \qquad k_1:=v_{\id{q}}(\z-spec-1),\qquad
k_{\infty}:=v_{\id{q}}(\z-spec^{-1}).
$$
  
\end{enumerate}
}
It is expected that all of the above properties actually hold for
hypergeometric motives of any rank (see~\S3).
In general if the rank is~$n$ the motive can be described as an
isotypical component of the middle cohomology of a corresponding Euler
variety
\begin{equation}
  \label{eq:euler-high-dim}
V:\qquad y^N = \prod_{i=1}^{n-1}x_i^{A_i}(1-x_i)^{B_i}(1-zx_1 \cdots
x_{n-1})^{A_n}
\end{equation}
under the action of the automorphism $y\mapsto \zeta_N y$
(see~\cite{Voight} for results on compactifications of these
varieties). For $K=\Q$ the motive may also be seen as the top weight
piece of the middle cohomology of a hypersurface in a torus
(see~\cite{BCM},~\cite{MR4442789}). Partial results on properties
(i)-(v) were obtained by Katz in \cite{MR1366651} (see also
\cite{MR4493579}). 

To define our hypergeometric motive we reason heuristically as
follows. Assume first that $d=1$. The hypergeometric series ${}_2F_1$
has the integral presentation
\begin{equation}
\label{eq:integr-repn}
\frac{\Gamma(b)\Gamma(c-b)}{\Gamma(c)} \leftindex_2{F}_1(a,b;c\,|\,z) = \int_0^1 x^{b-1}(1-x)^{c-b-1}(1-zx)^{-a} dx.
\end{equation}
We may interpret this formula as an identity of periods. The right
hand side (the integral) is a period of the holomorphic differential
$dx/y$ of the Euler curve $\C$ (see \eqref{curva2}); the first factor on the left hand
side (a product of Gamma values) is a period of a Jacobi motive
$J((a,b),(c,1))$ (defined in \S\ref{section:Jacobi}) and the second
factor (the series ${}_2F_1$ itself) should correspond to a period of
our hypergeometric motive $\hgmvecz$.

Suppose for the rest of the introduction that $\veca$ and $\vecb$ have
rank $2$. We prove \eqref{prop-2} and \eqref{prop-3} (see Theorem~\ref{thm:Trace-match})
\begin{thmA}
\label{thm:thmA}
Let $\id{p}$ be a prime ideal of $\Q(\zeta_N)$ satisfying~\eqref{prop-2}. Then
${\hgmvec}$ has good reduction at $\id{p}$
and the trace of the Frobenius automorphism $\Frob_{\id{p}}$ acting on
${\hgmvec}$ equals
$H_{\id{p}}(\veca,\vecb\,|\,z)$.
\end{thmA}

\begin{definition}
  \label{def:H}
  Let $H$ be the subgroup of elements in $(\Z/N)^\times$ that fix
  (under multiplication) the multisets
  $\{\alpha_1,\ldots,\alpha_n\}$ and $\{\beta_1,\ldots,\beta_n\}$
  modulo $\Z$.
\end{definition}
The group $H$ is naturally
identified with a subgroup of $\Gal(F/\Q)$. Since
$-1\in (\Z/N\Z)^\times$ corresponds to complex conjugation,
$K:=F^H\subseteq \Q(\zeta_N)^+$ (or equivalently, it is totally real)
if and only if $-1\in H$.

Theorems~\ref{thm:extension1} and \ref{thm:extension2} imply the
following version of \eqref{prop-1} (see also Corollary~\ref{coro:extension-motive}).
\begin{thmB}
\label{thm:thmB}
If $-1\in H$ then for all specializations $\z-spec$ of the parameter
$\hgmvecz$ maybe defined over the totally real field $K$.
\end{thmB}
Under some extra hypothesis (see
Proposition~\ref{prop:coef-lowerbound}) we can prove that for generic
values of the parameter $\z-spec\in K$, the motive $\hgmvecz$ has
coefficient field $K$. It seems to be true that always the field
of definition of the motive is the same as the coefficient field (even
when it is smaller than $K$, like in Example~\ref{example:reducible}).

\medskip

Let $p$ be a rational prime. For $a,b$ rational numbers, denote by
$a\sim_p b$ the relation defined by the condition that the denominator
of $a-b$ is a power of $p$. Extend this relation to vectors
component-wise.

\begin{thmC}
  Let $\veca,\vecb$ and $\veca',\vecb'$ be
  generic pairs of rational numbers (see~\ref{defi:generic}) such that
  $\veca \sim_p \veca'$ and
  $\vecb \sim_p \vecb'$ for some $p$ prime. Then
  $\hgms$ is congruent to
  $\hgmps$ modulo any prime ideal
  $\id{p}$ of  a common field of definition dividing $p$.
\end{thmC}

The previous results has two interesting applications: on the one
hand, it provides a ``lowering the level'' result, given by removing a
prime ideal $p$ from the denominator of any hypergeometric motive's
parameters. This is very useful for example in the study of
Diophantine equations (as used in~\cite{GP} to study solutions of
the generalized Fermat equation).

On the other hand, it gives a ``raising the level'' result, since
given a pair of rational parameters $\veca,\vecb$ we can
add a rational number $r/p^s$ to each $\alpha_i$ and $\beta_i$, with
$r,s$ positive integers and $p$ a prime not dividing their
denominator. Note that the resulting motive will a priory have wild
ramification at primes dividing $p$.  
\begin{example}
\label{example:Legendre}
Consider the hypergeometric motive with parameters $(1/2,1/2),(1,1)$,
corresponding (see Example~\ref{rem:Legendre}) to a quadratic twist by
$\kro{-4}{\cdot}$ of Legendre's family of elliptic curves
\[
E_z:y^2=x(1-x)(1-zx).
\]
Adding/subtracting $1/3$ to the first parameters, we obtain the
rational motive with parameters $(1/6,-1/6),(1,1)$
corresponding (as proved in \cite{Cohen}) to the family of elliptic
curves with equation
\[
\widetilde{E_z}: y^2=x^3-\frac{x^2}{4}+\frac{z}{432}.
\]
Then for any specialization $\z-spec$ of the parameter $z$ the elliptic
curves are congruent modulo $3$ as implied by Theorem~\ref{thm:thmB}.
\end{example}

We mostly restricted to rank two hypergeometric motives because of the
applications to the study of solutions to the generalized Fermat's
equation, following the program described by Darmon in \cite{Darmon},
as explained in the article \cite{GP}. We expect, however, that our
proofs extend to the arbitrary rank case by considering the general
Euler variety (\ref{eq:euler-high-dim}).

\vspace{2pt}

The article is organized as follows: section \S\ref{section:rank1} is
devoted to rank $1$ hypergeometric motives. Although the theory of
rank $1$ motives is quite elementary (solutions to the differential
equation are algebraic), it gives a flavor of the general theory.

Section \S\ref{section:hypergeometric-reps} starts recalling the
definition of the rank $n$ hypergeometric series, the differential
equation it satisfies, and some of its well known properties. It
includes a detailed description of the construction of the geometric
representation~(\ref{eq:geom-rep}) extending the monodromy one.

Section~\S\ref{section:HGM} contains a discussion of arbitrary rank
$n$ hypergeometric motives.  Following computations of David Roberts
and the third author, analogues of properties $(i)$ to $(v)$ are
expected to hold.
The section includes a proof of a weaker version of statement
$(ii)$ for general rank $n$ hypergeometric motives (based on
\cite{MR1116916}). We end the section by recalling an algorithm to
compute the Hodge numbers of the hypergeometric motive of
parameters~$\veca,\vecb$ and include different examples.

Section \S\ref{section:Jacobi} is a short survey on Jacobi motives
(following \cite{MR3838691} and \cite{MR0051263}). It includes results
needed to interpret the product/quotient of Gamma values
in~\eqref{eq:integr-repn} as a Jacobi motive.
Section~\S\ref{section:finite} gives the definition of the finite
field analogue $H_q(\veca,\vecb|z)$ of the hypergeometric
series~${}_nF_{n-1}$ (as studied in \cite{MR879564} and
\cite{katz}). Proposition~\ref{prop:props-hyperg-sum} summarizes the
main properties satisfied by finite hypergeometric sums needed in the
present article. There is an alternate $p$-adic definition of finite
hypergeometric sums in terms of the $p$-adic Gamma function (as
defined by Morita in
\cite{zbMATH03482465}). Theorem~\ref{thm:algebricity} proves that both
definitions coincide (using the Gross-Koblitz formula). The section
ends (Theorem~\ref{thm:hypergeometric-char-sum}) with a relation
between hypergeometric character sums and finite hypergeometric
series. The result plays a crucial role while computing the zeta
function of the Euler curve.

In Section~\S\ref{section:superelliptic} we compute the zeta function
of a general superelliptic curve, relating the number of points of its
isotypical components to finite character sums (see
Theorem~\ref{thm:trace-equality}).  In Section~\S\ref{section:hgm-2}
we focus on the study of rank $2$ hypergeometric motives. Previous
results on rank $2$ motives (like \cite{MR2005278} and
\cite{MR4493579}) always assumed that $d=1$. This a priory
``innocent'' assumption not only excludes many interesting motives,
but also has deep consequences since it implies that Euler's curve is
absolutely irreducible. For general values of $d$ this is no longer
true (as proved in Example~\ref{example:reducible}); Euler's curve can
have any number of irreducible components. To circumvent this problem,
we start studying the case of irreducible curves, and use them to
define a general rank $2$ hypergeometric motive in
Definition~\ref{defi:general-HGM}.  In \S\ref{section:rank2-hodge} we
study the possible Hodge numbers of a rank $2$ hypergeometric motive,
a very reach theory. In Example~\ref{example:shimura} we compute the
Hodge numbers of some curves studied by Shimura (in
\cite{MR176113}). The last part of the section is devoted to computing
the zeta function of the Euler curve $\C$ (and its factorization
according to the action of $\mubb_N$). The main result
(Theorem~\ref{thm:Trace-match}) gives a precise relation between
$H_{\id{p}}(\veca,\vecb|z)$ and the trace of a Frobenius element
$\Frob_{\id{p}}$ acting on $H^1_{\text{\'et},\chi_0}(\Cnsp,F_\lambda)$
(for $\id{p}$ a prime ideal of $F$ of good reduction for $\C$), where
$\Cnsp$ denotes the desingularization of the Euler curve. The result
is an arithmetic counterpart to~\eqref{eq:integr-repn}.

Section \S\ref{section:extension} is devoted to prove results on the
field of definition of the motive (property (i) of the
introduction). Our main result reads.
\begin{THM1}
  Let $(a,b),(c,d)$ be generic parameters such that $a+b$ and $c+d$
  are integers and $\irr$ holds. Then $\hgm$ is defined over
  $K=F^H$. For $\z-spec \in K$, the coefficient field of $\hgms$ is
  contained in $K$.
\end{THM1}
Its proof is interesting on its own, since it consists on constructing
involutions of Euler's curve. The section includes a study of the
Galois representation attached to specializations of the
hypergeometric motive. We prove irreducibility results
(Theorem~\ref{thm:irred-spec}) and extension results
(Theorem~\ref{thm:extension1} and
Corollary~\ref{coro:extension-motive}).

Section \S\ref{section:HGM-tot-real} is devoted to the study of
hypergeometric motives defined over totally real fields (meaning that
$-1 \in H$). In this case we can prove a stronger irreducibility
result (Theorem~\ref{thm:real-irred}).  The last section
\S\ref{section:congruences} contains the explicit statement and proof
of congruences between hypergeometric motives. Although the proof is
elementary, the result has deep implications in Darmon's program (as
exploited in \cite{GP}). The present article ends with two
appendices. The first one studies reduction types of the Euler curve
and is used to relate the monodromy of the differential equation with
the action of inertia on the hypergeometric motive $\hgmvec$.  It also
includes a proof of \eqref{prop-4} (see Theorem~\ref{thm:HGM-0} and
\ref{thm:HGM-1-infty}).

The second
appendix addresses the following problem: keeping the previous
notation, let $\id{p}$ be a prime ideal of $F$ not dividing $N$ but
(say) dividing the numerator of $\z-spec$. Then it is a priory a tame
prime for the motive, but it may actually be of good reduction (as
follows from $(4)$, since the matrix $M_0$ might have finite
order). In this case the stated formula (proven in
Theorem~\ref{thm:Trace-match}) does not apply. One can still compute
the trace of the Frobenius element geometrically, by computing the
stable model of Euler's equation (as explained in the appendix). It
would be interesting to have an explicit description for wild primes.

The third author wrote a script to compute $H_q(\veca,\vecb|\z-spec)$
$p$-adically in terms of the $p$-adic Gamma function (a different
algorithm to compute the $L$-series of a hypergeometric motive is
given in \cite{MR4235111} and \cite{MR4858169}). The script can be
downloaded from the github repository
\begin{center}
\url{https://github.com/frvillegas/frvmath}.  
\end{center}
\begin{example}
\label{example:1}  
Here is an example of how the code works. Consider the hypergeometric
motive $\calH_z$ with parameters~$\veca=(1/8,7/8), \vecb=(3/8,5/8)$,
Tate twisted so that it has weight one (or equivalently so that it is
an effective motive as explained in \S \ref{section:hodge}).  A priory
$\calH_z$ is defined over $\Q(\sqrt{2})$ as explained in
Example~\ref{example:genericity-condition}.

 First we take $z=9$ and  compute the action of Frobenius for the prime $p=7$.
\begin{verbatim}
? hgm(9,[1/8,7/8],[3/8,5/8],7)
3*7^-1 + 6 + 6*7 + 6*7^2 + 6*7^3 + 6*7^4 + 6*7^5 + 6*7^6 + 6*7^7 + 6*7^8 + 
6*7^9 + 6*7^10 + 6*7^11 + 6*7^12 + 6*7^13 + 6*7^14 + 6*7^15 + 6*7^16 + 6*7^17 + 
6*7^18 + O(7^19)
? recognizep(hgm(9,[1/8,7/8],[3/8,5/8],7))
-4/7
\end{verbatim}
We conclude that the trace of Frobenius on $\calH_9$ is $-4$.  To get
the full characteristic polynomial $L_7(T)$ of Frobenius we can
proceed as follows
\begin{verbatim}
? hgmfrob(9,[1/8,7/8],[3/8,5/8],7)
(7^-1 + O(7^18))*x^2 + (4*7^-1 + O(7^19))*x + 1
? polrecognizep(hgmfrob(9,[1/8,7/8],[3/8,5/8],7))
1/7*x^2 + 4/7*x + 1
\end{verbatim}
This case is special however since the specialization parameter is a
square, which results in the field of definition being $\Q$. The
characteristic polynomial of Frobenius at $p=7$ is then
$$
L_7(T)=7T^2+4T+1.
$$
For a further discussion of this example see Example~\ref{example:reducible}.
\end{example}
\begin{example}
\label{example:1a}  

For $z=3$ on the other hand, we have
\begin{verbatim}
? L1=hgmfrob(3,[1/8,7/8],[3/8,5/8],7);
? L2=hgmfrob(3,[3/8,5/8],[1/8,7/8],7);
? polrecognizep(L1*L2)
1/49*x^4 + 6/49*x^2 + 1
\end{verbatim}
In fact, looking at the polynomials {\tt L1,L2} individually we find
that the characteristic polynomial $L_\calP(T)$ of Frobenius for
primes $\calP$ of $\Q(\sqrt 2)$ dividing $7$ acting on $\calH_3$ is
$$
L_\calP(T)=7T^2\pm 2\sqrt 2T+1.
$$
Let $\chi$ be the quadratic character of
$\Gal(\overline{\Q}/\Q(\sqrt{2}))$ corresponding to the extension
$\Q(\sqrt{\sqrt{2}+2})$ over $\Q(\sqrt{2})$. Computing for more
primes, it is not hard to verify that the twist by $\chi$ of the
Frobenius polynomials of the hypergeometric motive $\calH_3$ match
those of a Hilbert newform over the real quadratic field
$\Q(\sqrt{2})$ of parallel weight $2$ and level $3^2\cdot \sqrt{2}^5$
for all prime ideals of $\Q(\sqrt{2})$ less than $200$. This form
is not in the LMFDB, but we can nevertheless uniquely identify it
using its standardized labeling/ordering of forms. To this end, we
compute the trace of the Hecke eigenvalue of each (Galois orbit)
eigenform for primes up to norm $31$ (sorted as explained in
\cite{2005.09491}). The form corresponds to the $11$-th one in the
resulting ordering. Here is the code in magma to compute the space of
all relevant Hilbert modular forms.
\begin{verbatim}
> R<x> := PolynomialRing(IntegerRing());
> F := NumberField(x^2-2);  OF := Integers(F);
> P2:=Factorisation(2*OF)[1][1];
> M := HilbertCuspForms(F, 9*P2^5);
> decomp := NewformDecomposition(NewSubspace(M));
\end{verbatim}
It seems plausible that the Faltings-Serre's method can be used to
prove modularity of the corresponding motive and hence the equality of
Frobenius polynomials for all $p$.
\end{example}

\begin{remarkintro}
  One important advantage of working $p$-adically is that there is
  typically no need to pass to an algebraic extension to compute the
  corresponding Gauss sums. For the above example, we would a priory
  need to compute in the field $\Q(\zeta_8)$ or possibly
  $\Q(\sqrt{2})$ if set up appropriately. As it stands we simply
  compute the trace of Frobenius in $\Z_p$ and then recognize it as an
  (algebraic) integer.
\end{remarkintro} 

We finish this introduction with a few challenges that we expect to
address in the near future:
\begin{itemize}
\item Can the inertial information be used to prove large image
  results for hypergeometric motives? (We do know (see~\cite{BH}) that
  the Zariski closure of the geometric monodromy is big.)
  
\item  Extend the analogue of property (iii) to primes of the field of
  definition of $\hgm$.

\item Prove properties (i) to (v) for rank $n$ hypergeometric
  motives.
\end{itemize}

\vspace{3pt}

\noindent {\bf Acknowledgments:} We would like to thank David Roberts for many useful conversations.

\section{Rank one hypergeometric motives}
\label{section:rank1}
We start by considering a baby example of the theory, that of rank $1$
motives. The definitions used in the present section will be explained
in detail in subsequent sections for the general case of rank $n$
hypergeometric motives.

Let $\alpha$ be a rational number of denominator $N>1$. Consider the differential
equation (see~(\ref{eq:diff-eq}) for its general definition)
\begin{equation}
  \label{eq:dif-rank1}
  D(\alpha,1):= z\frac{d}{dz} - z (z \frac{d}{dz}+\alpha)=
  z(1-z)\frac{d}{dz}-z\alpha=0. 
\end{equation}
Its one-dimensional space of holomorphic local solutions on a disk
centered at a base point $z_0\neq 0,1$ in the unit disk is spanned by
the algebraic function $H(z) = (1-z)^{-\alpha}$ (for any fixed branch
determination).  Analytic continuation along closed loops based at $z_0$
determines the \emph{monodromy representation}
\[
\rho: \pi_1(\CC\setminus\{1\},z_0) \to \CC^\times.
\]
It is easy to verify that if $\gamma$ is a simple loop around $1$ (in
counterclockwise direction) then
$\rho(\gamma)=\exp\left(\frac{2\pi i}{N}\right)$.
Let $F=\Q(\zeta_N)$. Algebraically, $\rho$ corresponds to the Galois representation
\[
\widetilde{\rho}: \Gal(\overline{\Q(z)}/F(z)) \to \CC^\times,
\]
given by
\[
\widetilde{\rho}(\sigma) = \frac{\sigma(\sqrt[N]{1-z})}{\sqrt[N]{1-z}}.
\]
The representation is ramified only at $1$ and $\infty$. For
$\z-spec \in \QQ$, $\z-spec \neq 1$, we can consider the
specialization of $\widetilde{\rho}$ at $\z-spec$, corresponding to the extension of number
fields $F(\sqrt[N]{1-\z-spec})/F$. 

Let $\Om_F$ denote the ring of integers of $F$, let $\id{q}$ a prime
ideal of $\Om_F$ unramified in the extension $F(\sqrt[N]{1-\z-spec})/F$ and let
$\FF_q$ denote the finite field $\Om_F/\id{q}$ of $q = \norm \id{q}$
elements. Since $N \mid q-1$, $\FF_q$ contains the $N$-th roots of
unity. Fix $\gen$ a generator of the character group of
$\FF_q^\times$ and $\psi$ a non-trivial additive character of $\FF_q$. For
$\z-spec \in \FF_q$, define (see \S \ref{section:finite})
\begin{equation}
  \label{eq:defi-n=1-hyper}
  H_q(\alpha,1|\z-spec):= \frac{1}{1-q}\sum_{\Char} \frac{\JJ(\alpha \Char,\Char)}{\JJ(\alpha,1)} \Char(\z-spec),
\end{equation}
where the sum runs over characters $\Char$ of $\FF_q^\times$,
\[
  \JJ(\alpha\Char,\Char):=g(\psi,\gen^{(q-1)\alpha} \Char) g(\psi^{-1},\overline \Char),
\]
and $g(\psi,\Char)$ denotes the standard Gauss sum (as recalled in (\ref{gsum})).
Denote by $\varepsilon$ the character $\gen^{(q-1)\alpha}$ (a
character of order $N$). Since
$g(\psi^a,\Char)=\overline\Char(a)g(\psi,\Char)$, we have
  \[
\frac{\JJ(\varepsilon\Char,\Char)}{\JJ(\varepsilon,1)}=-\Char(-1)\frac{g(\psi,\varepsilon{\Char})g(\psi,\overline{\Char})}{g(\psi,\varepsilon)} = -\Char(-1)J({\varepsilon\Char},\overline{\Char}),
\]
where for $\Char,\eta$ characters, $J(\Char,\eta)$ denotes the standard
Jacobi sum (whose definition is recalled in
Definition~\ref{defi:Jacobi}).
 Then
\eqref{eq:defi-n=1-hyper} becomes
\[
\frac{1}{q-1}\sum_{\Char} \sum_{t \neq 1} \Char(-\z-spect)\varepsilon(t)\overline{\Char}(1-t) = \frac{1}{q-1}\sum_{t \neq 1} \varepsilon(t) \sum_\Char \Char\left(\frac{-\z-spect}{1-t}\right).
\]
The last sum equals zero unless $-\z-spect=1-t$ and $q-1$ in that case. Therefore
\[
  H_q(\alpha,1|\z-spec)=\varepsilon({1-\z-spec})^{-1},
\]
showing that our finite hypergeometric series $H_q(\alpha,1|\z-spec)$
at a prime ideal $\id{q}$ matches 
the Dirichlet character of the extension
$F(\sqrt[N]{1-\z-spec})/F$.

Consider now the same problem for parameters $1,-\alpha$. The
differential equation becomes
\[
D(1,-\alpha) = z(1-z)\frac{d}{dz}-\alpha =0.
\]
Its space of solutions is spanned by the function
$G(z)= \left(\frac{z}{1-z}\right)^{\alpha}$.
Adjoining it yields the abelian Galois extension
$F\left(\sqrt[N]{\frac{z}{1-z}}\right)/F(z)$, unramified outside
$\{0,1\}$. An elementary computation proves that both the differential
equation and the field extension are related to the previous case by
the change of variables $z \to 1/z$. This is consistent with the
finite hypergeometric series behavior
\[
H_q(1,-\alpha|\z-spec)=\frac{1}{1-q}\sum_{\Char} \frac{J(\Char,\varepsilon^{-1} \Char)}{J(1,\varepsilon^{-1})} \Char(\z-spec)= \frac{1}{1-q}\sum_{\Char} \frac{J(\varepsilon\Char^{-1},\Char^{-1})}{J(\varepsilon,1)} \Char(\z-spec)= H_q(\alpha,1|\z-spec^{-1}).
\]

\section{Hypergeometric Series and Monodromy Representation}
\label{section:hypergeometric-reps}
For $\alpha \in \CC$ and $k$ a positive integer, let $(\alpha)_k$
denote the Pochammer symbol, defined by
\[
(\alpha)_k = \alpha (\alpha+1) \cdots (\alpha + k-1) = \frac{\Gamma(\alpha+k)}{\Gamma(\alpha)}.
\]
\begin{definition}
  Let $\alpha_1,\ldots,\alpha_n,\beta_1,\dots,\beta_{n-1}$ be complex
  numbers.  The hypergeometric series 
  with parameters
  $\veca:=(\alpha_1,\ldots,\alpha_n),\vecb:=(\beta_1,\dots,\beta_{n-1},1)$ is
  defined as the convergent power series 
\begin{equation}
  \label{eq:hypergeometric}
_nF_{n-1}(\veca,\vecb|z) = \sum_{k=0}^\infty \frac{(\alpha_1)_k\cdots(\alpha_n)_k z^k}{(\beta_1)_k \cdots (\beta_{n-1})_k k!},
\end{equation}
for $z \in \CC$ with $|z|<1$.
\end{definition}
We will briefly recall some the properties of this classical function
and refer the reader to~\cite{BH} and the references therein for
details.  For $\Real(\beta_i) > \Real(\alpha_i)>0$, for
$i=1,\ldots n-1$, we have the following integral presentation (see
formula (2.2.2) of \cite{MR1688958})
\begin{equation}
  \label{eq:integral-rep}
_nF_{n-1}(\veca;\vecb|z) =
\prod_{i=1}^{n-1}\frac{\Gamma(\beta_i)}{\Gamma(\alpha_i)\Gamma(\beta_i-\alpha_i)}\int_{0}^1 
\cdots \int_0^1
\frac{\prod_{i=1}^{n-1}(x_i^{\alpha_i-1}(1-x_i)^{\beta_i-\alpha_i-1}dx_i)}{(1-zx_1\cdots
  x_{n-1})^{\alpha_n}}, \quad |z|<1.  
\end{equation}

As it is well-known the hypergeometric series $_nF_{n-1}$ satisfies
the differential equation 
\begin{equation}
  \label{eq:diff-eq}
D(\veca;\vecb)u=0,\qquad
D(\veca;\vecb):= (\theta+\beta_1-1) \cdots (\theta +
\beta_n-1) - z(\theta+\alpha_1) \cdots(\theta+\alpha_n), 
\end{equation}
where $\theta = z\frac{d}{dz}$. This equation has regular
singularities at the points $z=0,1,\infty$.
\begin{remark}
  Note that in the definition of the series~\label{eq:hypergeometric}
  we assume $\beta_n=1$.  Since the order of the $\beta_i$'s (or, for
  that matter, the order of the $\alpha_i$'s) is irrelevant we could
  have assumed any given $\beta_i=1$ (adjusting the range for the
  parameters in~\eqref{eq:integral-rep}). We will now, in any case,
  relax this condition and let all $\beta_i$ be arbitrary, in
  particular when considering the differential equation
  $D(\veca;\vecb)$.
\end{remark}
 We have (\cite[\S 2]{BH}).
\begin{proposition}
  If the number $\beta_1,\ldots,\beta_n$ are distinct modulo $\ZZ$, $n$
  independent solutions of~\eqref{eq:diff-eq} are given by
  \begin{equation}
    \label{eq:ui-def}
 u_i(z):=z^{1-\beta_i} {}_nF_{n-1}(\veca_i,\vecb_i|z),     \qquad i=1,\ldots,n,
  \end{equation}
where
$$
\veca_i:=(1+\alpha_1-\beta_i,\ldots,1+\alpha_n-\beta_i),
\qquad 
\vecb_i:=(1+\beta_1-\beta_i,\ldots,1+\beta_n-\beta_i).
$$
\label{prop:twist}
\end{proposition}

\begin{lemma}
  \label{lemma:change-parameters}
  Let $\alpha_1,\ldots,\alpha_n,\beta_1,\ldots,\beta_n \in \CC$. Then
  the change of variables $z \to u=1/z$ sends solutions of the
  equation $D(\veca;\vecb)$ to solutions of the equation
  $D(\vecb',\veca')$, where
$$
\veca':= (-\alpha_1+1,\ldots,-\alpha_n+1), \qquad
\vecb':= (-\beta_1+1,\ldots,-\beta_n+1).
$$
\end{lemma}

\begin{proof}
  Set $u=1/z$, so $z \frac{d}{dz}= -u\frac{d}{du}$. In the variable
  $u$, equation~\eqref{eq:diff-eq} becomes
  \[
    \frac{(-1)^{n+1}}{u}\left((\theta -\alpha_1)\cdots (\theta
      -\alpha_n) - u(\theta - \beta_1+1)\cdots(\theta - \beta_n
     +1)\right).
    \]
\end{proof}

We are interested in the case when the parameters $\alpha_i, \beta_i$
are rational numbers, where we can expect a geometric origin for the
differential equation. We restrict to this case from now on.
Set~$N$~to be their least common denominator and define the
quantities:   
  \begin{equation}
    \label{eq:parameters}
  A_i:=N(\beta_n-\alpha_i), \quad
  B_i:=N(\alpha_i-\beta_i), \text{ for }i=1,\ldots,n-1, \quad A_n:=N(\alpha_n-\beta_n),\quad B_n:=N\beta_n.
\end{equation}

Then $u_n$ is a period of the middle cohomology of the twisted Euler's variety
\begin{equation}
  \label{eq:Euler-general}
  V_n:\quad y^N = z^{B_n}\prod_{i=1}^{n-1}x_i^{A_i}(1-x_i)^{B_i}(1-zx_1\cdots x_{n-1})^{A_n}.
\end{equation}

\begin{definition}
  Let $\veca = (\alpha_1,\ldots,\alpha_n)$ and
  $\vecb=(\beta_1,\ldots,\beta_n)$ be $n$-tuples of complex
  numbers. We say that $\veca$ and $\vecb$ are \emph{generic} if the
  sets $\{\exp(2\pi i \alpha_j)\},\{\exp(2\pi i \beta_k)\}$ are
  disjoint; equivalently, $\alpha_i\neq \beta_j \bmod \Z$ for
  $i,j=1,\ldots,n$. We will also say that such a pair determines a
  \emph{hypergeometric data}.
\label{defi:generic}
\end{definition}

\subsection{The monodromy representation of the differential equation}
Keep the notation of the previous section. For $X$ a topological
space, and $x_0 \in X$, let $\pi_1(X,x_0)$ denote the fundamental group of
$X$ based at the point $x_0$. Recall the following definition.

\begin{definition} Let $G$ be a group and let
  $\rho_i:G \to \GL_n(\CC)$, $i=1,2$, be two representations of the
  group $G$. The representations are isomorphic if there exists
  $M \in \GL_n(\CC)$ such that $\rho_1(g) = M \rho_2(g) M^{-1}$ for
  all $g \in G$.
\end{definition}

\begin{definition} Let $z_0$ be a complex number different from $0, 1$.  The
  differential equation $D(\veca;\vecb)$ has attached, up to isomorphism, a natural
  \emph{monodromy representation}
\begin{equation}
  \label{eq:monodromy-rep}
  \rho: \pi_1(\PP^1(\CC)\setminus \{0,1,\infty\},z_0) \to \GL_n(\CC)  
\end{equation}
obtained as follows: pick a basis $\{f_1,\ldots,f_n\}$ of $n$
independent solutions to $D(\veca;\vecb)$ around the point
$z_0$. Given a loop
$\gamma \in \pi_1(\PP^1(\CC)\setminus\{0,1,\infty\},z_0)$ extend
analytically each solution $f_i$ along $\gamma$, to obtain a new basis
$\{\tilde{f_1},\ldots,\tilde{f_n}\}$ of linearly independent solutions
at $z_0$. The matrix $\rho(\gamma)$ is defined as the change of basis
matrix.
\end{definition}
\begin{remark}
  \label{remark:rigidity}
 The generic case is \emph{rigid}: the representations $\rho_1$ and
 $\rho_2$ are isomorphic if and only if each monodromy matrix $\rho_1(\gamma_i)$
 is individually conjugate to $\rho_2(\gamma_i)$ (for $i=0,1,\infty$) (see
 \cite[Thm 3.5]{BH}).
\end{remark}

\begin{proposition}
  \label{prop:irreducible}
  The monodromy representation attached to the parameters
  $\veca,\vecb$ is irreducible if and only if the
  parameters are generic.
\end{proposition}
\begin{proof}
  See Propositions 2.7 and 3.3 of \cite{BH}.
\end{proof}

\begin{proposition}
  If the parameters $\veca$, $\vecb$ are generic then the
  monodromy representation for the parameters
  $\veca,\vecb$ is isomorphic to the one corresponding to the parameters
  $\veca+\pmb{n}$, $\vecb+\pmb{m}$ for any 
  $\pmb{n},\pmb{m} \in \ZZ^n$.
\end{proposition}

\begin{proof}
  See Corollary 2.6 of \cite{BH}.  
\end{proof}

\begin{figure}
  \centering
\centering
\FundGroup
\caption{Fundamental group of $\pi_1(\PP^1(\CC)\setminus \{0,1,\infty\},z_0)$}
\label{figure:fundamentalgroup}
\end{figure}

Let $\gamma_0,\gamma_1,\gamma_{\infty}$ denote paths going through $0, 1, \infty$
respectively in counterclockwise direction as in
Figure~\ref{figure:fundamentalgroup}. It follows from Van Kampen's theorem 
that
\[
  \pi_1(\PP^1(\CC)\setminus \{0,1,\infty\},z_0) = \langle
  \gamma_0,\gamma_1,\gamma_{\infty} \; : \; \gamma_0\gamma_1\gamma_{\infty}=1\rangle.
\]
In particular, the monodromy representation is determined by the
matrices
\begin{equation}
  \label{eq:monodromy-matrices}
  M_0:=\rho(\gamma_0),\qquad M_1:=\rho(\gamma_1),\qquad M_\infty:=\rho(\gamma_\infty),
\end{equation}
corresponding to the local monodromy matrices around $0,1$ and $\infty$
respectively.

\begin{definition}
A \emph{pseudo-reflection} is a matrix $M$ satisfying
that $M-1$ has rank one. Its determinant is called the special
eigenvalue of $M$.
\end{definition}

\begin{proposition}
  \label{prop:monodromy-eigenvalues}
  Let $\veca=(\alpha_1,\ldots,\alpha_n)$ and $\vecb=(\beta_1,\ldots,\beta_n)$ be generic. Let $a_j = \exp(2\pi i \alpha_j)$ and $b_j = \exp(2 \pi i \beta_j)$
  for $j=1,\ldots,n$. Define
  \begin{equation}
    \label{eq:monod-rep}
    q_\infty = \prod_{j=1}^n (t-a_j), \qquad q_0 = \prod_{j=1}^n (t-b_j),
  \end{equation}
  Then $M_\infty$ and $M_0$ are conjugate to the companion matrix of $q_\infty$ and $q_0$
  and $M_1$ is a pseudo-reflection with special eigenvalue
  $$
\mu:=\exp\left(2 \pi i \sum_{j=1}^n (\beta_j - \alpha_j)\right)
$$
\end{proposition}

\begin{proof}
  See  Proposition 2.10 and Proposition 3.2 of \cite{BH}.
\end{proof}

Let $\veca, \vecb$ be generic rational parameters and let $N$ be their
least common denominator. Then a construction due to Levelt (see
Theorem 3.5 of \cite{BH}) gives explicit matrices $M_0$, $M_1$,
$M_\infty$ for the monodromy representation, with image in $\GL_N(F)$,
where $F=\Q(\zeta_N)$.

\subsection{Geometric representations} A good reference for the
statements of the present section is \S 6.3 of \cite{MR2363329}. Let
$\K$ be an algebraic closed field, and $X$ a projective smooth curve
of genus $g$ over $\K$ (in our case $X = \PP^1$, so $g=0$). Let
$P_1,\ldots,P_k$ be distinct points in $X(\K)$. Let $\overline{\K(X)}$
be an algebraic closure of the function field $\K(X)$ of $X$ and let
$\Omega \subset \overline{\K(X)}$ be the maximal extension of $\K(X)$
unramified outside the points $P_1,\ldots,P_k$. The algebraic
fundamental group is defined by
\begin{equation}
  \label{eq:alg-fundamentalgroup}
  \pi_1^{\text{alg}}(X\setminus \{P_1,\ldots,P_k\}) := \Gal(\Omega/\K(X)).
\end{equation}

\begin{definition}
  Let $G$ be a discrete group. The \emph{profinite completion} of $G$
  is the topological group
  \[
   \widehat{G}:=\lim_{\longleftarrow} G/H,
 \]
 where the inverse limit is with respect to finite index normal
 subgroups $H$ of $G$.
\end{definition}
Given a discrete group $G$, we denote by $\widehat{G}$ its profinite
completion.

\begin{definition}
  Let $g,k$ be positive integers. The group $\pi_1(g,k)$ is defined as
  the group with generators and relations
  \begin{equation}
    \label{eq:profinite}
    \pi_1(g,k) = \langle a_1,b_1,\ldots,a_g,b_g,c_1,\ldots,c_k \; : \; a_1b_1a_1^{-1}b_1^{-1} \cdots a_g b_ga_g^{-1}b_g^{-1}c_1\cdots c_k = 1\rangle.
  \end{equation}
\end{definition}

\begin{theorem}
  The algebraic fundamental group
  $\pi_1^{\text{alg}}(X\setminus \{P_1,\ldots,P_k\})$ is isomorphic to
  $\widehat{\pi_1(g,k)}$, the profinite completion of the group
  $\pi_1(g,k)$.
\label{thm:p1-comparison}
\end{theorem}

\begin{proof}
  See \cite[Theorem 6.3.1]{MR2363329}. This is a particular case of
  the well known general result: the algebraic fundamental group of a
  complex variety is isomorphic to the profinite completion of the
  topological fundamental group.
\end{proof}
\begin{remark}
  As noted in \cite{MR2363329}, the canonical map
  $\pi_1(g,k) \to \widehat{\pi_1(g,k)}$ is in fact injective.
\end{remark}

Let $F$ be a number field, let $\Om$ be its ring of integers and let
$\id{p}$ be a prime ideal of $\Om$. We denote by $\Om_{\id{p}}$ the
completion of $\Om$ at $\id{p}$ (and by $F_{\id{p}}$ its field of
fractions).

Let $\veca$ and $\vecb$ be generic rational parameters. By Levelt's
theorem the hypergeometric monodromy representation can be chosen to
be defined over $\Om$. Fix such a choice (unique up to isomorphism
over $F$)
\[
\rho: \quad \pi_1(\PP^1\setminus \{0,1,\infty\},z_0) \rightarrow  \GL_n(\Om).
\]
\begin{corollary}
  Let $\veca$ and $\vecb$ be generic rational
  parameters. The representation
  \[
    \xymatrix{
\pi_1(\PP^1\setminus \{0,1,\infty\},z_0) \ar[r]^(0.7){\rho}&  \GL_n(\Om) \ar[r]^{\iota} &
      \prod_{\id{p}} \GL_n(\Om_{\id{p}}), 
      }
    \]
 Extends to a  continuous representation (unique up to isomorphism)
\begin{equation}
  \label{eq:p-adic-rep}
  \widehat{\rho}: \Gal(\overline{\K(z)}/\K(z)) \to \prod_{\id{p}}\GL_n(\Om_{\id{p}}),
\end{equation}
unramified outside $\{0,1,\infty\}$.
\label{coro:extension}
\end{corollary}

\begin{proof}
  By Van Kampen's theorem
  $\pi_1(\PP^1\setminus \{0,1,\infty\},z_0)\simeq \pi_1(0,3)$, so we
  can think of $\rho$ as a representation
  $\rho:\pi_1(0,3) \to \GL_n(\Om)$. Consider the following natural morphisms
  \[
    \xymatrix{
      \pi_1(0,3) \ar[r]^{\rho}&  \GL_n(\Om) \ar[r]^{\iota} &
      \prod_{\id{p}} \GL_n(\Om_{\id{p}}).
      }
    \]
    Since $\GL_n(\Om_{\id{p}})$ is profinite, by the universal
    property of the profinite completion (see \cite[Lemma
    3.2.1]{MR2599132}) we get a morphism
    \[
      \widehat{\rho}: \widehat{\pi_1(0,3)} \to \prod_{\id{p}}\GL_n(\Om_{\id{p}}).
    \]
    Theorem~\ref{thm:p1-comparison} (with $X = \PP^1$ and
    $\{P_1,P_2,P_3\} = \{0,1,\infty\}$) gives an isomorphism
    $\widehat{\pi_1(0,3)} \simeq \Gal(\Omega/\K(z))$, the later being
    a quotient of $\Gal(\overline{\K(z)}/\K(z))$ (corresponding to
    extensions unramified outside $\{0,1,\infty\}$). This proves the
    existence statement. Uniqueness follows from the fact that any
    discrete group $G$ is dense in its profinite completion
    $\widehat{G}$.
\end{proof}
Let $\rho_{\id{p}}$ be the map obtained by composing $\widehat{\rho}$
with the projection to $\GL_n(\Om_{\id{p}})$.  It follows from
Corollary~\ref{coro:extension} that for rational generic parameters,
the differential equation satisfied by the hypergeometric series
induces a compatible family of Galois representations
\begin{equation}
  \label{eq:geometric-rep}
\rho_{\id{p}}:\Gal(\overline{\Q(z)}/\overline{\Q}(z)) \to \GL_N(F_{\id{p}}).  
\end{equation}

\section{Generalities on hypergeometric motives}
\label{section:HGM}
We refer the reader to the survey~\cite{MR4442789} for an introduction
to the subject (see also Katz's article~\cite{MR1366651}).  Let
$\veca,\vecb$ be vectors of generic rational numbers (as in
Definition~\ref{defi:generic}), and let $N$ be their least common
denominator and let $r$ denote the number integral coordinates of $\veca, \vecb$. Keeping the previous notation, let $F =
\Q(\zeta_N)$. Then for $\z-spec \in F$, $\z-spec \neq 0,1$, there should exist
a motive $\HGM(\veca,\vecb|\z-spec)$, a \emph{hypergeometric motive or
  HGM}, attached to the parameters $\veca, \vecb$ satisfying the
following properties

\begin{enumerate}[(i)]
  
\item It is a pure motive of degree $n$ and weight $r-1$ defined over $F$.
\item Its $\ell$-\'adic \'etale realization is
related to the monodromy representation.
\item
  The primes of $F$ of bad reduction belong to
  $S_{\ptame}\cup S_{\pwild}$ ({\it potentially tame} and {\it
    potentially wild} primes respectively), where
$$
S_{\pwild}:=\{\id{p} \,|\,  v_{\id{p}}(N)>0 \},
\qquad
S_{\ptame}:= (S_0\cup S_1\cup S_\infty)\setminus S_{\pwild},
$$
and 
$$
S_0:=\{\id{p} \,|\,  v_{\id{p}}(\z-spec)> 0\},
\qquad S_1:=\{\id{p} \,|\, v_{\id{p}}(\z-spec-1)> 0\},
\qquad
S_\infty:=\{\id{p} \,|\, v_{\id{p}}(\z-spec) < 0\}.
$$
We let $S_g$ be the set of primes not in $S_{\ptame}\cup S_{\pwild}$.
\label{prop:iii}
\item For a prime $\id{p}\in S_g$   the
  trace of powers of the Frobenius automorphism at $\id{p}$ acting on
  $\HGM(\veca,\vecb|\z-spec)$ are given by an explicit finite
  hypergeometric sum.
\end{enumerate}
\begin{remark}
  As is well-known (iv) yields a way to compute the characteristic polynomial
  of the Frobenius automorphism at $\id{p}$ acting on
  $\HGM(\veca,\vecb|\z-spec)$ by means of Newton's formulas.
\end{remark}
We present a more precise statement in
Conjecture~\ref{conjecture:field-of-defi}. We use the term motive to
refer to a pure motive (as in \S 4 of \cite{MR2115000}), i.e. a pair
$(X,p)$ where $X$ is a projective variety and $p$ is a correspondence
on $X$ which is an idempotent. We say that the motive has a
\emph{realization} over $K$ if both the variety $X$ and the
correspondence $p$ are defined over $K$.
\begin{remark}
  The existence of a motive over a number field for general rigid
  local systems was proved by Katz in \cite{MR1366651}. See \S 6 of
  \cite{2003.05031} for a short exposition of the motivic
  interpretation.

  Among the extensive literature on topics related to HGMs we may
  highlight the following as most relevant for our purposes: for
  \emph{rational} HGMs (see Definition~\ref{defi:rational})~\cite{BCM}
  and ~\cite{MR4442789} and for the specific case of rank $n=2$ with
  $\beta_2=1$~\cite{MR2005278} and~\cite{MR4493579}.
\end{remark}
Let $H$ be the subgroup of $\Gal(F/\Q)$ that fixes the set
$\exp(2 \pi i \veca)$ and $\exp(2 \pi i \vecb)$.
Via the natural identification of $\Gal(F/\Q)$ with $(\Z/N)^\times$,
the group $H$ can be defined as the set of elements
$j \in (\Z/N)^\times$ satisfying
\begin{equation}
  \label{eq:H-defi}
 j \cdot \{\alpha_1,\ldots,\alpha_n\} \equiv
 \{\alpha_1,\dots,\alpha_n\}\bmod \Z,\quad \text{and} \quad 
 j \cdot \{\beta_1,\ldots,\beta_n\} \equiv
 \{\beta_1,\dots,\beta_n\}\bmod \Z.
\end{equation}

\begin{definition}
\label{defn:base-field}
We define the \emph{base field} $K:=\Q(\zeta_N)^H \subseteq \Q(\zeta_N)$ of the
hypergeometric data $(\veca,\vecb)$ as the field fixed  by $H$.
\end{definition}

\begin{definition}
\label{defn:coeff-field}
  Given a motive we will call the field generated by the traces of all
  Frobenius elements the \emph{coefficient field} of the motive.
\end{definition}

We have (see Proposition~\ref{prop:props-hyperg-sum})
\begin{equation}
  \label{eq:conjugation}
 \HGM(\veca,\vecb|\z-spec)^\sigma =
  \HGM(j\veca,j\vecb|\z-spec^\sigma),
\end{equation}
where $1\leq j \leq N-1$ coprime to $N$ corresponds to $\sigma$.
Hence it is natural to expect (and all numerical evidence supports)
the following
\begin{conjecture}
  Let $K$ be the base field of the hypergeometric data
  $\veca,\vecb$. Then
\begin{enumerate}
\item
The  motive $\HGM(\veca,\vecb|z)$ has a realization $\calH(z)$ 
over $K(z)$.
\item (Compatibility) The representation of the motive restricted to
  $\Gal(\overline{K(z)}/\overline{K}(z))$ is isomorphic to the
  geometric Galois representation~\eqref{eq:geometric-rep}.

\item Let $L$ be a finite extension of $K$.  The specialization
  $\calH(\z-spec)$ for generic $\z-spec\in L$ has coefficient field
  $K$.
\item The primes $\id{p}\in S_g$ of $L$ are of good reduction and the
  trace of Frobenius on $\calH(\z-spec)$ is given by the finite
  hypergeometric sum $H_{\id{p}}(\veca,\vecb\,|\,\z-spec)$.
\end{enumerate}
\label{conjecture:field-of-defi}
\end{conjecture}

\begin{remark}
  For particular values of the parameter $z$, we might by able to
  descend to a motive over a proper subfield $L\subseteq K$. For
  example the motive of Example~\ref{example:genericity-condition} has
  base field $\Q(\sqrt{2})$, but evaluated at squares is
  defined over $\Q$ and has coefficient field $\Q$.
\end{remark}

\begin{definition}
  Let $\rho \in \Q$. The \emph{hypergeometric twist by $\rho$} of the
  hypergeometric motive $(\veca,\vecb)$ is the motive with
  hypergeometric data $(\rho\veca,\rho\vecb)$, where
  $\rho(\alpha_1,\ldots,\alpha_n)=(\rho+\alpha_1,\ldots,\rho+\alpha_n)$.
\end{definition}

It follows from Proposition~\ref{prop:monodromy-eigenvalues} that the
hypergeometric twist by $\rho$ multiplies the local monodromies at $0$
and $\infty$ of $\HGM(\veca,\vecb|z)$ by the scalars $\lambda$ and
$\lambda^{-1}$ respectively, where $\lambda = \exp(-2\pi i \rho)$. In
particular, hypergeometric twisting does not change the projective
monodromy representation. As a consequence, for example, the
corresponding hypergeometric representations are simultaneously finite
or infinite.

\begin{example}
\label{example:base-field-switch}
A generic hypergeometric twist of a HGM typically increases
the degree of its base field. But in some special cases it does not.
For example, the following parameter set
$$
\veca=(1/24, 11/24, 17/24, 19/24), \qquad 
\vecb=(1/4, 1/2, 3/4, 1)
$$
has base field $\Q(\sqrt{-8})$. However, its twist by $-1/8$
$$
\veca'=(1/3, 7/12, 2/3, 11/12)  \qquad \vecb'=(1/8, 3/8, 5/8, 7/8)
$$
has base field $\Q(i)$. The HGM of parameters $(\veca',\vecb')$ is
number $38$ in the Beukers-Heckman~\cite{BH} list of algebraic
hypergeometric functions; it has weight zero and corresponds to an
Artin motive. The geometric monodromy group for $\HGM(\veca,\vecb|z)$
is of order $2304$ and is a central extension by $C_2$ of
$S_4 \wr S_2$. However, the motive of parameters $(\veca,\vecb)$
itself does not actually appear in the Beukers-Heckman list since they
mod out by twisting. There are only three cases of this kind in the
Beukers-Heckman list for rank $n>2$, all with base field
$\Q(\sqrt{-8})$. Namely,
$$
\begin{array}{cc}
(1/24, 11/24, 17/24, 19/24) & (1/4, 1/2, 3/4, 1)\\
(1/24, 11/24, 17/24, 19/24) & (1/4, 5/8, 3/4, 7/8)\\
(1/8, 1/4, 3/8, 3/4) & (1/6, 1/3, 2/3, 5/6)
\end{array}
$$
\end{example}

\begin{definition}
  A pair $(\veca,\vecb)$ of vectors is \emph{rational} if its base
  field is $\Q$ (equivalently, $H=\Gal(F/\QQ)$).
\label{defi:rational}
\end{definition}

\vspace{2pt}

Suppose for the rest of the section that the
Conjecture~\ref{conjecture:field-of-defi} holds for the motive
$\HGM(\veca,\vecb|z)$.
Then there exists a number field $L$ with ring of integers $\Om_L$ and
a compatible family of Galois representations
\[
\rho_{\id{p}}:\Gal(\overline{K(z)}/K(z)) \to \GL_n(L_{\id{p}}),
\]
for primes $\id{p}$ of $L$ extending the geometric
representation~\eqref{eq:geometric-rep}. The potentially wild primes
should be the ones dividing the denominator of $\veca$ or $\vecb$. We
can prove a weaker version.  Define
\begin{equation}
  \label{eq:constant}
  S(n,L) := \gcd\{\# \GL_n(\Om_L/\id{p}) : \id{p} \text{ is a prime ideal of }\Om_L\}.
\end{equation}

The following result follows the lines of \cite[Lemma 1.2]{Darmon}, and its proof resembles that of \cite[Corollary 2]{MR236190}.

\begin{theorem}
  \label{thm:unramified-primes}
  Keep the previous notation and assumptions and assume
  Conjecture~\ref{conjecture:field-of-defi} holds. Let
  $\z-spec \in \PP^1(K) \setminus \{0,1,\infty\}$ and let
  $\HGM(\veca,\vecb|{\z-spec})$ be the specialization at
  $\z-spec$. Let $p, q, r$ denote the order of the monodromy
  representation at the points $0,1,\infty$ respectively (possibly
  equal to $\infty$). Let $\id{n}$ be a prime ideal in $\Om_K$ not
  dividing~$S(n,L)$ and satisfying one of the following properties:
  \begin{itemize}
  \item $\id{n} \mid \z-spec$, $p$ is finite  and $p \mid v_{\id{n}}(\z-spec)$,
    
  \item $\id{n} \mid \z-spec-1$, $q$ is finite and
    $q \mid v_{\id{n}}(\z-spec-1)$,
    
  \item $v_{\id{n}}(\z-spec)<0$, $r$ is finite and $r \mid v_{\id{n}}(\z-spec)$.
    
  \item $v_{\id{n}}(\z-spec) = v_{\id{n}}(\z-spec-1)=0$.
  \end{itemize}
  Then the compatible family of Galois representations attached to
  $\HGM(\veca,\vecb|{\z-spec})$ is unramified at $\id{n}$.
\end{theorem}

\begin{proof} Let $\id{n}$ be a prime ideal of $K$ not dividing $S(n,L)$ and
  satisfying one of the required hypothesis. Then there exists a prime ideal
  $\id{p}$ of $L$ such that $\id{n} \nmid \#\GL_n(\Om_L/\id{p})$ (in
  particular, $\id{n}$ does not divide the norm of $\id{p}$). Let
  \[
\rho_{\id{p}}:\Gal(\overline K(z)/K(z)) \to \GL_n(L_{\id{p}})
\]
be the $\id{p}$-adic representation attached to the motive
$\HGM(\veca,\vecb|z)$. Consider first the \emph{geometric} part of the
representation, namely its restriction to
$G^{\text{geom}}:=\Gal(\overline{\Q(z)}/\overline{\Q}(z))$.

Part (2) of Conjecture~\ref{conjecture:field-of-defi} implies that the image of inertia
at $z=0$, $z=1$ and $z=\infty$ under (the restriction to
$G^{\text{geom}}$ of) $\rho_{\id{p}}$ has order $p$, $q$ and $r$
respectively.  Since the extension $\overline{\Q}(z)/K(z)$ is
unramified at $z=0$, $z=1$ and $z=\infty$, the image of the inertia
subgroup at these three points under $\rho_{\id{p}}$ also has order
$p, q$ and $r$ respectively.

After choosing a lattice fixed by $\rho_{\id{p}}$, we can assume that
our representation actually takes values in $\GL_n(\Om_{\id{p}})$ (the
completion of $\Om_K$ at $\id{p}$). Let $m$ be a positive integer,
and let $\overline{\rho_{\id{p},m}}$ be the reduction of
$\rho_{\id{p}}$ modulo $\id{p}^m$. The kernel of
$\overline{\rho_{\id{p},m}}$ corresponds to a curve $X_{\id{p},m}$
which is a finite cover of $\PP^1$ unramified outside
$\{0,1,\infty\}$. Then Theorem 1.2 of \cite{MR1116916} implies that if
one of our hypothesis is satisfied and if $\id{n}$ does not divide the
order of the image of $\overline{\rho_{\id{p},m}}$, then the image of
$I_{\id{n}}$ under the specialization map is trivial. But the order of
the image of $\overline{\rho_{\id{p},m}}$ divides the order of
$\GL_n(\Om_L/\id{p}^m)$, whose order is supported at the same primes
as $\GL_n(\Om_L/\id{p})$. Then, if
$\id{n} \nmid \#\GL_n(\Om_L/\id{p})$,
$\overline{\rho_{\id{p},m}}(I_{\id{n}})=1$ for all $m$, hence
$\rho_{\id{p}}(I_{\id{n}})=1$, and the family is unramified at
$\id{n}$.
\end{proof}

Suppose that the motive $\hgmvec$ has rank $n$ and is
rational. Then any potentially wild prime $p$ satisfies $p-1 \le
n$. The reason is that if a prime $p$ divides the denominator of
$\veca$ (respectively the denominator of $\vecb$), all values
$\frac{i}{p}$ for $i=1,\ldots,p-1$ are also coordinates of $\veca$
(respectively $\vecb$).
\begin{lemma}
  If $L = K=\Q$ and $n>1$, then $p \mid S(n,\Q)$ if and only if
  $p-1 \le n$.
\end{lemma}
\begin{proof}
  Let $p$ be a prime number such that $p-1 = \phi(p) \le n$. Then
  there is an injective map $\psi:\Z[\zeta_p] \to \GL_n(\Z)$; for
  example, let $M$ be the $n \times n$ matrix made up of two blocks on the
  diagonal (and zero elsewhere). The first block (of size
  $(p-1)\times (p-1)$) being the companion matrix of $(x^p-1)/(x-1)$
  and the second one being the identity. Then
  $\Z[\zeta_p] \simeq \Z[M] \subset \GL_n(\Z)$.

  Since $n \ge 2$, it is always the case that $p \mid
  |\GL_n(\Z/p)|$. Let $q$ be a rational prime number different from
  $p$. Since $q$ does not ramify in $\Z[\zeta_p]/\Z$, the group
  $\GL_n(\Z/q)$ contains an element of order $p$ (the image under the
  reduction map of $\psi(\zeta_p)$), so $p \mid \# \GL_n(\Z/q)$ hence
  $p \mid S(n,\Q)$.

  Reciprocally, let $1 \le x \le p-1$ be any integer and let $r$ be
  its order in $\FF_p^\times$. By Dirichlet's theorem on arithmetic
  progressions, there exists a rational prime $q$ congruent to $x$
  modulo $p$. By hypothesis,
  $p \mid \#\GL_n(\Z/q) = q^\ast (q-1)(q^2-1)\cdots (q^n-1)$, so
  $x^i \equiv q^i \equiv 1 \pmod p$ for some $i \le n$, i.e. any
  element modulo $p$ has multiplicative order at most $n$, hence
  $p-1 \le n$.
\end{proof}
\begin{corollary}
  If $\hgmvec$ is a rational motive of rank $n$ and
  $p\in S_{\pwild}$ then $p\mid S(n,\Q)$.  Conversely, if
  $p \mid S(n,\Q)$ then there exist a rational $\hgmvec$
  with $p\in S_{\pwild}$.
\end{corollary}

\begin{proof}
  The first statement follows from the lemma, since if
  $p \in S_{\pwild}$, $p-1 \le n$. To prove the second statement,
  suppose that $p$ is a prime number with $p-1 \le n$. Define
  $\veca=(\frac{1}{p},\ldots,\frac{p-1}{p},\frac{1}{2},\ldots,\frac{1}{2})$
  and $\vecb=(1,\ldots,1)$. Then $\veca$, $\vecb$ are generic, the
  motive $\hgmvec$ is rational and
  $p \in S_{\pwild}$.
\end{proof}

\begin{remark}
  As described in (\ref{prop:iii}), the set of potentially wild primes
  consists of those primes of $K$ dividing the least common
  denominator of $\veca$ and
  $\vecb$. Theorem~\ref{thm:unramified-primes} should hold for all
  prime ideals that are not potentially wild, but we do not know how
  to prove this stronger statement when the rank is larger than $2$.
  The rank $2$ case is proved in Appendix~\ref{appendix:Elisa} (see
  Theorems~\ref{thm:HGM-0} and \ref{thm:HGM-1-infty}).
\end{remark}

\subsection{Hodge numbers and normalization}
\label{section:hodge}
The Hodge numbers of a rational HGM can be computed using a formula
conjectured by Corti and Golyshev (proved in general in
\cite{MR2824960,MR3862114,1907.02722}).  The formula can be applied to
arbitrary hypergeometric parameters by means of the {\it zig-zag
  diagram} (see \S 5 of \cite{MR4442789}), giving the dimensions of
the associated grading of the Hodge filtration (see \cite[Figure
5.1]{MR4442789} for a nice rank 5 example). Let us recall the
procedure. Assume $\veca,\vecb$ are generic.

\begin{alg}{\bf The zig-zag procedure}
  
\begin{algorithmic}[1]
  \STATE Set $S$ the sequence of ordered parameters of
  $\veca$ as elements in $(0,1]$ and $\vecb$ as elements in $[0,1)$.

  \STATE To each $s \in S$ associate the color red if $s \in \veca$
  and blue if $s \in \vecb$.

  \STATE $P \leftarrow 0$
  
  \FOR {$i = 1 \ldots 2n$}

  \STATE Draw a point with the color $S[i]$ at $(i,P)$.

  \IF {$S[i]$ is blue}

  \STATE $P \leftarrow P-1$

  \ELSE

  \STATE $P \leftarrow P+1$

  \ENDIF
  \ENDFOR
\end{algorithmic}
\label{algo:zigzag}
\end{alg}
Let $r$ be the number of parameters of $\veca, \vecb$ in $\Z$. Then
the Hodge polynomial of the motive $\hgmvec$ equals
  \[
\sum_{P}x^{-p_2}y^{p_2+r-1},
\]
where the sum runs over the blue points $P=(p_1,p_2)$ of the zig-zag
procedure. In \S\ref{section:rank2-hodge} we will prove the relation
between the Hodge numbers of Euler's curve and the output of the
zig-zag procedure for rank two hypergeometric motives.

\begin{remark}
  In fact, this procedure can be extended to give the full mixed
  Hodge numbers of the middle cohomology of the corresponding toric
  model in the rational case, see~\cite{1907.02722}.
\end{remark}

\begin{example}
\label{ex:ell-curve}
As a first example, consider the rational HGM with
parameters $\veca=(1/2,1/2)$, $\vecb=(0,0)$.
The zig-zag procedure gives the following picture

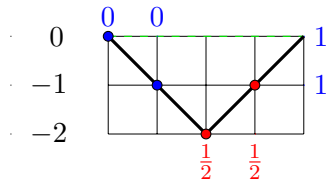
\begin{figure}[H]
\begin{center}
\begin{tikzpicture}[scale=0.65]
\draw[step=1.0,black, very thin] (0,0) grid (4,2);
\draw[black] (-2,0)--(-2,0) node[right]{$\  -2$};
\draw[black] (-2,1)--(-2,1) node[right]{$\  -1$};
\draw[black] (-2,2)--(-2,2) node[right]{$\ \ \,\,0$};

\draw[green, dashed] (0,2) -- (4,2);

\draw[blue] (0,2)--(0,2) node[above]{$0$};
\draw[blue] (1,2)--(1,2) node[above]{$0$};
\draw[red] (3,0)--(3,0) node[below]{$\frac{1}{2}$};
\draw[red] (2,0)--(2,0) node[below]{$\frac{1}{2}$};
\draw[blue] (4,1)--(4,1) node[right]{$1$};
\draw[blue] (4,2)--(4,2) node[right]{$1$};
\draw[very thick] (0,2) -- (1,1);
\draw[very thick] (1,1) -- (2,0);
\draw[very thick] (2,0) -- (3,1);
\draw[very thick] (3,1) -- (4,2);
\draw[fill=red] (2,0) circle (.1);
\draw[fill=red] (3,1) circle (.1);
\draw[fill=blue] (0,2) circle (.1);
\draw[fill=blue] (1,1) circle (.1);
\end{tikzpicture}
\end{center}
\end{figure}

There are two integral entries in $\veca, \vecb$ so $r=2$. The Hodge
polynomial then equals $y + x$ (as expected).  We will see below (see
Remark~\ref{rem:Legendre}) that $\HGM((1/2,1/2),(0,0)|z)$ is (a
quadratic twist by $\kro{-4}{\cdot}$ of) the first cohomology of
Legendre's family of elliptic curves
\[
  E_z: y^2 = x(1-x)(1-zx).
\]
\end{example}

\noindent {\bf Example \ref{example:1}.}(continued) Consider the hypergeometric
motive $\calH_z$ of the introduction, with parameters
$\veca=(1/8,7/8), \vecb=(3/8,5/8)$ defined over its
base field, the real quadratic field $K=\Q(\sqrt{2})\subset \Q(\zeta_8)$.
For $j \in (\Z/8)^\times$ the Hodge numbers of
$\calH_z^{\sigma_j}$ are obtained by applying the zig-zag procedure to
the parameters $\veca=(j/8,7j/8), \vecb=(3j/8,5j/8)$. The result is given in
Figure~\ref{zigzagfig-ex1}.
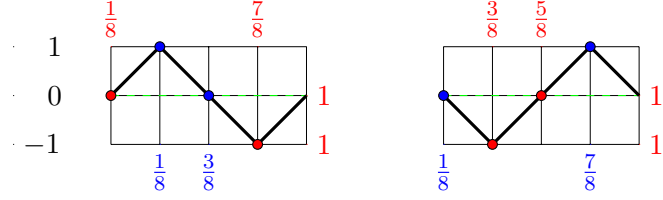
\begin{figure}
\begin{subfigure}{.3\textwidth}
\begin{center}
\begin{tikzpicture}[scale=0.65]
\draw[black] (-2,0)--(-2,0) node[right]{$-1$};
\draw[black] (-2,1)--(-2,1) node[right]{$\ \ \,0$};
\draw[black] (-2,2)--(-2,2) node[right]{$\ \ \,1$};
\draw[step=1.0,black, very thin] (0,0) grid (4,2);

\draw[green, dashed] (0,1) -- (4,1);

\draw[red] (0,2)--(0,2) node[above]{$\frac{1}{8}$};
\draw[blue] (1,0)--(1,0) node[below]{$\frac{1}{8}$};
\draw[blue] (2,0)--(2,0) node[below]{$\frac{3}{8}$};
\draw[red] (3,2)--(3,2) node[above]{$\frac{7}{8}$};
\draw[red] (4,1)--(4,1) node[right]{$1$};
\draw[red] (4,0)--(4,0) node[right]{$1$};
\draw[very thick] (0,1) -- (1,2);
\draw[very thick] (1,2) -- (2,1);
\draw[very thick] (2,1) -- (3,0);
\draw[very thick] (3,0) -- (4,1);
\draw[fill=red] (0,1) circle (.1);
\draw[fill=blue] (1,2) circle (.1);
\draw[fill=red] (3,0) circle (.1);
\draw[fill=blue] (2,1) circle (.1);
\end{tikzpicture}
\end{center}
\end{subfigure}\hspace{0em}
\begin{subfigure}{.3\textwidth}
\begin{center}
\begin{tikzpicture}[scale=0.65]
\draw[step=1.0,black, very thin] (0,0) grid (4,2);

\draw[green, dashed] (0,1) -- (4,1);

\draw[blue] (0,0)--(0,0) node[below]{$\frac{1}{8}$};
\draw[red] (1,2)--(1,2) node[above]{$\frac{3}{8}$};
\draw[blue] (3,0)--(3,0) node[below]{$\frac{7}{8}$};
\draw[red] (2,2)--(2,2) node[above]{$\frac{5}{8}$};
\draw[red] (4,0)--(4,0) node[right]{$1$};
\draw[red] (4,1)--(4,1) node[right]{$1$};
\draw[very thick] (0,1) -- (1,0);
\draw[very thick] (1,0) -- (2,1);
\draw[very thick] (2,1) -- (3,2);
\draw[very thick] (3,2) -- (4,1);
\draw[fill=red] (1,0) circle (.1);
\draw[fill=red] (2,1) circle (.1);
\draw[fill=blue] (0,1) circle (.1);
\draw[fill=blue] (3,2) circle (.1);
\end{tikzpicture}
\end{center}
\end{subfigure}
\caption{\label{zigzagfig-ex1} The Hodge numbers of $(\frac{1}{8},\frac{7}{8}),(\frac{3}{8},\frac{5}{8})$ and $(\frac{3}{8},\frac{5}{8}),(\frac{1}{8},\frac{7}{8})$}
\end{figure}
The Hodge polynomial in both cases is $x^{-1} + y^{-1}$.

\begin{example}
\label{example:non-ratnl}
As a further example of a motive not defined over $\Q$, let
$\veca=(\frac{1}{2},\frac{1}{2})$,
$\vecb=(0,\frac{1}{4})$. Here $N=4$, $r=1$ and $K=\Q(i)$. We find the
Hodge polynomials
$1+xy^{-1}$ and $1 + x^{-1}y$
for the
respective complex embeddings $K$, as illustrated in
Figure~\ref{zigzagfig-3}. The Hodge polynomial of the restriction of
scalars $\calH(z)$ to $\Q$ (the sum of
$\HGM(\veca,\vecb|z)^\sigma$ over the two complex
embeddings $\sigma$ of $K$) equals $x^{-1}y + 2 + xy^{-1}$.
\begin{figure}[h]
\begin{subfigure}{.3\textwidth}
\begin{center}
\begin{tikzpicture}[scale=0.65]
\draw[black] (-2,0)--(-2,0) node[right]{$-2$};
\draw[black] (-2,1)--(-2,1) node[right]{$-1$};
\draw[black] (-2,2)--(-2,2) node[right]{$\ \ \,0$};
\draw[black] (-2,3)--(-2,3) node[right]{$\ \ \,1$};
\draw[step=1.0,black, very thin] (0,0) grid (4,3);

\draw[green, dashed] (0,2) -- (4,2);

\draw[red] (3,3)--(3,3) node[above]{$\frac{1}{2}$};
\draw[red] (2,3)--(2,3) node[above]{$\frac{1}{2}$};
\draw[blue] (1,0)--(1,0) node[below]{$\frac{1}{4}$};
\draw[blue] (0,0)--(0,0) node[below]{$0$};
\draw[blue] (4,1)--(4,1) node[right]{$1$};
\draw[blue] (4,2)--(4,2) node[right]{$1$};
\draw[very thick] (0,2) -- (1,1);
\draw[very thick] (1,1) -- (2,0);
\draw[very thick] (2,0) -- (3,1);
\draw[very thick] (3,1) -- (4,2);
\draw[fill=red] (2,0) circle (.1);
\draw[fill=red] (3,1) circle (.1);
\draw[fill=blue] (0,2) circle (.1);
\draw[fill=blue] (1,1) circle (.1);
\end{tikzpicture}
\end{center}
\end{subfigure}\hspace{0em}
\begin{subfigure}{.3\textwidth}
\begin{center}
\begin{tikzpicture}[scale=0.65]
\draw[step=1.0,black, very thin] (0,0) grid (4,3);

\draw[green, dashed] (0,2) -- (4,2);

\draw[red] (2,3)--(2,3) node[above]{$\frac{1}{2}$};
\draw[red] (1,3)--(1,3) node[above]{$\frac{1}{2}$};
\draw[blue] (0,0)--(0,0) node[below]{$0$};
\draw[blue] (3,0)--(3,0) node[below]{$\frac{3}{4}$};
\draw[blue] (4,2)--(4,2) node[right]{$1$};
\draw[blue] (4,3)--(4,3) node[right]{$1$};
\draw[very thick] (0,2) -- (1,1);
\draw[very thick] (1,1) -- (2,2);
\draw[very thick] (2,2) -- (3,3);
\draw[very thick] (3,3) -- (4,2);
\draw[fill=red] (1,1) circle (.1);
\draw[fill=red] (2,2) circle (.1);
\draw[fill=blue] (0,2) circle (.1);
\draw[fill=blue] (3,3) circle (.1);
\end{tikzpicture}
\end{center}
\end{subfigure}
\caption{\label{zigzagfig-3} The Hodge numbers of $(\frac{1}{2},\frac{1}{2}),(0,\frac{1}{4})$ and $(\frac{1}{2},\frac{1}{2}),(0,\frac{3}{4})$}
\end{figure}
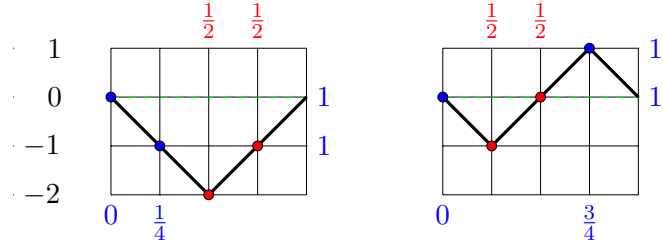

We can check this calculation numerically in Pari/GP using the third's
author package. Taking for example $z=3$ and $p=17$, we get:
\begin{verbatim}
? polrecognizep(hgmfrob(3,[1/2,1/2],[3/4,1],17)*hgmfrob(3,[1/2,1/2],[1/4,1],17))
x^4 - 10/17*x^3 + 18/17*x^2 - 10/17*x + 1
\end{verbatim}
This is the Euler factor $L_{17}(T)$ of $\calH(3)$. By a well-known
fact, its Newton polygon at $p=17$ should lie above the Hodge
polynomial. In fact they are actually equal in this case.  Concretely,
$L_{17}$ has roots of absolute value $\{-1,0,0,1\}$ with
multiplicities matching the Hodge vector $(1,2,1)$ of $\calH(z)$.
\end{example}

\begin{remark}
  Taking a Tate twist has the effect of
  increasing or decreasing the exponents of the Hodge polynomial by an even integer.
  When working with Galois representations, it is customary
  to take the minimal Tate twist that makes the Hodge numbers of all
  Galois conjugates non-negative (making them the Hodge numbers of an
  {\it effective} motive) because this is the standard normalization
  for automorphic forms. In particular, the resulting Euler factors
  have integral coefficients.  We will call this process the
  \emph{effective normalization}.

  In the present article (and also in the Pari/GP code written by the
  third author) we do not make any additional Tate twist when
  considering HGM's unless explicitly mentioned. The reader should
  bear this in mind when matching a HGM to an automorphic form.
\end{remark}
The effective normalization is the Tate twist that shifts the minimum
value of the zig-zag diagrams over all embeddings to zero. As an
illustration, we have the following.

\vspace{5pt}

\noindent {\bf Example~\ref{example:non-ratnl}.}(continued)
The value $r-1=0$ so the effective normalization corresponds to a Tate
twist by the minimum value of the zig-zag function over the two
diagrams in Fig.~\ref{zigzagfig-3} which equals $-1$. Therefore the effective
normalization is the Tate twist
$$
\HGM((\tfrac{1}{2},\tfrac{1}{2}),(\tfrac{1}{4},1)|z)(1)
$$
In terms of Euler factors (like $L_{17}(T)$ above) it amounts to
replacing $T$ by $qT$.

\begin{remark}
  Let $\veca$ and $\vecb$ be generic rational parameters and let $N$
  be their least common denominator.  Let $F = \Q(\zeta_N)$ and let
  $H$ be the subgroup of $(\Z/N)^\times$ defined in \ref{def:H}. Then
  for $\z-spec \in \Q$, the motive $\HGM(\veca,\vecb|\z-spec)$ is expected to
  be defined over $F^H$ and hence has $\phi(N)/|H|$ Hodge vectors
  (indexed by the embeddings of $F^H$ into $\CC$) instead of
  $\phi(N)$. This is consistent with the fact that if $j \in H$ then
  the Hodge vector of $\veca,\vecb$ is the same as the one of
  $j\veca,j\vecb$.
\end{remark} 

\section{Jacobi motives}
\label{section:Jacobi}
In this section we review some basic facts about Jacobi motives and set
up the notation we will use. The main reference are the articles
\cite{MR0029393,MR0051263}, \cite{zbMATH03948452}, \cite{MR935127} and \cite{MR3838691} (the software package
Magma \cite{MR1484478} contains an implementation of Jacobi motives
due to Mark Watkins).

Fix an integer $N>1$.  For a prime ideal $\id{p}$ in
$F:=\Q(\zeta_N)\subseteq \CC$ of norm $q$, where $\zeta_N$ is a primitive
$N$-st root of unity, let $\chip$ be the character
\begin{equation}
  \label{eq:gen-defi}
\chip: (\calO_F/\id{p})^\times \to \calO_F^\times.  
\end{equation}
of order $N$ satisfying
$$
\chip(x)\equiv x^{(q-1)/N} \bmod \id{p}.
$$
Extend the definition by setting
$\chip(x)=0$ if $\id{p} \mid x$. Here $\calO_F$ is the ring of
integers of $F$.

\begin{remark}
  \label{rem:char-extension}
  If $L$ is a finite extension of $F$ and $\widehat{\id{p}}$ is a
  prime ideal of $L$ dividing $\id{p}$ we can similarly define a
  character
  $\widehat{\chip} : (\Om_L/\widehat{\id{p}})^\times \to \CC^\times$
  by composing $\chip$ with the norm map. To avoid heavy notation, we
  will also denote by $\chip$ the character $\widehat{\chip}$.
\end{remark}

We define Gauss sums as functions on
$\frac1N\Z/\Z$ as follows.  Fix a non-trivial additive character
$\psi$ on $\FF_q$ and let
\begin{equation}
  \label{eq:gauss-sum-defi}
g(\psi,a,\id{p}) := \sum_{x \in \FF_q^\times} \psi(x)\chi_{\id{p}}^{Na}(x),  
\qquad \qquad
a \in\frac1N\Z/\Z.
\end{equation}
The dependence on the choice of additive character is
straightforward. For any non-zero $y\in \calO_F$ coprime to $\id{p}$
if $\psi'(x):=\psi(y^{-1}x)$ then
\begin{equation}
\label{gauss-sum-prop}
g(\psi',a,\id{p})=\chip^{Na}(y)g(\psi,a,\id{p}).
\end{equation}

\begin{definition}
  \label{def:jacobi-motive}
  Let ${\pmb \theta}=\sum_in_i\langle \theta_i\rangle \in \Z[\frac1{N}\Z/\Z]$ satisfying
\begin{equation}
\label{jacobi-condition}
\sum_i n_i\theta_i \equiv 0 \bmod \Z.
\end{equation}
The \emph{Jacobi sum} attached to ${\pmb \theta}$ for the prime ideal
$\id{p}$ is defined as
  \begin{equation}
    \label{eq:Jac-motive}
    \JacMot({\pmb \theta})(\id{p}) := (-1)^{\sum_in_i+1} 
    \prod_{i}g(\psi,\theta_i,\id{p})^{n_i}
  \end{equation}
\end{definition}
It is easy to verify, thanks to~\eqref{jacobi-condition}
and~\eqref{gauss-sum-prop}, that the definition is independent of the
choice of the additive character~$\psi$ and therefore there is no need
to include it in the notation.

By the main result of \cite{MR0051263} the map
$$
\id{p}\mapsto \JacMot({\pmb \theta})(\id{p})
$$
determines a Gr\"ossencharacter $\JacMot({\pmb \theta})$ of $F$. In
particular, it has an associated compatible system of $l$-adic Galois
representations where $\JacMot({\pmb \theta})(\id{p})$ equals the
trace of $\Frob_{\id{p}}$. This is an incarnation of the \emph{Jacobi
  motive} associated to ${\pmb \theta}$ (see \cite[\S
5]{zbMATH03948452}). We will denote this pure motive also
by~$\JacMot({\pmb \theta})$ if there is no risk of confusion.

By \cite[Theorem 1]{zbMATH03948452} (see also Remark 2.3.2 of loc. cit) the
infinity type of the Gr\"ossencharacter associated
to~$\JacMot({\pmb \theta})$ is given by
\begin{equation}
  \label{eq:hodge-numbers}
\sigma_j\mapsto \sum_in_i\{j\theta_i\}, \qquad \gcd(j,N)=1,
\end{equation}
where $\{\cdot\}$ denotes the fractional part of real numbers and
$\sigma_j\in \Gal(F/\Q)$ is the automorphism satisfying
$\sigma_j(\zeta_N)=\zeta_N^j$. The (motivic) weight of~$\JacMot({\pmb \theta})$ equals 
$$
w:=\sum_{i} n_i,
$$
where the sum only includes indexes $i$ for which
$\theta_i \not \in \Z$.

The motive~$\JacMot({\pmb \theta})$ appears
(up to a Tate twist) in the middle cohomology $H^\bullet(X,F)$ of a smooth
Fermat hypersurface~$X$ (see for example~\S10 of
\cite{zbMATH03948452}).
Moreover, if we define
\begin{equation}
  \label{eq:hodge}
  p:=\sum_i n_i \{\theta_i\},\qquad q:=\sum_i n_i \{-\theta_i\}
\end{equation}
then (up to a Tate twist) $\JacMot({\pmb \theta})\subseteq H^{p,q}(X,F)$.

It will sometimes be convenient to use an alternative notation for Jacobi
motives as a pair of tuples of rational numbers ${\pmb
  a}=(a_1,\ldots,a_r)$ and ${\pmb b}=(b_1,\ldots,b_s)$ 
with common denominator $N$ such that
$$
\sum_{i=1}^ra_i \equiv \sum_{i=1}^sb_i\bmod \Z,
$$
corresponding
to 
$$
{\pmb \theta}:=\sum_{i=1}^r\langle a_i\rangle - \sum_{i=1}^s\langle b_i\rangle.
$$
We will then simply write $\JacMot({\pmb a},{\pmb b})$ instead
of~$\JacMot({\pmb \theta})$.

\begin{example} Consider the Jacobi motive $\JacMot({\pmb a},{\pmb b})$ where
$$
\begin{array}{ll}
{\pmb a}&:=(1/10, 1/10, 1/10, 3/10, 13/30, 7/10, 23/30, 9/10)\\
{\pmb b}&:=( 1/5, 1/3, 2/5,2/3, 4/5,1/5, 3/10, 1/2).
\end{array}
$$
It is perhaps not immediately obvious but this motive is 
an Artin motive. But we can quickly check that we have
$$
 \sum_i\{ja_i\}-\sum_i \{jb_i\}=0, \qquad \qquad \gcd(j,30)=1,
$$
and in particular $p=q=0$ for all $j$. Computations with MAGMA show that the
associated Artin motive is given by a Dirichlet character of
$\Q(\zeta_5)$ of order $10$ and conductor
$2^2\cdot3\cdot(1-\zeta_5)^2$.
\end{example}

The Jacobi sums $\JacMot({\pmb \theta})$ belong to $F$ and for
 $\sigma_j \in \Gal(F/\Q)$ as before we have
$$
\sigma_j(\JacMot({\pmb \theta}))(\id{p})=\JacMot(j{\pmb \theta})(\id{p}),
$$
where
$$
j{\pmb \theta}:=\sum_in_i\langle j\theta_i\rangle.
$$

Let $T = \{\sigma \in \Gal(F/\Q) \; : \; \sigma(\JacMot({\pmb \theta})) =
\JacMot({\pmb \theta})\}$. Then the base field of the motive
$\JacMot({\pmb \theta})$ is $F^T$ (see \S 2.3 of
\cite{MR3838691}). Note its resemblance with
Conjecture~\ref{conjecture:field-of-defi}.
In some very particular instances, the Gr\"ossencharacter
$\JacMot({\pmb \theta})$ can be defined over a proper subextension of
$F^T$.

\begin{example} Consider the Jacobi motive $\JacMot(\pmb \theta)$ where
  \[
   \pmb \theta = \left\langle \frac{1}{3}\right\rangle +\left\langle
      \frac{2}{3}\right\rangle +\left\langle \frac{1}{5}\right\rangle
    +\left\langle \frac{4}{5}\right\rangle+ \left\langle
      \frac{7}{15}\right\rangle+ \left\langle
      \frac{8}{15}\right\rangle.
  \]
  The set $T$ of elements in $\left(\Z/15\right)^\times$ fixing
  $\theta$ equals $\{\pm 1\}$, so the base field of
  $\JacMot({\pmb \theta})$ is $\Q(\zeta_{15})^+$ (the maximal totally
  real subfield of $\Q(\zeta_{15})$). However the values of the Jacobi
  sum attached to our choice of parameters at a prime ideal $\id{p}$
  of $\Q(\zeta_{15})$ equals $\normid{p}^3$ (see
  Lemma~\ref{lemma:Tate-twist}). Then there is a second extension of
  the Gr\"ossencharacter to $\Gal_\Q$ given by the cubic power of the
  cyclotomic character.
\end{example}

\section{Finite Hypergeometric Sums}
\label{section:finite}

Let us recall Katz's definition (given in~\cite[p.258]{katz}) of the
finite version of a hypergeometric series. Let
$\pmb \alpha= (\alpha_1 , \ldots, \alpha_n), \vecb = (\beta_1
, \ldots, \beta_n) \in \Q^n$ be arbitrary vectors of rational numbers
 and let $N$ be the common denominator of the $\alpha$'s and
 $\beta$'s. For the remaining of the section we fix
$q\equiv 1 \bmod N$ a prime power
and a generator $\gen$ of $\widehat{\FF_q^\times}$. Following
\cite[p. 258]{katz} (but note that our parameter $z$ is the inverse of
Katz's though it matches the normalization in~\cite{BCM}), define for
$z \in \FF_q$ the exponential sum (here $x_i,y_j\in \FF_q^\times$):
\begin{align}
\text{Hyp}_q(\veca,\vecb|z) = \sum_{zx_1 \cdots x_n = y_1 \cdots y_n}
  \psi(x_1+ \cdots + x_n - y_1 - \cdots -
  y_n)\gen(\vecx)^{\veca (q-1)}\overline{\gen}(\vecy)^{\vecb (q-1)}, 
\end{align}
where
$\gen(\vecx)^{\veca(q-1)} = \gen(x_1)^{\alpha_1(q-1)} \cdots
\gen(x_n)^{\alpha_n(q-1)}$ and similarly with $\beta$. Katz relates
(using the Lefschetz trace formula) such finite hypergeometric sum to
the trace of Frobenius acting on some concrete hypergeometric
$\D$-module~\cite[Chap. 8.2]{katz}. 

It will be convenient to expand $\text{Hyp}_q$ in terms of characters of
$\FF_q^\times$. The calculation is straightforward (see \cite[\S
8.2.8]{katz}); we sketch it here for the reader's convenience. The
coefficient $c_\Char$ of $\Char \in \widehat{\FF_q^\times}$ equals
$$
c_\Char=\frac 1{q-1}\sum_{z\in \FF_q^\times} \overline \Char(z)
\text{Hyp}_q(\veca,\vecb|z).
$$
Interchanging order of summation we get
$$
c_\Char=\frac 1{q-1}
 \sum_{(\vecx,\vecy)\in \T_z}
  \psi(x_1+ \cdots + x_n - y_1 - \cdots -
  y_n)\gen(\vecx)^{\veca(q-1)}\overline{\gen}(\vecy)^{\vecb(q-1)}
  \Char(x_1\cdots x_ny_1^{-1}\cdots y_n^{-1})
$$
and hence
$$
c_\Char=\frac 1{q-1}\prod_{i=1}^n\sum_{x_i\in
  \FF_q^\times}\psi(x_i)\gen(x_i)^{(q-1)\alpha_i} \Char(x_i) 
\prod_{i=1}^n\sum_{y_i\in \FF_q^\times}\psi(-y_i){\overline \gen(y_i)}^{(q-1)\beta_i}
\overline \Char(y_i). 
$$
In terms of Gauss sums
\begin{equation}
\label{gsum}
g(\psi,\Char) := \sum_{x \in \FF_q^\times} \Char(x)\psi(x),
\end{equation}  
we finally have
$$
c_\Char=\frac 1{q-1}\prod_{i=1}^n g(\psi,\gen^{(q-1)\alpha_i} \Char)
\prod_{i=1}^n g(\psi^{-1},{\overline \gen}^{(q-1)\beta_i}\overline \Char)
$$
and therefore
$$
\text{Hyp}_q(\veca,\vecb|z) = \frac 1{q-1}\sum_\Char 
\prod_{i=1}^n g(\psi,\gen^{(q-1)\alpha_i} \Char)
\prod_{i=1}^n g(\psi^{-1},{\overline \gen}^{(q-1)\beta_i}\overline \Char) \Char(z).
$$
It is instructive to check the simplest non-trivial case where $n=1$,
$\alpha_1=1/2,\beta_1=1$ and $q=p$ is an odd prime. In this case, a
quick calculation shows that 
$$
\text{Hyp}(1/2,1\,|\,z)=g(\psi_{z-1},\epsilon),
$$
where $\psi_u(x):=\psi(ux)$ and
$\epsilon$ is the quadratic character of
$\FF_p^\times$. It follows that
$$
\text{Hyp}(1/2,1\,|\,z) = \epsilon(1-z)\sqrt{p^*}, \qquad \qquad
p^*:=(-1)^{(p-1)/2}p, 
$$
as is well known. In particular, the values $\text{Hyp}(1/2,1\,|\,\z-spec)$ as
$p$ varies, for fixed $\z-spec\in
\Q$ say, do not lie in any given number field. We therefore cannot
expect $\text{Hyp}(1/2,1\,|\,\z-spec)$ to be the trace of Frobenius of a motive.  To
achieve this we normalize the hypergeometric sum
$\text{Hyp}_q(\veca,\vecb|z)$ by dividing by the appropriate
constant so that the sum of its values over all $z\in
\F_q^\times$ equals $-1$. Concretely, we consider the following.
\begin{definition}
\label{defi:q-finite-hgm}
For $z \in \FF_q$, define the
finite hypergeometric sum $H_q(\veca,\vecb|z)$ by
\begin{equation}
  \label{eq:hpergeom-sum}
  H_q(\veca,\vecb|z):=\frac1{1-q} \sum_\Char
\frac{\JJ(\veca\Char,\vecb\Char)}{\JJ(\veca,\vecb)}\Char(z),
\end{equation}
where
$$
\JJ(\veca\Char,\vecb\Char):=\prod_{i=1}^n g(\psi,\gen^{(q-1)\alpha_i} \Char)
\prod_{i=1}^n g(\psi^{-1},{\overline \gen}^{(q-1)\beta_i}\overline \Char).
$$
\end{definition}
Note that for any integer $a$ coprime with $p$ we have
\begin{equation}
  \label{eq:additive-char-power}
g(\psi^a,\Char)=\overline\Char(a)g(\psi,\Char).
\end{equation}
Therefore, we may alternatively define
\begin{equation}
  \label{eq:j-alt}
  \JJ^*(\veca\Char,\vecb\Char)=\Char(-1)^n\prod_{i=1}^n g(\psi,\gen^{(q-1)\alpha_i} \Char)
\prod_{i=1}^n g(\psi,{\overline \gen}^{(q-1)\beta_i}\overline \Char).
\end{equation}
and then
$$
  H_q(\veca,\vecb|z)=\frac1{1-q} \sum_\Char
\frac{\JJ^*(\veca\Char,\vecb\Char)}{\JJ^*(\veca,\vecb)}\Char(z).
$$

To simplify the notation set $g(m):=g(\psi,\gen^m)$ for  $m \in
\Z$. Then, explicitly,
\begin{equation}
\label{eq:hpergeom-BCM}  
H_q(\veca,\vecb|z)=\frac{1}{1-q} \sum_{m=0}^{q-2} \prod_{i=1}^{n}
  \left(\frac{g(m+\alpha_i(q-1)) g(-m-\beta_i (q-1))}{g(\alpha_i (q-1))
      g(-\beta_i(q-1))}\right) \gen((-1)^n z)^m.
\end{equation}
The finite hypergeometric sum $H_q(\veca,\vecb|z)$ does
not depend on the choice of the additive character $\psi$, but it does
(in general) depend on the choice of the character $\gen$.
For our previous rank one example we have (see~\S\ref{section:rank1})
$$
H_p(1/2,1\,|\z-spec)=\epsilon(1-\z-spec),
$$
where $\epsilon$ is the quadratic character of $\F_q$,
matching
$$
\sum_{n\geq 0}\frac{(1/2)_n}{(1)_n}\z-spec^n=\frac1{\sqrt{1-\z-spec}}.
$$
Note that $\sum_{p\in \F_p^\times}H_p(1/2,1\,|\z-spec)=-\epsilon(1)=-1$ as desired.

\subsection{Properties of $H_q(\veca,\vecb|z)$}

\begin{proposition}
  \label{prop:props-hyperg-sum}
The hypergeometric sum~$H_q(\veca,\vecb|z)$ satisfies the following
basic properties 
\begin{enumerate}
\item 
(Ordering)
  $H_q(\veca,\vecb|z)$ is independent of the ordering of
  $\{\alpha_1,\ldots,\alpha_n\}$ and $\{\beta_1,\ldots,\beta_n\}$.
\item
(Inversion) For $z \in \FF_q^\times$ we have
$$
    H_q(\veca,\vecb|z) =
    H_q(-\vecb,-\veca|z^{-1}).
$$

\item
(Galois action)  Let $z \in \FF_q^\times$ and $\sigma \in \Gal_{\Q}$ with
  $\sigma(\zeta_N)=\zeta_N^j$ for some $j$ coprime with $N$ then
  $$
H_q(\veca,\vecb|z)^\sigma=H_q(j\veca,j\vecb|z)
$$

\item
(Coefficient Field)  Let $z \in \FF_q^\times$ and let $H$ be the subgroup of $(\Z/N)^\times$ defined in
  \ref{def:H}. We have
  $$
  H_q(\veca,\vecb|z) \in \Q(\zeta_N)^H.
$$
\item
(Twists) Let $\rho\in \Q$ and
$\rho\veca:=(\alpha_1+\rho,\ldots,\alpha_n+\rho)$ and 
$\rho\vecb:=(\beta_1+\rho,\ldots,\beta_n+\rho)$ then
$$
H_q(\rho\veca,\rho\vecb|z)=
(-1)^{n(q-1)\rho}\overline{\gen}(z)^{(q-1)\rho}\frac{\JJ(\veca,\vecb)}
{\JJ(\rho\veca,\rho\vecb)} H_q(\veca,\vecb|z),
$$
whenever both sides are well defined.
\item
  (Non-generic) If $\alpha_1\equiv \beta_1 \bmod \Z$ let $\vecg=(\alpha_2,\ldots,\alpha_n)$ and $\vecd=(\beta_2,\ldots,\beta_n)$. Then
$$
H_q(\veca,\vecb|z)=
q^{\delta}\left(\frac{\JJ((-\alpha_1)\vecg,(-\alpha_1)\vecd)}{\JJ(\vecg,\vecd)}+qH_q(\vecg,\vecd)\right),
$$
where 
$$
\delta=
\begin{cases}
\, \, \, \, 0 &   \text{if } \alpha_1 \in \ZZ\\
-1 &   \text{otherwise}
\end{cases}
$$

\end{enumerate}
\end{proposition}

\begin{proof} The first property is clear from its definition. To prove (2), it follows from \eqref{eq:j-alt} that
  \[
\JJ(\veca\Char,\vecb\Char)=\JJ(-\vecb\,\overline{\Char},-\veca\,\overline{\Char}).
  \]
  The result follows from the change of variables
  $\Char \to \overline{\Char}$ in~\eqref{eq:hpergeom-sum}.

\noindent (3)   Since $\sigma$ is a ring morphism,
  \begin{equation}
    \label{eq:proof-gal-action}
  \sigma(H_q(\veca,\vecb|z))=\frac{1}{1-q} \sum_{\Char} \prod_{i=1}^{n}
  \left(\frac{\sigma(g(\psi,\gen^{(q-1)\alpha_i}\Char))\sigma(g(\psi^{-1},\gen^{-(q-1)\beta_i}\Char^{-1})}{\sigma(g(\psi,\gen^{(q-1)\alpha_i}))\sigma(g(\psi^{-1},\gen^{-(q-1)\beta_i}))}\right) \sigma(\Char(z)).
  \end{equation}
  If $\sigma \in \Gal_\Q$, $\Char$ is a multiplicative character of
  $\FF_q^\times$ and $\psi$ is an additive character, then
  $$\sigma(g(\psi,\Char)) = g(\psi^\sigma,\Char^\sigma),$$ where
  $\Char^{\sigma}(x)=\sigma(\Char(x))$ (similarly for $\psi$).
  Let $b$
  be an integer (prime to $q-1$) satisfying
  $\sigma(\zeta_{q-1})=\zeta_{q-1}^b$, for $\zeta_{q-1}$ a primitive
  $(q-1)$-th root of unity. Then
\[
  \sigma(\gen^{m+a(q-1)}) = \gen^{bm + ba(q-1)}=\gen^{bm} \gen^{ba(q-1)}.
\]
The rational number $ba$ equals $ja$ up to translation by an
integer. Then changing the summation order (replacing $\Char$ by $\sigma(\Char)$),
the left hand side of (\ref{eq:proof-gal-action}) becomes
\[
  \frac{1}{1-q} \sum_{\Char} \prod_{i=1}^{n}
  \left(\frac{g(\psi,\gen^{(q-1)j\alpha_i}\Char))g(\psi^{-1},\gen^{-(q-1)j\beta_i}\Char^{-1})}{g(\psi,\gen^{(q-1)j\alpha_i})g(\psi^{-1},\gen^{-(q-1)j\beta_i})}\right) \Char(z).
\]
(4) For $z \in \FF_q$, the standard properties of Gauss sums imply
that the value $H_q(\veca,\vecb|z)$ is an element of
$\Q(\zeta_{q-1})$ (see also \cite[Proposition 3.2]{BCM}).  Note that
all elements of $(\Z/(q-1))^\times$ that are congruent to $1$ modulo
$N$ leave the sets $\{\alpha_1,\ldots,\alpha_n\}$ and
$\{\beta_1,\ldots,\beta_n\}$ stable under multiplication (up to
translation by integers), hence the second statement 
(and Galois theory) imply that actually
$H_q(\veca,\vecb|z) \in F=\Q(\zeta_N)$. Since
multiplication by elements of $H$ also fix the sets
$\{\alpha_1,\ldots,\alpha_n\}$ and $\{\beta_1,\ldots,\beta_n\}$, the
result follows.

\noindent (5) It follows from~\eqref{eq:j-alt} that
\[
\JJ(\rho\veca\Char,\rho\vecb\Char)=\JJ(\veca\Char\gen^{(q-1)\rho},\vecb\Char\gen^{(q-1)\rho})(-1)^{n(q-1)\rho}.
  \]
  Then
  \[
    H_q(\rho\veca,\rho\vecb|z)=\frac{1}{1-q}\sum_\Char \frac{\JJ(\rho\veca\Char,\rho\vecb\Char)}{\JJ(\rho\veca,\rho\vecb)}\Char(z)=(-1)^{n(q-1)\rho}\frac{\JJ(\veca,\vecb)}{\JJ(\rho\veca,\rho\vecb)}\overline{\gen}(z)^{(q-1)\rho}H_q(\veca,\vecb|z).
    \]
\noindent (6) If $\psi$ is an additive character of $\FF_q$ and $\chi$ is a multiplicative character, then 
\[
  g(\psi,\chi)g(\psi^{-1},\chi^{-1})=
  \begin{cases}
    q & \text{ if }\chi= 1,\\
    1 & \text{ otherwise}.
  \end{cases}
  \]
  If $\alpha_1$ and $\beta_1$ are integers, then
  \[
H_q(\veca,\vecb|z)=\frac{1}{1-q}\left(\sum_{\varphi \neq 1}\frac{q\JJ(\vecg\varphi,\vecd\varphi)}{\JJ(\vecg,\vecd)}+1\right)=q H_q(\vecg,\vecd) + 1.
\]
Similarly, if $\alpha_1\equiv \beta_1 \bmod \ZZ$ but they are not
integers, let $\varphi_0=\gen^{-(q-1)\alpha_1}$. Then
  \[
H_q(\veca,\vecb|z)=\frac{1}{1-q}\left(\sum_{\varphi \neq \varphi_0}\frac{\JJ(\vecg\varphi,\vecd\varphi)}{\JJ(\vecg,\vecd)}+\frac{\JJ(\vecg\varphi_0,\vecd\varphi_0)}{q\JJ(\vecg,\vecd)}\right)=H_q(\vecg,\vecd)+\frac{\JJ((-\alpha_1)\vecg,(-\alpha_1)\vecd)}{q\JJ(\vecg,\vecd)}.
\]  
\end{proof}

Recall that for $L$ a finite extension of $F=\Q(\zeta_N)$ and
$\idL{p}$ a prime ideal of $L$, we defined in
Remark~\ref{rem:char-extension} an order $N$ character
$\widehat{\chi_{\id{p}}}$ of $(\Om_L/\idL{p})^\times$.
\begin{definition}
  \label{defi:finite-hgm}
  Let $\veca,\vecb\in \Q^n$ be generic parameters and let $N$ be their
  least common denominator. Let $L$ be a finite extension of
  $F$ and let $\idL{p}$ be a prime ideal of $L$ of norm
  $q$ not dividing $N$. Let $\z-spec \in L$ satisfy
  $v_{\idL{p}}(\z-spec)=v_{\idL{p}}(\z-spec-1)=0$. The
  \emph{$\idL{p}$-hypergeometric sum}, denoted
  $H_{\idL{p}}(\veca,\vecb|\z-spec)$ is the value
  $H_q(\veca,\vecb|\z-spec)$ for $\gen$ a generator of
  $\widehat{\FF_q^\times}$ satisfying
  $\gen^{\frac{q-1}{N}}=\widehat{\chi_{\id{p}}}^{-1}$.
\end{definition}

The independence of $\gen$ follows from the following result.

\begin{lemma}
  \label{lemma:character-dependence}
  Let $\veca,\vecb\in \Q^n$ be generic parameters and let
  $N$ be their least common denominator. Let $\gen_1$ and $\gen_2$ be
  two generators of $\widehat{\FF_q^\times}$ such that
  $\gen_1^{\frac{q-1}{N}} = \gen_2^{\frac{q-1}{N}}$. Let $\z-spec \in \FF_q^\times$,
  $\z-spec \neq 1$. Then the values of $H_q(\veca,\vecb|z)$
 for $\gen_1$ and $\gen_2$ are the same.
\end{lemma}

\begin{proof}
  Let $j \in (\Z/(q-1))^\times$ be such that $\gen_2=\gen_1^j$. The
  hypothesis $\gen_1^{\frac{q-1}{N}} = \gen_2^{\frac{q-1}{N}}$ implies that
  $j \equiv 1 \pmod N$. The value $H_q(\veca,\vecb|z)$
  taking $\gen_2$ as a generator equals
  \[
\frac{1}{1-q}\sum_\varphi \frac{\prod_{i=1}^ng(\psi,\gen_1^{(q-1)j\alpha_i}\varphi)g(\psi^{-1},\overline{\gen_1}^{(q-1)j\beta_i}\overline{\varphi})}{\prod_{i=1}^n g(\psi,\gen_1^{(q-1)j\alpha_i})g(\psi^{-1},\overline{\gen_1}^{(q-1)j\beta_i})}\varphi(z),
\]
which matches the value of the finite hypergeometric series for the
generator $\gen_1$ with parameters $j\veca$, $j\vecb$. Since
$j\equiv 1 \pmod N$, $\veca = j\veca \pmod{\Z^n}$ and the same is true
for $\vecb$ hence the result.
\end{proof}

\subsection{Relation with the $p$-adic Gamma function}
\label{subsection:gammap}
For the reader's convenience, we will use bold letters for
$p$-adic functions.
Let $p$ be a prime number. Recall the following definition (see \cite{zbMATH03482465}).

\begin{definition}
  The \emph{$p$-adic Gamma function} is the continuous function
  $\Gammap:\ZZ_p \to \QQ_p$ that at a positive integer $n$ takes the
  value
  \begin{equation}
    \label{eq:p-adic}
\Gammap(n)=(-1)^n\prod_{\stackrel{i=1}{p\nmid i}}^{n-1} i.    
  \end{equation}
\end{definition}

As proved in \cite[Lemma 1]{zbMATH03482465}, the $p$-adic Gamma
function satisfies the relation
\[
  \Gammap(n+p^rm) \equiv \Gammap(n) \pmod{p^r},
\]
hence condition~(\ref{eq:p-adic}) determines it uniquely.
An important property of the $p$-adic Gamma function is that it determines
an analytic function.
\begin{theorem}[Morita]
  Set $Q=8$ if $p=2$ and $Q=1$ otherwise. Then the $p$-adic $\Gammap$
  function is an analytic function from $Q\ZZ_p \to \QQ_p$.
\end{theorem}
\begin{proof}
  See \cite[Theorem 3]{zbMATH03482465}.
\end{proof}
Let $q=p^f$ for a positive integer $f$. The following functions play
an important role.
\begin{definition}
  Let $\star \in \{0,\infty\}$. The function
  $\{ \,\cdot \, \}^\star:\Q/\Z \to \Q$ is defined by
\begin{equation}
  \label{eq:bracket}
\{x\}^\infty:= x - [x], \qquad \text{and} \qquad \{x\}^0:= 1-\{-x\}^\infty,  
\end{equation}
where $[x]$ denotes the floor of $x$ (so $\{x\}^\infty \in [0,1)$
while $\{x\}^0 \in (0,1]$). Consider the following functions

  \begin{eqnarray}
    \label{eq:gammap}
\Gammap_q^\star: \ZZ_{(p)}/\ZZ \to \QQ_p, & &  \Gammap_q^\star(x):=\prod_{i=0}^{f-1}\Gammap(\{p^ix\}^\star).\\
  \label{eq:eta-defi}
\eta_q^\star:\Q/\ZZ \to \QQ, & &   \eta_q^\star(x):=\sum_{i=0}^{f-1}\{p^ix\}^\star.  \\
  \label{eq:eta-m-defi}
\eta_{q,m}^\star:\Q/\ZZ \to \QQ,& &   \eta^\star_{q,m}(x):=\eta_q^\star\left(x+\frac{m}{1-q}\right) - \eta_q^\star(x).
\end{eqnarray}
If $\vecx \in (\Q/\Z)^n$, we extend the last map linearly component-wise, namely
  \[
    \eta_{q,m}^\star(\vecx) = \sum_{i=1}^n \eta_{q,m}^\star(x_i).
\]
\end{definition}

\begin{definition}
  Let $x \in \QQ/\Z$ be such that its denominator is not divisible by
  $p$. The (finite) \emph{$q$-orbit} of $x$ is the set
  \[
    \orbit_q(x):=\{x,qx,q^2x,\ldots\} \subset \QQ/\Z.
  \]
  We denote by $\len_q(x)$ its size.
\end{definition}
\begin{remark}
  If $x \in \Q/\Z$ and its $q$-orbit has $r$ elements, then
  $(q^r-1)x \in \ZZ$. In particular, the $q$-orbit of $x$ has a unique
  element precisely when $\denominator(x) \mid q-1$.
\label{remark:orbit-size}
\end{remark}

\begin{lemma}
  Let $x \in \Q/\Z$ be such that its denominator is not divisible by
  $p$.  Then
  \begin{equation}
    \label{eq:gamma-orbit}
    \prod_{u \in \orbit_q(x)} \Gammap_q^\star(u) = \Gammap_{q^{\len_q(x)}}^\star(x).
  \end{equation}
  Similarly,
  \begin{equation}
    \label{eq:eta-orbit}
    \sum_{u \in \orbit_q(x)} \eta_{q,m}^\star(u) = \eta_{q^{\len_q(x)},m}^\star(x).
  \end{equation}
  \label{lemma:gamma-prod}
\end{lemma}
\begin{proof}
  By definition of the $q$-Gamma function
  \[
    \prod_{u \in \orbit_q(x)} \Gammap_q^\star(u) = \prod_{i=0}^{\len_q(x)-1} \,\prod_{j=0}^{f-1}\Gammap^\star(\{p^j q^ix\}) = \prod_{i=0}^{f\len_q(x)-1}\Gammap^\star(\{p^ix\})=\Gammap_{q^{\len_q(x)}}^\star(x).
  \]
The second statement is an additive version of the same proof.
\end{proof}

\begin{definition}
  Let $p$ be a prime number, $q$ a $p$-th power and  $n$  a positive
  integer. Let $x \in \ZZ_{(p)}/\ZZ$. For $\star \in \{0,\infty\}$, define the
\emph{$p$-adic  Pochammer symbol}
  \begin{equation}
    \label{eq:pochammer}
  (x)_{q,n}^\star:=\frac{\Gammap_q^\star(x+\frac{n}{1-q})}{\Gammap_q^\star(x)}.
  \end{equation}
\end{definition}

Let $\Om$ denote the ring of integers of $\Q_p(\zeta_{q-1})$ and
let $\id{p}$ be its maximal ideal. For the rest of the section we let
$\F_q:=\Om/\id{p}$.

Let $\Teich$ be the
\emph{Teichmuller character}
\begin{equation}
  \label{eq:Teich}
\Teich:\F_q^\times \to \Om^\times  
\end{equation}
characterized by $\pi\circ\Teich=\Id$, where  $\pi:\Om \to \F_q$ is the reduction map.
(This character  is the $p$-adic version of the character $\chip$ of
\eqref{eq:gen-defi} for $N=q-1$).

\begin{definition}
  Let $\veca,\vecb \in (\Q/\ZZ)^n$. Let $p$ be a prime
  number not dividing the denominator of $\veca$ nor
  $\vecb$ and let $q = p^f$. The \emph{$q$-adic finite
    hypergeometric sum} is the function defined 
for   $z \in \F_q$ by
  \begin{equation}
    \label{eq:hyperg-geom}
    \hgmpadic_q(\veca,\vecb|z):=\frac{1}{1-q}\sum_{m=0}^{q-2} \frac{(\alpha_1)_{q,m}^\infty \cdots (\alpha_n)_{q,m}^\infty}{(\beta_1)_{q,m}^0\cdots (\beta_n)_ {q,m}^0}(-p)^{\eta_{q,m}^\infty(\veca)-\eta_{q,m}^0(\vecb)}\Teich(z)^m.
  \end{equation}
\end{definition}
\begin{remark}
  The routine ``hgm'' of the Pari/GP package described in
  Example~\ref{example:1} computes, given as input the specialization
  $\z-spec$ (a rational number), the parameters $\veca,\vecb$ and a
  prime number $p$, the $p$-adic number
  $\hgmpadic_q(\veca,\vecb|\z-spec)$ of \eqref{eq:hyperg-geom}. By
  default, the output is an element of the $p$-adic field $\Q_p$
  computed with a default precision of 20 digits.
\end{remark}

\begin{definition}
  Let $\veca \in (\Q/\Z)^n$ and let $q=p^f$, with $p$ not
  dividing the denominator of the coordinates of $\veca$. The
  vector $\veca$ is called \emph{$q$-stable} if 
  $\{q\alpha_1,\ldots,q\alpha_n\} = \{\alpha_1,\ldots,\alpha_n\}$.
\end{definition}
We will prove in Theorem~\ref{thm:algebricity} that
$\hgmpadic_q(\veca,\vecb|z)$ is algebraic for $q$-stable parameters
$\veca$ and $\vecb$. For that purpose, we start with some preliminaries.

Let $\psip:\FF_p \to \ZZ_p[\zeta_p]^\times$ be a non-trivial
additive character.  For $a \in \frac{1}{q-1}\ZZ/\ZZ$, define the
\emph{$p$-adic Gauss sum} 
\begin{equation}
  \label{eq:p-adic-Gauss}
  \gaussp(\psip,a,q):=\sum_{u \in \FF_q^\times} \Teich(u)^{-a(q-1)}\psip(\trace_{\FF_q/\FF_p}u). 
\end{equation}

Let $M=\Q(\zeta_{q-1})$ and let $L=M(\zeta_p)$. Let $\id{p}$ be a
prime ideal of $M$ dividing $p$ and let $\id{q}$ be a prime ideal of
$L$ dividing $\id{p}$.  Let $\iota:L \hookrightarrow L_{\id{q}}$ be
the natural inclusion, where $L_{\id{q}}$ denotes the completion of
$L$ at $\id{q}$. Let $\psi:\FF_q \to \Z[\zeta_p]^\times$ be the
additive character satisfying $\iota \circ \psi = \psip \circ \trace_{\FF_q/\FF_p}$.  Let
$\gen$ be a generator of $\widehat{\FF_q^\times}$ such that
$\iota \circ \gen = \Teich^{-1}$.  Then for $a \in \frac{1}{q-1}\ZZ/\ZZ$
  \begin{equation}
    \label{eq:gauss-comparison}
    \gaussp(\psip,a,q)=\iota(g(\psi,a,\id{p})),    
  \end{equation}
  where $g(\psi,a,\id{p}) \in L$ is the complex Gauss sum defined in
  \eqref{eq:gauss-sum-defi}.

  Set $\zeta_p := \psi(1)$ and let $\pi \in \overline{\QQ}$ satisfy
  \[
    \pi^{p-1}=-p, \qquad \pi \equiv (\zeta_p-1) \pmod{(\zeta_p-1)^2}.
  \]

\begin{theorem}[Gross-Koblitz]
  Let $a \in \frac{1}{q-1}\ZZ$. Then
  \begin{equation}
    \label{eq:GK-infty}
    \gaussp(\psip,a,q)=-\pi^{(p-1)\eta_q^\infty(a)} \Gammap_q^\infty(a),
  \end{equation}
  and
  \begin{equation}
    \label{eq:GK-0}
   \frac{q}{\gaussp(\psip^{-1},-a,q)}= -\pi^{(p-1)\eta_q^0(a)}\Gammap_q^0(a).
  \end{equation}
\label{thm:GK}
\end{theorem}

\begin{proof}
  The first result follows from \cite[Theorem 1.7]{zbMATH03630880} and
  \eqref{eq:gauss-comparison}, taking (in Gross-Koblitz's notation)
  $N=q-1$.
  If we replace $\Gammap_q^\infty$ and $\eta_q^\infty$ by their
  $\Gammap_q^0$ and $\eta_q^0$ counterparts, we find that
  \[
    \pi^{(p-1)\eta_q^0(a)}\Gammap_q^0(a) =
    \begin{cases}
      q & \text{ if }a \in \ZZ,\\
      \gaussp(\psip,a,q) & \text{ if }a \not \in \ZZ.
    \end{cases}
\]
The formula follows from the well known relation
  \[
    \gaussp(\psip,a,q)\cdot \gaussp(\psip^{-1},-a,q) =
    \begin{cases}
      1 & \text{ if }a \in \ZZ,\\
      q & \text{ if }a \not \in \ZZ.
    \end{cases}
  \]
\end{proof}
\begin{remark}
  The reason for the minus sign in the formulas in the Theorem
  (compared to \cite[Theorem 1.7]{zbMATH03630880}) is due to a
  different choice of normalization for Gauss sums.
\end{remark}
\begin{theorem}
  Let $\veca,\vecb \in \QQ^n$ be $q$-stable
  vectors. Consider the decomposition of each set as a disjoint union
  of orbits (after relabeling)
  \[
   \{\alpha_1,\ldots,\alpha_n\}= \orbit_q(\alpha_1) \cup \cdots \cup \orbit_q(\alpha_u),\qquad \{\beta_1,\ldots,\beta_n\}= \orbit_q(\beta_1) \cup \cdots \cup \orbit_q(\beta_v).
 \]
 To ease notation, let $l_i^\infty = \len_q(\alpha_i)$ and
 $l_i^0 = \len_q(\beta_i)$ be the size of the respective orbits.  Then
we have  for $z \in \FF_q$:
\begin{enumerate}
\item
The following formula holds
 \begin{equation}
   \label{eq:Hq-algebricity}
   \hgmpadic_q(\veca,\vecb|z) = \frac{1}{1-q}\sum_{m=0}^{q-2} \prod_{i=1}^u \frac{\gaussp(\psip,\alpha_i+\frac{m}{1-q},q^{l_i^\infty})}{\gaussp(\psip,\alpha_i,q^{l_i^\infty})} \prod_{i=1}^v \frac{\gaussp(\psip^{-1},-\beta_i-\frac{m}{1-q},q^{l_i^0})}{\gaussp(\psip^{-1},-\beta_i,q^{l_i^0})}   \Teich(z)^m.
 \end{equation}
\item Let $L$ be a finite extension of $F=\Q(\zeta_N)$ and let
  $\id{p}$ be a prime ideal of $L$ of norm $q$. Let
  $\iota: L \hookrightarrow L_{\id{p}}$ be the
  natural map to its completion at $\id{p}$. The global and local 
  hypergeometric sums are related as follows. For $\z-spec \in
  L$ coprime with $\id{p}$ we have
  \begin{equation}
    \label{eq:p-adic-global}
    \iota(H_{\id{p}}(\veca,\vecb|\z-spec)) =
    \hgmpadic_q(\veca,\vecb|\overline{\z-spec}),
  \end{equation}
where $\overline{\z-spec}\in (\calO/\id{p})^\times$ is the reduction of $\z-spec$
modulo $\id{p}$.  In particular,
$\hgmpadic_q(\veca,\vecb|\overline{\z-spec})$ is algebraic over~$\Q$; 
  more precisely, $\hgmpadic_q(\veca,\vecb|\overline{\z-spec}) \in \iota(F^H)$.
\end{enumerate}
  \label{thm:algebricity}
\end{theorem}

\begin{proof}
1)  Since $q(\alpha_i + \frac{m}{1-q}) = q\alpha_i + \frac{m}{1-q}$ in
  $\Q/\Z$, the decomposition in $q$-orbits of
  $\{\alpha_1+\frac{m}{1-q},\ldots,\alpha_n+\frac{m}{1-q}\}$ mimics
  that of $\{\alpha_1,\ldots,\alpha_n\}$. Then using Lemma~\ref{lemma:gamma-prod}
  \[
(\alpha_1)_{q,m}^\infty \cdots (\alpha_n)_{q,m}^\infty=\prod_{i=1}^n\frac{\Gammap_q^\infty(\alpha_i+\frac{m}{1-q})}{\Gammap_q^\infty(\alpha_i)} = \prod_{i=1}^u \frac{\Gammap_{q^{l_i^\infty}}^\infty(\alpha_i + \frac{m}{1-q})}{\Gammap^\infty_{q^{l_i^\infty}}(\alpha_i)}.
\]
By Remark~\ref{remark:orbit-size}, both $(q^{l_i^\infty}-1)\alpha_i$ and
$(q^{l_i^\infty}-1)(\alpha_i + \frac{m}{1-q})$ are integers, so 
Theorem~\ref{thm:GK} implies that the product equals
\[
\prod_{i=1}^u \frac{\gaussp(\psip,\alpha_i+\frac{m}{1-q},q^{l_i^\infty}) }{\gaussp(\psip,\alpha_i,q^{l_i^\infty}) } \pi^{(1-p)\left(\eta_{q^{l_i^\infty}}^\infty(\alpha_i+\frac{m}{1-q})-\eta_{q^{l_i^\infty}}^\infty(\alpha_i)\right)}.
\]
The same argument applies to the values $(\beta_i)_{q,m}^0$ using
(\ref{eq:GK-0}). Then the result follows from 
Lemma~\ref{lemma:gamma-prod} and the fact that $\pi^{1-p}=(-p)$. 

2) Since by assumption $\zeta_N\in L$ we have $N \mid q-1$. Hence, the
sets $\{\alpha_1,\ldots,\alpha_n\}$ and $\{\beta_1,\ldots,\beta_n\}$
decompose into $n$-disjoint $q$-orbits of length $1$ (so
$l_i^\infty=l_i^0=1$). As $\iota(\chi_{\id{p}})=\Teich^{(q-1)/N}$,
$\iota(\chi_{\id{p}}^{N\alpha_i}(x))=\Teich^{(q-1)\alpha_i}(x)$ the
right hand side of (\ref{eq:Hq-algebricity}) matches
the~$\iota(H_{\id{p}}(\veca,\vecb|z))$. The algebricity statement
follows from \eqref{eq:gauss-comparison} and
Proposition~\ref{prop:props-hyperg-sum} (4).
\end{proof}
\begin{remark}
  If $\veca$ or $\vecb$ be are not $q$-stable then most likely
  $H_q(\veca,\vecb|z)$ (as defined in~\eqref{eq:hyperg-geom}) is
  \emph{not} algebraic.
\end{remark}
\begin{remark}
\label{p-hyperg-sum}
If $L$ is a finite extension of $K$ and $\id{p}$ is a prime ideal of
$L$ of norm $q$, then the parameters $\veca$ and $\vecb$ are
$q$-stable. Hence by Theorem~\ref{thm:algebricity} we may extend the
definition of $\id{p}$-hypergeometric sums to $\xi \in L$ as
 \[
H_{\id{p}}(\veca,\vecb|\z-spec):= \frac{1}{1-q}\sum_\varphi \prod_{i=1}^u\frac{g(\psi,\chi_{\id{q}}^{-N\alpha_i}\varphi,\id{q})}{g(\psi,\chi_{\id{q}}^{-N\alpha_i},\id{q})}\prod_{i=1}^v\frac{g(\psi^{-1},\chi_{\id{q}}^{N\beta_i}\overline{\varphi},\id{q})}{g(\psi^{-1},\chi_{\id{q}}^{N\beta_i},\id{q})}\,\varphi(\z-spec),
   \]
   where the sum runs over characters of $\FF_q^\times$. We expect
   this sum to be the trace of $\Frob_{\id{p}}$ of a motive defined
   over $L$ (see Conjecture~\ref{conjecture:field-of-defi}). 
\end{remark}

\subsection{Hypergeometric character sums}
\label{subsection:HMsums}
The result of the present section will be used to relate the number of
points of a variety to a finite hypergeometric sum. Although the
definition of Gauss sums depend on the choice of an additive
character, the products/quotients considered in the present section
will not, so we omit writing the dependence to easy notation.  Recall
the following well known definition.

\begin{definition} Let $\Char, \eta$ be two multiplicative characters on
  $\FF_q^\times$. The Jacobi sum attached to them is defined by
  \[
J(\Char,\eta) = \sum_{x \in \FF_q}\Char(x)\eta(1-x),
\]
where $\Char(0)=\eta(0)=0$.
\label{defi:Jacobi}
\end{definition}
\begin{lemma}
  Let $\Char, \eta$ be characters of $\FF_q$, with $\eta$ non-trivial. Then
  \[
J(\Char,\eta)  = \Char(-1) \frac{g(\Char^{-1}\eta^{-1})g(\Char)}{g(\eta^{-1})}.
    \]
\label{lemma:Jacobi}
  \end{lemma}

\begin{proof}
  If $\Char$ and $\Char \eta$ are non-trivial, then the left hand side
  is the usual Jacobi sum, and its value equals
  \[
J(\Char,\eta) = \frac{g(\Char)g(\eta)}{g(\Char\eta)}.
\]
For $\varphi$ a non-trivial character, the equality
$g(\varphi) = \varphi(-1)\frac{q}{g(\varphi^{-1})}$ applied to $\varphi = \Char \eta$ and $\eta$ implies that
\[
J(\Char,\eta) = \Char(-1)\frac{g(\Char)g(\Char^{-1}\eta^{-1})}{g(\eta^{-1})}.
\]
If $\Char = 1$, the left hand side equals $-1$ (because $\Char(0)=0$),
which equals the right hand side as well. At last, if $\Char \eta = 1$,
then the left hand side equals $J(\Char,\Char^{-1}) = -\Char(-1)$, which
also equals the value of the right hand side, since $g(1) = -1$.
\end{proof}

Let $\varepsilon_1,\eta_1,\ldots,\varepsilon_{n-1}, \eta_{n-1},\chi$
be characters of $\FF_q^\times$ (extended to $\FF_q$ by setting their
value at $0$ to be $0$). For $z \in \FF_q$ let
\begin{equation}
  \label{eq:H}
  H(z)=\sum_{x_1,\ldots,x_{n-1}} \prod_{i=1}^{n-1}\varepsilon_i(x_i)\eta_i(1-x_i)\chi^{-1}(1-zx_1\cdots x_{n-1}).
\end{equation}

\begin{theorem}
\label{thm:hypergeometric-char-sum}
Keep the previous notation, and set $\varepsilon_n=\chi$ and
$\varepsilon_n \eta_n = 1$. Let $\veca,\vecb$ be rational
numbers such that $\gen^{(q-1)\alpha_i}=\varepsilon_i$ and
$\gen^{(q-1)\beta_i}=\varepsilon_i\eta_i$ for $i=1,\ldots,n$.  Then
$H(z)$ equals
  \[
    H(z) = (-1)^{n-1}\left(\prod_{i=1}^{n-1} \varepsilon_i(-1)\right) \JacMot((\veca,-\vecb),(\veca-\vecb)) H_q(\veca,\vecb|z).
\]
\end{theorem}
\begin{proof}
Add one more variable $x_n$ to the definition of $H(z)$ defined by
\[
x_n:=zx_1\cdots x_{n-1}.
  \]
Then
\[
H(z) = \frac{1}{q-1} \sum_{\Char} \sum_{x_1,\ldots,x_n} \prod_{i=1}^{n-1} \varepsilon_i(x_i)\eta_i(1-x_i)\chi^{-1}(1-x_n) \Char(zx_n^{-1}x_1\cdots x_{n-1}),
\]
where the first sum ranges over all characters $\Char$ of $\FF_q^\times$ (we
are abusing notation, declaring $x_n^{-1}$ to be $0$ if $x_n$ equals
$0$).  Interchanging sum and product, we get
\begin{equation}
  \label{eq:H-2}
  H(z) = \frac{1}{q-1}\sum_\Char\left( \prod_{i=1}^{n-1}  \left(\sum_{x_i}\varepsilon_i\Char(x_i)\eta_i(1-x_i)\right)\cdot \sum_{x_n} \Char^{-1}(x_n)\chi^{-1}(1-x_n)\right)\Char(z).  
\end{equation}
Applying Lemma~\ref{lemma:Jacobi} we get that
\begin{eqnarray*}
  H(z) &=& \frac{1}{q-1}\sum_\Char \left(\prod_{i=1}^{n-1} \varepsilon_i(-1)\Char(-1)\frac{g(\varepsilon_i\Char)g(\varepsilon_i^{-1}\eta_i^{-1}\Char^{-1})}{g(\eta_i^{-1})}\right) \cdot \frac{g(\Char \chi)g(\Char^ {-1})}{g(\chi)} \Char(-1) \Char(z)\\
  &=&\prod_{i=1}^{n-1} \frac{\varepsilon_i(-1)g(\varepsilon_i)g(\varepsilon_i^{-1}\eta_i^{-1})}{g(\eta_i^{-1})} \left(\frac{(-1)}{(q-1)}\sum_\Char\left( \prod_{i=1}^n \frac{g(\varepsilon_i \Char)g(\varepsilon_i^{-1}\eta_i^{-1}\Char^{-1})}{g(\varepsilon_i)g(\varepsilon_i^{-1}\eta_i^{-1})}\right)\Char(-1)^n\Char(z)\right),
\end{eqnarray*}
as claimed (with the convention $\varepsilon_n=\chi$ and
$\varepsilon_n \eta_n = 1$), the extra $-1$ coming from the fact that $g(1)=-1$.
\end{proof}

\section{Zeta function of  Superelliptic curves}
\label{section:superelliptic}
For this section let $N$ be a positive integer and $L$ be a number
field or a local field containing the $N$-th roots of unity. Let
$f(x) \in L[x]$. Let $\C$ be the superelliptic curve with equation
\begin{equation}
  \label{eq:superelliptic}
  \C: y^N = f(x).  
\end{equation}
Assume that the curve $\C$ is irreducible and the model is
\emph{minimal} in the sense that all irreducible factors of $f(x)$
have degree less than $N$. Let $\widehat \C$ be a smooth projective
model.
The group $\mubb_N$ of $N$-th roots of unity in $L$ acts on the
$L$-rational points $\C(L)$ of $\C$ by
\[
\zeta\cdot (x,y) = (x,\zeta y), \qquad \zeta \in  \mubb_N.
\]

\subsection{The new part}\label{section:new-part}
Let $\ell$ be a prime number.
There is a natural decomposition
$$
 H^1_{\text{\'et}}(\Cnsp,\Q_\ell) =
 H^1_{\text{\'et}}(\Cnsp,\Q_\ell)^{\rm new}\oplus H,
$$
where $H$ is the sum of the contributions of the curves
 \[
\C_d:\quad y^d = f(x),
 \]
for $d$ a proper divisor of $N$.

Let $F:=\Q(\zeta_N)$ and let $\lambda$ be a prime ideal of $F$. Define
  \[
H^1_{\text{\'et}}(\Cnsp,F_\lambda)^{\rm new}:= H^1_{\text{\'et}}(\Cnsp,\Q_\ell)^{\rm new}\otimes_{\Q_\ell}F_\lambda.
\]
The action of the group $\mubb_N$ on the curve $\C$ induces a
decomposition
\begin{equation}
  \label{eq:etale-decomposition}
  H^1_{\text{\'et}}(\Cnsp,F_\lambda)^{\rm new} := 
\bigoplus_{\chi}H^1_{\text{\'et},\chi}(\Cnsp,F_\lambda),  
\end{equation}
where $H^1_{\text{\'et},\chi}(\Cnsp,F_\lambda)$ is the $\chi$-eigenspace
for $\chi: \mubb_N \to F_\lambda^\times$ a character of order $N$.  

Let $\Om$ be the discrete valuation ring consisting of the ring of
integers of $L$ (when $L$ is a local field) or a completion of its
ring of integers at a prime ideal (when $L$ is a global field). Let $\id{p}$
be its maximal ideal and $k$ its residue field, a finite field of
characteristic $p$ with $q$ elements (so $N \mid q-1$). Note that 
$\Frob_{\id{p}}$ preserves the
decomposition~\eqref{eq:etale-decomposition} as  $L$ contains $\mubb_N$.

For the rest of the section we make the following assumptions.
\begin{ass}
Keeping the previous notation, the
polynomial $f$, the integer $N$ and the residual characteristic $p$
satisfy the following properties:
\begin{enumerate}
\item $p \nmid N$,
  
\item $p \nmid \Disc\left(\frac{f}{\gcd(f,f')}\right)$,
\item The leading coefficient $c$ of $f$ is a unit in $\Om$,
\item The curve $\C$ is irreducible over $\overline{L}$.
\end{enumerate}
\label{assumption}  
\end{ass}
These assumptions assure that the reduction of
$\widehat \C$ modulo $\id{p}$ is smooth.
\begin{remark}
  \label{remark:assumption}
  For $\C$ the Euler curve \eqref{curva2} specialized at
  $z=\xi\in \PP^1(L)\setminus \{0,1,\infty\}$, a prime ideal $\id{p}$
  satisfies Assumption~\ref{assumption} precisely when
  $\id{p} \in S_g$ (see~\ref{section:HGM}). 
\end{remark}

\begin{definition}
  \label{defi:N-defi}
  Let $\omega$ be a character of $(\Om/\id{p})^\times$. The
  \emph{counting function} $\Count(\omega)$ is defined by
\begin{equation}
  \label{eq:N-defi}
  \Count(\omega) := \sum_{x \in \F_q} \omega(f(x)) +
  \begin{cases}
  \omega(c) & \text{ if }N \mid \deg(f),\\
  0 & \text{ otherwise}.  
  \end{cases}
  \end{equation}
\end{definition}
Fix a generator $\zeta$ of $\mubb_N \subset F \subset L$.  Let $k \in \FF_q$
be the reduction of $\zeta$ modulo $\id{p}$ and let
$\alpha \in \overline{\FF_q}$ be such that $\alpha^{q-1}= k$. Then
$\varepsilon:=\alpha^N \in \FF_q$ (since it is a root of $x^{q-1}-1$)
and generates $\FF_q^\times$.
\begin{lemma}
 \label{lemma:generator}
  The character $\chi_{\id{p}}$ (defined in \eqref{eq:gen-defi}) satisfies
  \begin{equation}
    \label{eq:gen-rel}
    \chi_{\id{p}}(\varepsilon) = \zeta.      
  \end{equation}
 \end{lemma}
\begin{proof} By definition
$  \chip(\varepsilon)= \chip(\alpha^N) \equiv (\alpha^N)^{(q-1)/N} = \alpha ^{q-1} = \kappa \equiv \zeta \pmod{\id{p}}.$
\end{proof}

For $\lambda$ a prime ideal of $F$ define the group morphism
$\chil : \mubb_N \to F_\lambda^\times$ to be the
restriction of the map $\iota:F \hookrightarrow F_\lambda$ to
$\mubb_N$.

\begin{theorem}
  Let $\lambda$ be a prime ideal of $F$. Let $\C$ be the
  superelliptic curve given by (\ref{eq:superelliptic}) for $N>1$ and
  suppose Assumption~\ref{assumption} holds.  Then the trace of
  $\Frob_{\id{q}}$ on $H_{et,\chil}^{1}(\Cnsp, F_\lambda)$ equals
  $-\iota_\lambda(\Count(\chip))$.
\label{thm:trace-equality}
\end{theorem}

\begin{proof}
It is not hard to verify that
  \begin{equation}
    \label{eq:ap-formula}
    \trace(\Frob)|_{H^1_{\text{\'et}}(\Cnsp,F_\lambda)^{\text{new}}} =
    -\sum_{\omega} \sum_{x \in k} \omega(f(x)) - 
    \begin{cases}
      \sum_{\omega}\omega(c) & \text{ if }N\mid \deg(f),\\
      0 & \text{ otherwise},
      \end{cases}
    \end{equation}
    where $\omega$ runs over characters of $k^\times$ of order $N$ and
    $c$ denotes the leading coefficient of $f$.  For
    $\chi:\mubb_N \to F_\lambda^\times$ a character of order $N$, let
    $A_\chi$ be the matrix of $\Frob_{\id{p}}$ acting on $\HetL$.  By
    \eqref{eq:ap-formula}
  \begin{align*}
     \sum_{\chi} \text{Tr}(A_\chi)=-\sum_{\omega} \iota_\lambda(\Count(\omega)).
  \end{align*}
  Let $X/\F_q$ be a non-singular variety and let
  $\sigma \in \Aut_{\F_q}(X)$. By Lefschetz's trace formula
  \begin{equation}
    \label{eq_2}
    \# \{x \in X(\overline{\F_q}): \Frob_q(\sigma(x)) = x\} = \sum_{i=1}^{2 \dim(X)} (-1)^i \text{Tr}(\Frob_q \circ \ \sigma )|_{H_{\text{\'et}}^i(X,F_\lambda)}.
  \end{equation}

  Let $\sigma_i$ be the automorphism of $\Cnsp$ defined by
  $(x,y) \mapsto (x, \zeta^i y)$. Then \eqref{eq_2} with
  $\sigma = \sigma_i$ gives
  \begin{equation}
  \label{eq:points-twist}
    \# \{(x,y) \in \Cnsp(\F_q): x \in \F_q, y^{q-1} = k^{-iq}\} = 1 +
    q - \sum_{\chi} \chi(\zeta)^i\operatorname{Tr}(A_\chi) + \text{``old contribution''},
  \end{equation}
where \emph{``old contribution''} comes from the curves $\C_d$ for
a proper divisors $d\mid N$.

Since $k \in \FF_q$, $k^q = k$ and
  \begin{equation*}
    \{(x,y) \in \Cnsp(\F_q): x \in \F_q, y^{q-1} = k^{-iq}\} =  \{(x,y) \in \Cnsp(\F_q): x \in \F_q, y^{q-1} = k^{-i}\}.
  \end{equation*}
  Recall that $\alpha \in \overline{\F_q}$
  satisfies $\alpha^{q-1}=k$ and $\varepsilon := \alpha^N$. Then the
  map $y \leftrightarrow \tilde{y}:=y\alpha^i$ gives a bijection
  between the sets
  \begin{equation}
    \label{eq_3}
    \{(x,y) \in \Cnsp(\F_q) \ \text{fixed by } \Frob_q \circ\  \sigma_i \} \leftrightarrow \{(x,\tilde{y}) \in \F_q^2: \varepsilon^{-i} \tilde{y}^N = f(x)\},
  \end{equation}
  Indeed,
  \[
    y^{q-1} = k^{-i} \Leftrightarrow
    \frac{\tilde{y}^{q-1}}{(\alpha^{q-1})^i} = k^{-i} \Leftrightarrow
    \tilde{y}^q = \tilde{y}.
\]
The elements of the right hand side of \eqref{eq_3} are the
$\F_q$-points of the twisted curve
  \begin{align*}
    \bfC_i:  y^N = \varepsilon^{i}f(x).
  \end{align*}
  Then the left hand side of \eqref{eq:points-twist} equals
  \begin{align*}
    \# \Cnsp_i(\F_q) = 1 + q + \sum_{\omega}\omega
    (\varepsilon^i) \Count(\omega) + \text{``old contribution''}
  \end{align*}
so for all $i$
  \[
    \sum_{\chi} \chi(\zeta)^i \trace(A_\chi) = -\sum_{\omega} \iota_\lambda(\omega(\varepsilon))^i\iota_\lambda(\Count(\omega)).
  \]
  It follows that
  $\trace(A_{\chil}) = -\iota_\lambda(\Count(\omega))$ for
  $\omega$ the character of $\FF_q$ satisfying
  $\iota_\lambda(\omega(\varepsilon))=\chil(\zeta) =
  \iota_\lambda(\zeta)$, which by Lemma~\ref{lemma:generator}
  corresponds to $\omega=\chip$.
\end{proof}

\section{Rank two Hypergeometric Motives}
\label{section:hgm-2}
The main result of \cite{BCM} provides an explicit formula relating
the function $H_p(\vec{\alpha},\vec{\beta}|\z-spec)$ to the number of points
of a non-singular projective variety $V$ when the parameters are
rational. Such relations allow to realize the motive
$\HGM(\vec{\alpha},\vec{\beta}|\z-spec)$ in $V$. A nice instance of their
result is the following result (see for example~\cite{MR1407498}).

\begin{theorem}
  \label{thm:Ono}
  Let $p$ be an odd prime power and $\z-spec \in \FF_p$ and
  $\z-spec \neq 0,1$. Let $E_{\z-spec}$ be the Legendre elliptic curve
with affine equation
  \[
    E_{\z-spec}: y^2=x(x-1)(x-\z-spec)
  \]
  Then the set of $\F_p$-rational points (including the one at infinity) equals
  \[
|E_{\z-spec}(\FF_p)| = p+1 - (-1)^{(p-1)/2}H_p((1/2,1/2),(1,1)|\z-spec).
    \]
\end{theorem}
Equivalently,
\begin{equation}
  \label{eq:Ono}
  H_p((1/2,1/2),(1,1)|\z-spec)= (-1)^{(p-1)/2} a_p(E_{\z-spec}),  
\end{equation}
where $a_p(E_{\z-spec})$ is the trace of the Frobenius endomorphism
acting on $E_{\z-spec}$.

Our definition of the geometric realization of the hypergeometric motive
(Definition~\ref{defi:HGM}) yields an analogue of~\eqref{eq:Ono}. The
hypergeometric motive 
appears in $H^1$ of the associated Euler
curve~(Definition~\ref{def:Euler-curve}) tensored with a Jacobi motive
and the explicit relation between the corresponding trace of Frobenius
and the finite hypergeometric sum is a particular instance of
Theorem~\ref{thm:Trace-match} below.
\begin{definition}
  \label{def:Euler-curve}
  Let $(a,b)$,$(c,d)$ be rational numbers, and let $N$ be their least
  common denominator. Define the quantities
\begin{equation}
  \label{eq:ABCD-parameters}
A=(d-b)N, \qquad B=(b+1-c)N, \qquad C=(1+a-d)N, \qquad D=(d-1)N.  
\end{equation}
Assume $(A,B,C,D)$ are non-negative. Define the  Euler curve attached
to~$(a,b),(c,d)$ as the affine curve over $\Q(z)$ with equation  
\begin{align}
\label{curva2}
\C:y^N = x^A(1-x)^B(1-zx)^C z^D.
\end{align}
\end{definition}

\begin{remark}
  \label{rem:Euler-translation} The Euler curve depends only on
  $(A,B,C,D)\bmod N$ so it maybe defined for any rationals
  $(a,b),(c,d)$ with common denominator $N$.
\end{remark}

\begin{definition}
  \label{defi:condition}
  Let $(a,b), (c,d)$ be a pair of generic rationals parameters. The
  parameters satisfy condition $\irr$ if the Euler curve~(\ref{curva2})
  is irreducible over $\overline{\Q(z)}$.
\end{definition}
\begin{lemma}
  Let $(a,b)$,$(c,d)$ be generic rational numbers and let $N$ be their least
  common denominator.
Then condition $\irr$ holds if and only if
  \begin{equation}
    \label{eq:dag}
N= \lcm\{\den(d-b),\den(b-c),\den(a-d)\}. 
  \end{equation}
  \label{lemma:equiv-irr}
\end{lemma}

\begin{proof}
  By making the change of variables (over $\overline{\Q(z)}$)
  $y = \sqrt[N]{z^D}y$, it is enough to study the curve
  \[
   y^N = x^A(1-x)^B(1-zx)^C,
 \]
 which is irreducible if and only of
 $\gcd(A,B,C,N)=1 = \gcd((d-b)N,(b-c)N,(a-d)N,N)$. The later equality holds if and
 only if
 \[
    \lcm\{\den(d-b),\den(b-c),\den(a-d)\}=N=\lcm\{\den(a),\den(b),\den(c),\den(d)\}.
  \]
\end{proof}

\begin{remark}
  If $r$ denotes the quotient of $N$ by
  $\lcm\{\den(d-b),\den(b-c),\den(a-d)\}$, then the curve $\C$
  decomposes as the union of $r$ irreducible components defined over
  $\Q(\zeta_r)[\sqrt[r]{z}]$. This phenomena does not appear 
 when $d=1$ since then condition $\irr$ always holds.
\label{remark:reducible}
\end{remark}
\begin{example}
  \label{example:reducible}
Continuing with Example~\ref{example:1}, let $a=1/8, b=7/8, c=3/8,
d=5/8$ so $N=8$. The Euler curve is defined by  
  \[
\C: \quad y^8 = x^6(1-x)^4(1-zx)^4z^{5}.
\]
Over $\Q(w)$, where $w^2=z$, it becomes the union of the curves
\[
\C_{\pm}: \quad y^4=\pm wx^3(1-x)^2(1-w^2x)^2.
  \]
These curves have genus two with hyperelliptic model
$$
y^2=wx(x^2\mp w)(x^2\mp w^{-1})
$$
The map
$$
\iota: \quad (x,y)\mapsto (x^{-1},yx^{-3})
$$
yields an involution of these curves and it is not to hard to see that
the quotient of $\C_{\pm}$ by $\iota$ equals
$$E_{\pm}: y^2=x^3-4wx^2\mp w(w-1)^2x.$$
For $d \in \ZZ$, let $\theta_d$ be the quadratic character of the
extension $\Q(\sqrt{d})/\Q$.  It can be verified (using the involution
$(x,y) \to (x^{-1},-yx^{-3})$) that actually
$$
\Jac(\C_\pm)\simeq E_{\pm}(w) \oplus E_{\pm}(w)\otimes \theta_{-1}.
$$
(See Examples~\ref{ex:reducible}
and~\ref{example:genericity-condition} for more on this curve.)
\begin{remark}
The point $P:=(0,0)$ is a $2$-torsion point on $E_\pm(w)$ and 
$$
E_{\pm}(w)/\langle P\rangle \simeq E_{\mp}(w) \otimes \theta_{-2}.
$$
Over the cyclotomic extension $\Q(\zeta_8)$
$$
\Jac(\C_\pm)\simeq E_{+}(w) \oplus E_{+}(w).
$$
\end{remark}
If we now take $w=3$, so $z=9$ as in Example~\ref{example:1}, the
curve $E_+(w)$ specializes to the elliptic curve
$$
E: \quad y^2=x^3-60x+176,
$$
of conductor $576$ and CM by $\Q(\sqrt{-12})$. We can quickly verify
for small primes  $p$ split in $\Q(\sqrt 2)$ that the trace of
Frobenius $\Frob_p$ on $H^1(E)$ and on $\calH_9$ indeed agree (the
latter may be computed in $\Z_p$ as
\begin{verbatim}
? e=ellinit([0,0,0,-60,176]);
? forprime(p=5,100,if(kronecker(2,p)==1,
    print(p,"  ",ellap(e,p),"  ",recognizep(hgm(9,[1/8,-1/8],[3/8,-3/8],p)*p))))

7  -4  -4
17  0  0
23  0  0
31  -4  -4
41  0  0
47  0  0
71  0  0
73  -10  -10
79  -4  -4
89  0  0
97  14  14
\end{verbatim}
with the GP code already mentioned).
\end{example}

\subsection{Hypergeometric motive definition}
In this section we give a definition of a hypergeometric motive for
general parameters in the case of rank two. For this we will use the
Euler curve. Many issues arise, however, if the Euler curve is reducible.

We proceed as follows. First, we consider the case where the Euler
curve is irreducible, see Definition~\ref{defi:HGM}, and $(a,b),(c,d)$
are generic and rational but otherwise arbitrary.

Second, we give an alternative definition, see Definition~\ref{defi:general-HGM},
which includes the case where the Euler curve is reducible, as a twist
of the motive attached to the (irreducible) Euler curve for parameters
$(a-d,b-d),(c-d,1)$ (in analogy with formula~\eqref{eq:ui-def}).

The fact that both definitions essentially agree when the Euler curve
is irreducible is not entirely obvious; we prove it in
Theorem~\ref{thm:isogenous}. Neither is it clear why the definitions
are independent of the ordering of $a,b,c,d$ since this is not true
for either the corresponding Euler curve or Jacobi motive. However, we
will prove that this is the case in Proposition~\ref{prop:symmetries}.

We start by studying the analogue of the character~$(-1)^{(p-1)/2}$
in~\eqref{eq:Ono}. 

\begin{definition}
  \label{defi:twist}
  Let $N$ be a positive integer, and let $\twist_N$ be the character
  of $\Gal(\overline \Q/\Q(\zeta_N))$ (of order dividing $2$)
  corresponding to the abelian extension $\Q(\zeta_{2N})/\Q(\zeta_N)$.
\end{definition}

\begin{lemma}
  \label{lemma:char-values}
  Let $\id{p}$ be a prime ideal of $F=\Q(\zeta_N)$ not dividing $2N$. Then 
\begin{equation}
  \label{eq:quad-char}
  \twist_N(\id{p})=\chip(-1)=(-1)^{(\normid{p}-1)/N},
\end{equation}
where $\chip$ is defined in~\eqref{eq:gen-defi}.
\end{lemma}

\begin{proof}
  The second equality follows from the fact that
  $\chip(-1)\equiv (-1)^{(\normid{p}-1)/N} \bmod\id{p}$ by definition. Let
  $t=v_2(N) \ge0$. Let $\normid{p}=p^r$ for $p$ a rational prime
  number, so $r$ is the order of $p$ modulo $N$. Then
  \[
  (-1)^{(\normid{p}-1)/N}=(-1)^{(p^r-1)/2^t}.
\]
Then the right hand side equals $1$ if $p^r \equiv 1 \pmod{2^{t+1}}$
and $-1$ otherwise. By the Chinese reminder theorem it equals $1$ if
$p^r \equiv 1 \pmod{2N}$ and $-1$ otherwise; note that
$p^r \equiv 1 \pmod{2N}$ precisely when $\id{p}$ splits in
$\Q(\zeta_{2N})/F$.
\end{proof}

\begin{corollary}
  Suppose that $t = v_2(N) \ge 1$. Let $\kappa$ be a primitive
  Dirichlet character of conductor $2^{t+1}$ thought as a character of
  $\Gal(\overline{\Q}/\Q)$. Then the quadratic character $\twist$ equals the
  restriction of $\kappa$ to $\Gal(\overline{\Q}/{F})$.
 \label{lemma:char-extension}
\end{corollary}

\begin{proof}
  The fixed field of the character $\kappa$ corresponds to the Galois
  extension $\Q(\zeta_{2^{t+1}})$. Then the
  restriction of $\kappa$ to $\Gal(\overline{\Q}/F)$ has as fixed
  field $\Q(\zeta_{2N})=\Q(\zeta_N) \cap \Q(\zeta_{2^{t+1}})$.
\end{proof}

Let $\Euler$ denote the $\chi_0$ isotypical component of the de Rham
cohomology group~$H^1(\Cnsp,F)=H^1(\Cnsp,\Q)\otimes_{\Q}F$, where
$\chi_0:\mubb\rightarrow F^\times$ is the identity character. It is
known to be of rank two (see \cite{MR2005278} formula (16)).

\begin{definition}
  \label{defi:HGM}
  Let $(a,b),(c,d)$ be a pair of generic rational numbers satisfying
  condition $\irr$ and let $N$ be their least common denominator. The
  hypergeometric motive with parameters $(a,b),(c,d)$ is defined as 
\begin{equation}
  \label{eq:motive-definition}
  \hgm := \Euler \otimes \JacMot((-a,-b,c,d),(c-b,d-a))^{-1} \twist_N^{(d-b)N}.
\end{equation}
\end{definition}

\begin{example}
  \label{rem:Legendre}  
  Let $\veca=(1/2,1/2)$ and $\vecb=(1,1)$ as in Example~\ref{example:Legendre}. Formula
  \eqref{eq:ABCD-parameters} gives the values $A=1$, $B=1$,
  $C=1$, $D=0$ and $N=2$. Then Euler's curve corresponds to
  Legendre's elliptic curve
  \[
  E_z: y^2 = x(1-x)(1-zx).
\]
The Jacobi motive $\JacMot((-1/2,-1/2,1,1)(1/2,1/2))$ is trivial. Then
the hypergeometric motive $\HGM((1/2,1/2),(1,1)|z)$ corresponds to the
quadratic twist by the character $\kro{-4}{\cdot}$ of Legendre's
elliptic curve (as expected from~\eqref{eq:Ono}).
\end{example}

Formula~\eqref{eq:motive-definition} involves three different motives,
with a priory different fields of definition (see
Example~\ref{example:Jacobi}).  This is the reason why we cannot prove
a full version of Conjecture~\ref{conjecture:field-of-defi} (but we
can under some extra hypotheses, for example, conditions that make the
Jacobi motive to be defined over~$\Q$). To prove cases not covered by
our assumptions, one probably needs to study other geometric varieties
where the HGM appears. Let us illustrate the situation with some
examples.

\begin{example}
  Let $a=\frac{1}{15}, b=\frac{4}{15}, c=\frac{1}{8}$ and
  $d=\frac{3}{8}$. The Jacobi motive for these parameters
  is~$\JacMot(\pmb \theta)$, where
  \[
  \pmb \theta=\left\langle \frac{14}{15} \right\rangle + \left\langle \frac{11}{15}\right\rangle + \left\langle \frac{1}{8} \right\rangle + \left\langle \frac{3}{8}\right\rangle - \left\langle \frac{103}{120}\right\rangle - \left\langle \frac{37}{120}\right\rangle.
 \]
 Since $19$ is the only element in $\left(\Z/120\right)^\times$ that
 fixes (under multiplication) the sets
 $\{\frac{14}{15},\frac{11}{15},\frac{1}{8},\frac{3}{8}\}$ and
 $\{\frac{103}{120},\frac{37}{120}\}$ (in $\Q/\Z$), the Jacobi motive
 is defined over $F_0:=\Q(\zeta_{120})^{\sigma_{19}}$. The base field
 $K$  of~$\HGM((\frac{1}{15},\frac{4}{15}), (\frac{1}{8},\frac{3}{8})|z)$,
however, is that fixed by $H=\langle 49,91\rangle$. In
 particular, $K \subsetneq F_0$.

 On the other hand, if we take
 $a=\frac{1}{8}, b=\frac{7}{8}, c=\frac{3}{8}$ and $d=\frac{5}{8}$
 then the Jacobi motive is defined over~$\Q$ (since multiplication by
 any odd integer fixes the sets
 $\{\frac{1}{8},\frac{3}{8},\frac{5}{8},\frac{7}{8}\}$ and
 $\{\frac{1}{2}\}$), while the base field
 of the corresponding hypergeometric motive is~$\Q(\sqrt{2})$.
\label{example:Jacobi}
\end{example}

\begin{definition}
  Let $\alpha=\frac{a}{N}$ be a rational number, with $a,N$ coprime
  integers. Define 
  $$
\eta_\alpha:\Gal(\overline{\Q(z)}/\Q(z,\zeta_N)) \to
\overline{\QQ}^\times,
$$
to be the character that factors through
$\Gal(\QQ(\sqrt[N]{z},\zeta_N)/\QQ(z,\zeta_N))$ and whose value at
$\sigma$ equals
  \[
\eta_{\alpha}(\sigma) = \left(\frac{\sigma(\sqrt[N]{z})}{\sqrt[N]{z}}\right)^a.
    \]
\end{definition}

\begin{remark}
  The character $\eta_\alpha$ is a realization of the hypergeometric motive
  $\HGM(\alpha,1|1-z)$ studied in \S \ref{section:rank1}.
\end{remark}

Let $\z-spec \in F=\Q(\zeta_N)$, $\z-spec \neq 0$ and consider the
specialization of $\eta_{\alpha}$ at $z=\z-spec$. Let $\id{p}$ be a
prime ideal of $F$ of norm $q$, not dividing $N$ and such that
$v_{\id{p}}(\z-spec)=0$. Let $\widetilde{\id{p}}$ be a prime ideal of
$\Q(\sqrt[N]{\z-spec},\zeta_N)$ dividing $\id{p}$.  It follows from
its definition that
\[
  \Frob_{\widetilde{\id{p}}}(\sqrt[N]{\z-spec}) \equiv (\sqrt[N]{\z-spec})^{q}
  \pmod{\widetilde{\id{p}}},
\]
hence
\[
  \frac{\Frob_{\widetilde{\id{p}}}(\sqrt[N]{\z-spec})}{\sqrt[N]{\z-spec}} \equiv
  \z-spec^{\frac{q-1}{N}} \pmod{\id{p}}.
\]

This implies that the specialization of $\eta_{\alpha}$ at $\z-spec$ satisfies
\begin{equation}
  \label{eq:eta-weil}
\eta_{\alpha}(\Frob_{\id{p}})= \chip(\z-spec)^a,  
\end{equation}
where $\chip$ is the character defined in \eqref{eq:gen-defi}.

\begin{definition}
  Let $(a,b),(c,d)$ be generic rational numbers, and let $N$ be their
  least common denominator. Set
  $F=\Q(\zeta_N)$. The hypergeometric motive $\hgm$ is defined by
  \begin{equation}
    \label{eq:motive-general-definition}
   \HGM((a-d,b-d),(c-d,1)|z)\otimes
   \JacMot((-a,-b,c,d),(d-a,d-b,c-d))^{-1} \eta_d(z),     
  \end{equation}
  where the hypergeometric motive $\HGM((a-d,b-d),(c-d,1);z)$ (base
  changed to $F(z)$ if necessary) is defined in~\ref{defi:HGM}.
\label{defi:general-HGM}
\end{definition}

This definition of~$\hgm$ has the disadvantage of making more
difficult to understand its descent to smaller fields.  In this
regard, it is not clear whether there is a more suitable definition of
the motive when condition $\irr$ is not
satisfied.

\begin{example} \label{ex:reducible} Continuing with Example
  \ref{example:1}, take $a=1/8, b=7/8, c=3/8, d=5/8$, so $N=8$. The
  values $\tilde{a}=a-d=1/2$, $\tilde{b}=b-d=1/4$ and
  $\tilde{c}=c-d=3/4$ have denominator $4$ so
  $\HGM((a-d,b-d),(c-d,1);z)$ has base
  field~$\Q(i)\subsetneq F:=\Q(\zeta_8)$.  Also, the curve $\C$ is
  reducible (as described in detail in
  Example~\ref{example:reducible}). The Euler curve of  parameters
  $(\tilde{a},\tilde{b}),(\tilde{c},1)$ has the affine   equation
  \[
    \widetilde{\C}:y^4=x^3(1-x)^2(1-zx)^2.
  \]
  It is an irreducible curve isomorphic to a twist of an irreducible
  component of $\C$. 
  Note that $\HGM((1/2,1/4),(3/4,1)|z)$ is defined over its base
  field~$\Q(i)$ whereas $\HGM((1/8,7/8),(3/8,5/8)|z)$ has base
  field~$\Q(\sqrt 2)$ (see Example~\ref{example:base-field-switch} and
  Corollary~\ref{coro:extension-motive}).
\end{example}

\begin{theorem}
  \label{thm:Trace-match}
  Let $(a,b),(c,d)$ be generic rational parameters and let $N$ be
  their least common denominator. Let $L$ be a finite extension of
  $F$. Let $\z-spec \in L$ with $\z-spec \neq 0,1$. Let $\id{p}$ be a
  prime ideal of $L$ satisfying that $\id{p} \nmid N$ and
  $v_{\id{p}}(\z-spec(\z-spec-1))=0$. Then the trace of the Frobenius
  automorphism $\Frob_{\id{p}}$ acting on $\hgms$ equals
  $H_{\id{p}}((a,b),(c,d)|\z-spec)$.
\end{theorem}

Before proving the statement, recall that we gave two different
definitions of $\hgms$ depending on whether $\irr$ holds or not. The
theorem holds for both definitions.
We need the following auxiliary result.

\begin{lemma}
  \label{lemma:hypergeom-fin-rel}
  Let $(a,b),(c,d)$ be rational generic parameters and let $N$ be
  their least common denominator. Let $\id{p}$ be a prime ideal of $L$
  not dividing $N$ and such that $v_{\id{p}}(\z-spec(\z-spec-1))=0$. Then
  \[
    H_{\id{p}}((a-d,b-d),(c-d,1)|\z-spec) = \gent(\z-spec)^{-dN}\JacMot((-a,-b,c,d),(d-a,d-b,c-d)) H_{\id{p}}((a,b),(c,d)|\z-spec).    
\]
\end{lemma}

\begin{proof} Let $q = \normid{p}$ and let $\gen$ be a generator of
  the character group $\widehat{\FF_q^\times}$ so that
  $\gen^{(q-1)/N}=\gent^{-1}$. Then \cite[Theorem 3.4]{BCM} implies
  that
  \[
    H_{\id{p}}((a-d,b-d),(c-d,1)|\z-spec) = \gen(\z-spec)^{d(q-1)}\frac{g(\gent^{-aN})g(\gent^{cN})}{g(\gent^{(d-a)N})g(\gent^{(c-d)N})} \frac{g(\gent^{-bN})g(\gent^{dN})}{g(\gent^{(d-b)N})g(1)} H_{\id{p}}((a,b),(c,d)|\z-spec).
  \]
  From its definition (Definition~\ref{def:jacobi-motive}) the middle
  factor equals the stated Jacobi motive.  As a side remark, the first
  statement of \cite[Theorem 3.4]{BCM} is correct, but the second one
  is not. The numerator should be the denominator and vice-versa
  (since the function $S_q$ is the product of $H_q$ with a Gauss sum
  involving the same parameters, as follows from (3.1) of loc. cit.)
\end{proof}

\begin{proof}[Proof of Theorem~\ref{thm:Trace-match}]
  Suppose that $\irr$ holds and let $\C$ denote the Euler curve as
  defined in \eqref{curva2}.  The hypothesis on $\id{p}$
  implies (by Remark~\ref{remark:assumption}) that
  Assumption~\ref{assumption} holds.  The genericity condition on the
  parameters imply that $a-c, b-c, a-d, b-d \not \in \ZZ$, so
  $N\nmid A$, $N\nmid B$, $N\nmid C$ and
  $N\nmid \deg(f(x))=A+B+C=(a-c)N$. By
  Theorem~\ref{thm:trace-equality} the trace of Frobenius on $\Het$
  equals
  $$-\Count(\chip;\z-spec) = -\chip(\z-spec)^D \sum_{x \in
    \FF_q}\chip(x^A(1-x)^B(1-\z-specx)^C),$$ 
  where $A=(d-b)N$, $B=(b-c)N$, $C=(a-d)N$ and $D=dN$. Set
  $\alpha_1=(d-b)N$, $\alpha_2=(d-a)N$, $\beta_1=(d-c)N$ and
  $\beta_2=0$. Then Theorem~\ref{thm:hypergeometric-char-sum} implies
  that
  \begin{equation}
    \label{eq:12}
    -\Count(\chip;\z-spec)=\chip(\z-spec)^D\chip(-1)^A\JacMot((\alpha_1,-\beta_1),(\alpha_1-\beta_1))H_q(\veca,\vecb|\z-spec).
  \end{equation}
  By definition (\ref{defi:finite-hgm}) the value
  $H_{\id{p}}(\veca,\vecb|\z-spec)$ equals the value $H_q(\veca,\vecb|\z-spec)$
  choosing a generator $\gen$ satisfying $\gen^{(q-1)/N}=\chip^{-1}$,
  hence
\[
H_q(\veca,\vecb|\z-spec)=  H_{\id{p}}(((a-d)N,(b-d)N),((c-d)N,1)|\z-spec).
\]
Then~\eqref{eq:12} equals
  \[
\gent(\z-spec)^{dN}\gent^{(d-b)N}(-1)\JacMot((d-b,c-d),(c-b)) H_{\id{p}}\left((a-d,b-d),(c-d,1)|\z-spec\right).
\]
Using Lemma~\ref{lemma:hypergeom-fin-rel} this value equals
\[
\gent^{(d-b)N}(-1)\JacMot((-a,-b,c,d),(c-b,d-a)) H_{\id{p}}\left((a,b),(c,d)|\z-spec\right),
\]
which matches the motive of Definition~\ref{defi:HGM} (recall that $\twist_N(\id{p})=\gent(-1)$ by Lemma~\ref{lemma:char-values}). 

To prove that $\hgms$ also matches
\eqref{eq:motive-general-definition}, we apply the proven result to the
parameters $(a-d,b-d),(c-d,0)$ and are led to prove the equality
\[
H_q((a-d,b-d),(c-d,1)|\z-spec) \JacMot((-a,-b,c,d),(d-a,d-b,c-d))^{-1}\eta_d(\z-spec)=H_q((a,b),(c,d)|\z-spec),
  \]
which follows from Lemma~\ref{lemma:hypergeom-fin-rel} since $\eta_d(\Frob_{\id{p}})(\z-spec) = \gent(\z-spec)^{dN}$.
\end{proof}

\begin{lemma}
  Let $L$ be a finite extension of $F$ and let $\z-spec \in L$, with
  $\z-spec \neq 0,1$. The coefficient field of the motive $\hgms$
  (i.e. the field extension of $\Q$ generated by the trace of
  Frobenius' elements) is contained in $K$.
  \label{lemma:coefficient field}
\end{lemma}
\begin{proof}
  Follows from Theorem~\ref{thm:Trace-match} and Proposition~\ref{prop:props-hyperg-sum}, part (3).
\end{proof}

As a corollary we get the following slightly softer version of
Conjecture~\ref{conjecture:field-of-defi} (basically where the
field~$K$ is replaced by $F$).

\begin{corollary}
  \label{coro:conjecture}
  Let $\veca=(a,b), \vecb=(c,d)$ be rational generic parameters with
  denominator $N$. Let $F=\Q(\zeta_N)$ and $K=F^H$, with $H$ as in
  Definition~\ref{def:H}. Then
  \begin{enumerate}
  \item The  motive $\HGM(\veca,\vecb|z)$ has a realization $\calH(z)$ 
over $F(z)$.
\item The representation of the motive restricted to
  $\Gal(\overline{K(z)}/\overline{F}(z))$ is isomorphic to the
  geometric Galois representation~\eqref{eq:geometric-rep}.

\item Let $L$ be a finite extension of $F$.  The specialization
  $\calH(\z-spec)$ for generic $\z-spec\in L$ has coefficient field
  $K$.
\item For primes $\id{p}$ of $L$ of good reduction the trace of
  Frobenius on $\calH(\z-spec)$ is given by the finite hypergeometric
  sum $H_{\id{p}}(\veca,\vecb\,|\,\z-spec)$.
\end{enumerate}
\end{corollary}

\begin{proof}
  The first two statements follow from Definitions~\ref{defi:HGM} and \ref{defi:general-HGM} and
  Theorem~\ref{thm:motive-function-field}; the third statement follows
  from the last lemma, while the last one follows from Theorem~\ref{thm:Trace-match}.
\end{proof}

\begin{remark} Let $(a,b),(c,d)$ be rational generic parameters with
  denominator $N$. Let $\z-spec \in \Q$ and let $\id{p}$ be a prime
  ideal of $F$ not dividing $N$, but dividing $\z-spec$. Suppose that
  the monodromy at~$0$ is finite of order $m$ and that
  $m\mid v_{\id{p}}(\z-spec)$.  Then
  Theorem~\ref{thm:unramified-primes} implies that the motive $\hgms$
  is unramified at $\id{p}$, so the trace of the action of
  $\Frob_{\id{p}}$ on our motive is well defined
  but is not given by~$H_{\id{p}}((a,b),(c,d)|\z-spec)$. Nevertheless,
  there is an explicit formula, see Appendix~\ref{section:frob-trace}.

\label{remark:appendix}
\end{remark}

\subsection{The Galois representation}

The motive $\Euler$ has attached a family of $2$-dimensional Galois
representations indexed by prime ideals $\id{p}$ of $F$ (as defined by
Serre in \cite{MR1484415}). Denote by 
\begin{equation}
  \label{eq:Galois-rep}
\rho_{\hgm,\id{p}}: \Gal(\overline{\Q(z)}/{F(z)}) \to \GL_2(F_{\id{p}}),
\end{equation}
the representation of the hypergeometric motive $\hgm$.

\begin{theorem}
  \label{thm:motive-function-field}
  The Galois representation $\rho_{\hgm,\id{p}}$ restricted to
  $\Gal(\overline{\Q(z)}/{\overline{\Q}(z)})$ matches the representation given in
  Corollary~\ref{coro:extension}.
\end{theorem}

\begin{proof}
  The Galois representation attached to the Jacobi motive does not
  depend on $z$ and its restriction to $\Gal(\overline{\Q(z)}/{\overline{\Q}(z)})$ is
  trivial. Then to prove the result it is enough to prove that the
  image under $\rho_{\hgm,\id{p}}$ of the inertia group at $z=0$
  (respectively $z=1$ and $z=\infty$) of $\overline{\Q}(z)$ is
  conjugate to the matrix $M_0$ (respectively $M_1$ and $M_\infty$),
  which is proved in Theorems~\ref{thm:HGM-0} and~\ref{thm:HGM-1-infty}.
\end{proof}

\begin{definition}
  \label{defi:isogenous} The hypergeometric motives $\hgm$ and
  $\HGM((a',b'),(c',d')|z)$ are called \emph{isogenous}, denoted
  by $\hgm \simeq \HGM((a',b'),(c',d')|z)$, if their Galois
  representations are isomorphic.
\end{definition}

\begin{lemma}
  Let $\rho_i:\Gal(\overline{\Q(z)}/F(z)) \to \GL_2(\overline{\Q_p})$
  for $i=1,2$ be two irreducible continuous representations satisfying
  that their restrictions to the subgroup
  $\Gal(\overline{\Q(z)}/\overline{\Q}(z))$ are irreducible and
  isomorphic.  Then there exists
  $\chi:\Gal(\overline{\Q}/F) \to \overline{\Q_p}^\times$ such that
  $\rho_1 \simeq \rho_2 \otimes \chi$.
\label{lemma:isom-condition}
\end{lemma}

\begin{proof}
  See \cite[Corollary 2.5]{Pac}.
\end{proof}

\begin{theorem}
  \label{thm:isogenous}
Let $(a,b), (c,d)$ be a pair of generic rational numbers.
Then the following  properties hold:
\begin{enumerate}
\item
(Twist) With the above notation
$$
\hgm \simeq    \HGM((a-d,b-d),(c-d,1)|z)\otimes
   \JacMot((-a,-b,c,d),(d-a,d-b,c-d))^{-1} \eta_d(z)
$$
\item
\label{prop:symmetries}
(Permutation)  Let $(a,b),(c,d)$ be generic parameters. Then
  \[
    \HGM((a,b),(c,d)|z) \simeq \HGM((b,a),(c,d)|z) \simeq \HGM((a,b),(d,c)|z) \simeq \HGM((b,a),(d,c)|z).
    \]
\item
  \label{prop:change-parmeters-motive}
(Inverse)  The motives $\HGM((a,b),(c,d)|z)$ and
  $\HGM((-c,-d),(-a,-b)|z^{-1})$ are isomorphic.
\end{enumerate}
\end{theorem}

\begin{proof}
(1)  If the Euler curve is reducible (i.e.,~$\irr$ does not hold) the claim
  follows from definition. Otherwise, by
  Theorem~\ref{thm:motive-function-field} the restriction of both
  Galois representations to $\Gal(\overline{\Q(z)}/\overline{\Q}(z))$
  is isomorphic to their respectively monodromy representations. Using
  the description of the monodromy matrices given in
  \ref{prop:monodromy-eigenvalues} it is easy to verify that the two
  monodromy representations are indeed isomorphic. The genericity
  hypothesis on the parameters implies that the monodromy
  representations are irreducible (by
  Proposition~\ref{prop:irreducible}) so the lemma implies that the
  two representations are isomorphic up to a twist by a character
  $\chi$ independent of the variable $z$. If $\z-spec \in F$,
  $\z-spec \neq 0,1$, then the two specialized representations are
  isomorphic up to a twist by $\chi$. By Theorem~\ref{thm:Trace-match}
  for any prime ideal $\id{p}$ of $F$ not dividing $N$ nor
  $\z-spec(\z-spec-1)$, the trace of the Frobenius element
  $\Frob_{\id{p}}$ is the same for both representations, so the
  character $\chi$ is trivial.

(2)   All four motives have isomorphic monodromy representations (i.e. the
  restriction of their Galois representations to
  $\Gal(\overline{\Q(z)}/\overline{\Q}(z))$ are isomorphic). By
  Lemma~\ref{lemma:isom-condition} they are isomorphic up to the twist
  by a finite order character $\chi$ (independent of the parameter
  $z$). Let $\z-spec \in \Q$ be any number satisfying
  $\z-spec \neq 0,1$, and let $\id{p}$ be a prime ideal of $F$ not
  dividing $N \z-spec (\z-spec-1)$. Then Theorem~\ref{thm:Trace-match}
  implies that the trace of $\Frob_{\id{p}}$ on
  $\HGM(\veca,\vecb|\z-spec)$ equals
  $\HGM_{\id{p}}(\veca,\vecb|\z-spec)$, but it is clear from its
  definition that
  \[
    \HGM_{\id{p}}((a,b),(c,d)|\z-spec)=\HGM_{\id{p}}((b,a),(c,d)|\z-spec)= \HGM_{\id{p}}((a,b),(d,c)|\z-spec)= \HGM_{\id{p}}((b,a),(d,c)|\z-spec),
  \]
  so $\chi$ must be trivial.  

(3)    Let $N$ be the least common multiple of the denominators of
  $a,b,c,d$ and assume (for simplicity) that condition $\irr$
  holds. Keeping the previous notation, the Euler curve for the
  parameters $(-c,-d),(-a,-b)$ at the variable $z^{-1}$ has equation
  \[
y^N = x^A (1-x)^C(1-x/z)^B z^{D-A}.
\]
The change of variables $x=zx'$, $y=y'$ provide an isomorphism between
this curve and Euler's curve with parameters $(a,b),(c,d)$.  Clearly
the quadratic character $(-1)^{d-b}$ is invariant under the
substitution $(a,b),(c,d) \to (-c,-d),(-a,-b)$ and the same holds for
the Jacobi motive $\JacMot((-a,-b,c,d),(c-b,d-a))$ appearing in
the motive definition~(\ref{eq:motive-definition}).
\end{proof}

\begin{remark}
  It is likely that the isogenies in the theorem are actually
  isomorphisms but we have not attempted to do this except for (3).
\end{remark}

\subsection{Hodge numbers}
\label{section:rank2-hodge}
For rank two motives, we can give an explicit relation between the
Hodge vector of $\hgm$ and the output of the zig-zag procedure, since
the Hodge numbers of the Euler curve (and of its eigenspaces) as well as
the ones of the Jacobi motive are well known.
\begin{theorem}
  \label{thm:zig-zag-rank2}
  Let $(a,b),(c,d)$ be generic rational parameters. Let $r$ be the
  number of parameters among $a,b,c,d$ that are in $\Z$. Then:

(1) The motive $\hgm$ is pure of weight $r-1$.

(2) For $j$ coprime to $N$ let $\sigma_j\in \Gal(F/\Q)$ be the
automorphism sending $\zeta_N \to \zeta_N^j$. Then the Hodge
polynomial of $\hgm$ for the embedding $\sigma_j$ equals
  \[
\sum_{P}x^{-p_2}y^{p_2+r-1},
\]
where the sum runs over the blue points $P=(p_1,p_2)$ of the zig-zag procedure applied to the parameters $(aj,bj),(cj,dj)$.
\end{theorem}
\begin{proof}
By definition, the hypergeometric motive $\hgm$ satisfies
\[
  \hgm \otimes \JacMot((-a,-b,c,d),(c-b,d-a))= \Euler  \twist_N^{(d-b)N}.
\]
The space of differentials of the first kind on $\C$ has an action of
$\mubb_N$ (the $N$-th roots of unity); the dimension of the
$\zeta_N^j$ eigenspace (for $\gcd(j,N)=1$) is given by the formula
\begin{equation}
  \label{eq:euler-hodge}
  \left\{\frac{jA}{N}\right\} + \left\{\frac{jB}{N}\right\} + \left\{\frac{jC}{N}\right\} - \left\{\frac{j(A+B+C)}{N}\right\},
\end{equation}
where $\{a\}$ denotes the fractional part of $a$ (see for example \S
4, formula (21) of \cite{MR0931211}). 

The motivic weight of the Jacobi motive equals $n=2-r$.  The Hodge
number (corresponding to the embedding $\sigma_j$) of the Jacobi
motive is given (see \eqref{eq:hodge}) by
\[
p:=\{-aj\} + \{-bj\} + \{cj\} + \{dj\} -\{(c-b)j\} - \{(d-a)j\}, \quad p+q=n.
\]
Table~\ref{table:Hodge} contains, for $a\leq b \in (0,1]$ and
$c \leq d \in [0,1)$, all possible outcomes of the zig-zag procedure
together with the hodge polynomial of the Jacobi motive and of Euler's
curve for the embedding $\sigma(\zeta_N)=\zeta_N$ in the case $r=0$.
The statement follows by comparing the last two columns.
\begin{table}[h]
\begin{tabular}{|c|c|c|c|c|c|}
\hline
  Case & Condition & $\JacMot$ & $\text{Euler}$ & $\hgm$ & ZigZag \\
  \hline
  I & $a\leq b<c\leq d$ & $x^2$ & $x+y$ & $x^{-1}+x^{-2}y$ &
\begin{tikzpicture}[scale=0.3]
\draw[step=1.0,black, very thin] (0,0) grid (4,2);

\draw[green, dashed] (0,0) -- (4,0);

\draw[very thick] (0,0) -- (1,1);
\draw[very thick] (1,1) -- (2,2);
\draw[very thick] (2,2) -- (3,1);
\draw[very thick] (3,1) -- (4,0);

\draw[fill=red] (0,0) circle (.2);
\draw[fill=red] (1,1) circle (.2);
\draw[fill=blue] (2,2) circle (.2);
\draw[fill=blue] (3,1) circle (.2);
\end{tikzpicture}\\
  \hline
  II & $a<c<b<d$ & $xy$ & $2y$ & $2x^{-1}$ & 
\begin{tikzpicture}[scale=0.3]
\draw[step=1.0,black, very thin] (0,0) grid (4,2);

\draw[green, dashed] (0,0) -- (4,0);

\draw[very thick] (0,0) -- (1,1);
\draw[very thick] (1,1) -- (2,0);
\draw[very thick] (2,0) -- (3,1);
\draw[very thick] (3,1) -- (4,0);
\draw[fill=red] (0,0) circle (.2);
\draw[fill=red] (2,0) circle (.2);
\draw[fill=blue] (1,1) circle (.2);
\draw[fill=blue] (3,1) circle (.2);
\end{tikzpicture} \\
  \hline
  III & $a<c \leq d<b$ & $xy$ & $x + y$ & $y^{-1}+x^{-1}$ & 
\begin{tikzpicture}[scale=0.3]
  \draw[step=1.0,black, very thin] (0,0) grid (4,2);

\draw[green, dashed] (0,1) -- (4,1);

\draw[very thick] (0,1) -- (1,2);
\draw[very thick] (1,2) -- (2,1);
\draw[very thick] (2,1) -- (3,0);
\draw[very thick] (3,0) -- (4,1);
\draw[fill=red] (0,1) circle (.2);
\draw[fill=red] (3,0) circle (.2);
\draw[fill=blue] (1,2) circle (.2);
\draw[fill=blue] (2,1) circle (.2);
\end{tikzpicture}\\
  \hline
  IV & $c\leq d<a \leq b$ & $y^2$ & $x+y$ & $xy^{-2}+y^{-1}$ & 
\begin{tikzpicture}[scale=0.3]
\draw[step=1.0,black, very thin] (0,0) grid (4,2);

\draw[green, dashed] (0,2) -- (4,2);

\draw[very thick] (0,2) -- (1,1);
\draw[very thick] (1,1) -- (2,0);
\draw[very thick] (2,0) -- (3,1);
\draw[very thick] (3,1) -- (4,2);
\draw[fill=red] (2,0) circle (.2);
\draw[fill=red] (3,1) circle (.2);
\draw[fill=blue] (0,2) circle (.2);
\draw[fill=blue] (1,1) circle (.2);
\end{tikzpicture}
  \\
\hline
  V & $c<a < d<b$ & $xy$ & $2x$ & $2y^{-1}$ & 
\begin{tikzpicture}[scale=0.3]
\draw[step=1.0,black, very thin] (0,0) grid (4,2);

\draw[green, dashed] (0,2) -- (4,2);

\draw[very thick] (0,2) -- (1,1);
\draw[very thick] (1,1) -- (2,2);
\draw[very thick] (2,2) -- (3,1);
\draw[very thick] (3,1) -- (4,2);
\draw[fill=red] (1,1) circle (.2);
\draw[fill=red] (3,1) circle (.2);
\draw[fill=blue] (0,2) circle (.2);
\draw[fill=blue] (2,2) circle (.2);
\end{tikzpicture} \\
\hline
  VI & $c<a \leq b<d$ & $xy$ & $x+y$ & $x^{-1}+y^{-1}$ &
\begin{tikzpicture}[scale=0.3]
\draw[step=1.0,black, very thin] (0,0) grid (4,2);

\draw[green, dashed] (0,1) -- (4,1);

\draw[very thick] (0,1) -- (1,0);
\draw[very thick] (1,0) -- (2,1);
\draw[very thick] (2,1) -- (3,2);
\draw[very thick] (3,2) -- (4,1);
\draw[fill=red] (1,0) circle (.2);
\draw[fill=red] (2,1) circle (.2);
\draw[fill=blue] (0,1) circle (.2);
\draw[fill=blue] (3,2) circle (.2);
\end{tikzpicture} \\
\hline
\end{tabular}
\caption{\label{zigzagfig-4} Possible Hodge diagrams in the rank $2$ case}
\end{table}

It is easy to verify that the values of the columns (namely
$\JacMot$, Euler and $\hgm$) remain the same while permuting $(a,b)$
and $(c,d)$ in cases I, III, IV and VI. In case II, when $b<c<a<d$ or
when $a<d<b<c$, Euler's curve has Hodge polynomial $2x$ while
$\JacMot$ has Hodge polynomial $x^2$, so the value for $\HGM$ stays the
same. A similar phenomena occurs in case V.

When $r$ is non-zero, the $p$-value of the Jacobi motive remains the
same, but the motivic weight decreases by $r$ hence the same occurs to $q$.
\end{proof}
\begin{example}
\label{example:shimura}
  Let $\C$ be the Euler curve with equation
  \[
\C: y^5=x(1-x)(1-tx),
\]
as studied by Shimura in~\cite{MR176113} (case (4) of the table in the
first page of his paper). It corresponds to the parameters
$\left(\frac{1}{5},\frac{4}{5}\right),\left(\frac{3}{5},1\right)$
(using formula~\eqref{eq:ABCD-parameters}). The field $\Q(\zeta_5)$
has four embeddings into $\CC$, parametrized (once we choose a fifth
root of unity in $\CC$) by the elements $i \in (\Z/5)^\times$.  It is
easy to verify that the values $i=1, 4$ give type V, the value $i=2$
gives type IV, while the value $i=3$ gives type VI (see
Figure~\ref{zigzagfig-4}). We summarize the individual Hodge numbers
for each motive (hypergeometric, Euler and Jacobi respectively) in the
Table~\ref{table:Hodge} (note that here $r=1$).

\begin{table}[h]
\begin{tabular}{|r||r||r|r|r|r|r||r|r|r|r|r||r|r|r|r|r|}
\hline
$i$&Case & \multicolumn{5}{|c||}{$\HGM^{j,-j}$} & \multicolumn{5}{|c||}{Euler$^{j,1-j}$} 
& \multicolumn{5}{|c|}{Jacobi$^{j,1-j}$} \\
  \hline
j&&  $-2$ & $-1$ & $0$ & $1$ & $2$ & $-2$ & $-1$ & $0$ & $1$ & $2$
& $-2$ & $-1$ & $0$ & $1$ & $2$\\
  \hline
\hline
  $1$&V& & & $2$ & & & & & &$2$ &
 & & & &$1$ &\\
  \hline
  $2$&IV& && $1$ & $1$ & & & &$1$ & $1$ & &
 & &$1$ & & \\
 
  \hline
  $3$&VI&  & $1$ &$1$ & & & & &$1$ & $1$ & &
& & &$1$ &  \\
  \hline
  $4$&V && & $2$ & & & & &$2$ &  & &
 & &$1$ & & \\
  \hline
\hline
Total&&$0$&$1$&$6$&$1$&$0$&$0$&$0$& $4$&$4$&$0$
& & &$2$ &$2$ & \\
\hline
\end{tabular}
\caption{Hodge values \label{table:Hodge}}
\end{table}

\begin{remark} We see that the
  $\HGM((\frac{1}{5},{4}{5}),(\frac{3}{5},1)|z)$ is the
  \emph{half-twist} of $H^1$ of the Euler curve $\C$, in the notation
  of van Geemen~\cite{VG1}, with respect to the CM type
  $\Sigma:=\{\sigma_1,\sigma_2\}$ of $\Q(\zeta_5)$.
\end{remark}

Similarly, we can study the case (6) of Shimura's table (in
\cite{MR176113}). It is given by Euler's curve with equation
\[
\C: y^7=x(1-x)(1-tx),
\]
corresponding to the parameters
$\left(\frac{1}{7},\frac{6}{7}\right),\left(\frac{5}{7},1\right)$. Now
the embeddings are parametrized by the elements $i \in
(\Z/7)^\times$. The values $i=1, 2, 5, 6$ yield zig-zag diagrams of
type
V, with Hodge polynomial $2$ while for $i=3$ we get type
IV with Hodge polynomial
 $1+xy^{-1}$
 and for $i=4$ type
 VI with polynomial
$1+x^{-1}y$. Summarizing, the Hodge polynomials are
$$
2, 2, 1+xy^{-1}, 1+x^{-1}y, 2, 2.
$$
\end{example}

\subsection{On the coefficient field of the hypergeometric motive}

\begin{lemma}
  \label{lemma:coef-field}
  Let $a,b,c,d$ be generic rational parameters such that
  $c \not\equiv d \bmod \Z$. Let $N$ be their least common denominator. Let
  $L$ be a finite extension of $F$. Let $\z-spec \in L$,
  $\z-spec \neq 0,1$, be such that there exists a prime ideal $\id{p}$
  of $L$ satisfying
  \begin{itemize}
    
  \item $\id{p}\nmid N$,
    
  \item  $v_{\id{p}}(\z-spec)>0$ and prime to $N$.
  \end{itemize}
  Then both $\zeta_{c+d}$ and $\zeta_{c} + \zeta_{d}$ belong to the
  coefficient field of $\hgms$. Similarly, if $a \not\equiv b$ and there
  exists a prime ideal $\id{p}$ of $L$ not dividing $N$ such that
  $v_{\id{p}}(\z-spec)<0$ and is prime to $N$ then both $\zeta_{a+b}$
  and $\zeta_{a} + \zeta_{b}$ belong to the coefficient field of
  $\hgms$.
\end{lemma}
\begin{proof}
  Recall that the geometric representation attached to $\hgm$ has
  monodromy matrices
  $M_\infty$ conjugate to $\left(\begin{smallmatrix} \zeta_a & 0\\ 0 &
      \zeta_b\end{smallmatrix}\right)$ at $\infty$ if $a \not \equiv
  b\bmod \Z$ and
  $M_0$ conjugate to $\left(\begin{smallmatrix} \zeta{-c} & 0\\ 0 &
      \zeta_{-d}\end{smallmatrix}\right)$ at $0$ if $c \not \equiv d$.

  Let $\id{p}$ be a prime ideal of $L$ satisfying the hypothesis. Then
  by Corollary~\ref{coro:specialized-inertia} (see also \cite[Theorem
  1.2]{MR1116916}) the image of inertia $I_{\id{p}}$ is generated by
  $M_0^r$ (up to conjugation), so \cite[Lemma 15]{MR4269428} implies
  that the coefficient field contains both $\zeta_{r(c+d)}$ and
  $\zeta_{-rc} + \zeta_{-rd}$, which is equivalent to say that it
  contains $\zeta_{c+d}$ and $\zeta_{-c} + \zeta_{-d}$ (because the
  extension $\Q(\zeta_c,\zeta_d)/\Q$ is abelian). The second statement
  follows from a similar argument.
\end{proof}
\begin{proposition}
  \label{prop:coef-lowerbound}
With the hypothesis of the previous lemma, let $\z-spec \in L$ be
  such that there exists prime ideals $\id{p}$ and $\id{q}$ of $L$
  prime to $N$ satisfying:
  \begin{itemize}
    
  \item $v_{\id{p}}(\z-spec) >0$ and prime to $N$,
    
  \item $v_{\id{q}}(\z-spec) <0$ and prime to $N$.
  \end{itemize}
  Then if $\irr$ holds and $|H|=4$, the coefficient field of $\hgms$
  equals $F^H$.
\end{proposition}
\begin{proof}
  Let $M$ denote the coefficient field. By
  Theorem~\ref{thm:extension2}, $M$ is contained in $F^H$ hence
  $[F:M]\ge |H|$. The hypothesis on $\z-spec$ (and the last lemma) imply that
  $\Q(\zeta_{a+b},\zeta_a+\zeta_b,\zeta_{c+d},\zeta_c+\zeta_d) \subset
  M \subset F^H\subset F$, so $[F:M]\le 4$ (because
  $F=\Q(\zeta_a,\zeta_b,\zeta_c,\zeta_d)$). Since $|H|=4$, $M=F^ H$.
\end{proof}

\begin{example} 
  \label{example:genericity-condition}
  The last proposition suggests that some genericity condition on the
  specialization parameter $\z-spec$ is needed for the coefficient
  field to match $F^H$. Consider the case of parameters
  $(1/8,7/8),(3/8,5/8)$ (studied in Examples~\ref{example:1}
  and~\ref{example:reducible}). The group $H$ equals
  $\langle -1 \rangle \subseteq \left(\Z/8\right)^\times$, so for any
  specialization $\z-spec$ of the parameter, the motive is defined
  over $\Q(\sqrt{2})$ (by Theorem~\ref{thm:extension2}). But in
  Example~\ref{example:reducible} we showed that when $\z-spec$ is a
  square, the field of definition is $\Q$ (the motive actually
  corresponds to a elliptic curve over $\Q$).
\end{example}

\section{Extension to $K$}
\label{section:extension}
Let $H$ be the group \eqref{eq:H-defi} for generic parameters
$(a,b),(c,d)$. 
\begin{lemma}
  The group $H$ is a subgroup of $\Z/2 \times \Z/2$.
  \label{lemma:H-bound}
\end{lemma}

\begin{proof}
  For $S$ a finite set, let $\Perm(S)$ denote the group of bijective
  maps on $S$.  Set $S=\{a,b\} \times \{c,d\}$ (we are not assuming
  $a \neq b$ nor $c \neq d$). Then there is a group morphism
  \[
    \psi: H \to \Perm(\{a,b\}) \times \Perm(\{c,d\}),
  \]
  sending an element $h$ to the bijective map given by multiplication
  by $h$ on the sets $\{a,b\}$ and $\{c,d\}$. Since
  $\Perm(\{a,b\}) \times \Perm(\{c,d\}) \hookrightarrow \Z/2 \times
  \Z/2$, it is enough to prove that $\psi$ is injective. Let $N$ be
  the least common denominator of $a,b,c,d$. Let $i \in (\Z/N)^\times$
  be an element in the kernel of $\psi$, so $i$ satisfies 
  \[
    ia = a,\qquad ib=b,\qquad ic=c,\qquad id=d,
  \]
  as elements of $\Q/\Z$. Write  $a = \frac{\alpha}{N}$,
  $b=\frac{\beta}{N}$, $c=\frac{\gamma}{N}$, $d=\frac{\delta}{N}$ with
  $\gcd(\alpha,\beta,\gamma,\delta,N)=1$. Then there are integers
  $x_1,x_2,x_3,x_4$ such that the following congruence holds
  \[
\alpha x_1+\beta x_2 + \gamma x_3 + \delta x_4 \equiv 1 \pmod{N}.
\]
Multiplying by $i$ we get that
\[
i \equiv i\alpha x_1+i\beta x_2 + i\gamma x_3 + i\delta x_4 \equiv \alpha x_1+\beta x_2 + \gamma x_3 + \delta x_4 \equiv 1 \pmod N.
\]
\end{proof}
As stated in Conjecture~\ref{conjecture:field-of-defi}, it is expected
that $K=F^H$ is both the field of definition and the coefficient field
of the motive $\hgm$.

\begin{theorem}
\label{thm:extension2}
Let $(a,b),(c,d)$ be generic parameters such that $a+b$ and $c+d$ are
integers and $\irr$ holds. Then $\hgm$ descends to $F^H$ and for any
specialization its coefficient field is contained in $F^H$.
\end{theorem}
We start proving some preliminary results.
\begin{lemma}
\label{lemma:Tate-twist}
Let $(a,b),(c,d)$ be generic rational parameters satisfying that $a+b$
and $c+d$ are integers. Let $N$ be their least common
denominator. Set
  \[
\delta =
\begin{cases}
  1 & \text{ if }a,c \not \in \ZZ,\\
  0 & \text{ otherwise}.
\end{cases}
\]
Let $\id{p}$ be a prime ideal of $\Q(\zeta_N)$ of norm $q$, not
dividing $N$. Then
  \[
\JacMot((-a,-b,c,d),(c-b,d-a))(\id{p}) = q^\delta.
\]
\end{lemma}
\begin{proof}
  The hypothesis $a+b \in \ZZ$ (respectively $c+d \in \ZZ$) imply that
  we can replace $b$ by $-a$ (respectively $d$ by $-c$) in the Jacobi
  motive
  \[
\JacMot((-a,-b,c,d),(c-b,d-a))(\id{p})=\JacMot((-a,a,c,-c),(a+c,-a-c))(\id{p}).
    \]
    If $\chi$ is a character of $\FF_q^\times$,
    \[
g(\chi)g(\overline{\chi}) = \chi(-1)
\begin{cases}
  q & \text{ if } \chi\neq 1,\\
  1 & \text{ if } \chi = 1.
\end{cases}
\]
The genericity condition on the coefficients imply that
$a+c = a-d \not \in \ZZ$ and at most one of $a, c$ is an integer.
\end{proof}
It follows that if $a+b$ and $c+d$ are integers then the hypergeometric
motive $\hgm$ is (up to a Tate twist) the twist by $\twist_N$ (see
Definition~\ref{defi:twist}) of the motive coming from Euler's
curve. Assume that $t=v_2(N)>0$ (as otherwise $\twist_N$ is
trivial). Then Corollary~\ref{lemma:char-extension} implies that
$\twist_N$ matches the restriction to $\Gal(\overline{\Q}/F)$ of a
character $\kappa$ of $\Gal(\overline{\Q}/\Q)$ of conductor $2^{t+1}$.
\begin{lemma}
  If $(a,b),(c,d)$ are generic parameters such that $a+b$ and $c+d$
  are integers, then the character $\kappa$ is at most quadratic when
  restricted to $\Gal(\overline{\Q}/K)$.
  \label{lemma:char-values}
\end{lemma}

\begin{proof}
  We can assume that $t>0$ (as otherwise $\twist_N$ and $\kappa$ are
  trivial). Since $-1 \in H$ (complex conjugation), a simple Galois
  theory exercise shows that
  $\Q(\zeta_{2^{t+1}})\cap F^H=\Q(\zeta_{2^t})^+$.
  Then $\kappa$ restricted to $\Gal(\overline{\Q}/K)$ factors through
  $\Gal(\Q(\zeta_{2^{t+1}})/\Q(\zeta_{2^t})^+)$, a group isomorphic to
  $\Z/2 \times \Z/2$ if $t > 1$ and $\Z/2$ if $t=1$, hence $\kappa$ is
  at most quadratic.
\end{proof}
Our goal is to construct a motive $\mathfrak{M}$ defined over $K$
whose base change to $F$ is isomorphic to Euler's motive and define
$\hgm$ as in \eqref{eq:motive-definition}. Since $\kappa$ is (at most)
quadratic and the Jacobi motive corresponds to a Tate twist, for any
specialization $\z-spec$ of $z$, the coefficient field of $\hgms$ is
the same as that of $\mathfrak{M}$.  The motive $\mathfrak{M}$ appears
in the quotient of Euler's curve $\C$ by involutions that we now
describe.
\begin{proposition}
  \label{prop:involutions}
  Let $a,b,c,d$ be generic parameters such that $a+b=c+d=0$. For
  $j \in H$ define a morphism $\iota_j:\C \to \C$ by
  \begin{itemize}
  \item If $j=-1$, 
    \begin{equation}
      \label{eq:inv}
   \iota_{-1}(x,y)=\left(\frac{1}{zx},\frac{1}{y}\right).      
    \end{equation}
  \item If $ja=a+r$ and $jc=-c+s$, for some $r,s\in \ZZ$,
  \begin{equation}
    \label{eq:involution-2}
   \iota_{j}(x,y)=\left(\frac{x-1}{zx-1},\frac{x^{r-s}(1-x)^{-r-s}(1-xz)^{r+s}z^{-s}}{y^j}\right).    
  \end{equation}
\item If $ja=-a+r$ and $jc=c+s$, for $r,s\in \ZZ$,
  \begin{equation}
    \label{eq:involution-3}
   \iota_{j}(x,y)=\left(\frac{zx-1}{z(x-1)},\frac{x^{r-s}(1-x)^{-r-s}(1-xz)^{r+s}z^{-s}}{y^j}\right).    
  \end{equation}
  \end{itemize}
  The map $\iota_j$ is an involution of $\C$ and the new part of
  $\C/\iota_j$ has genus $\phi(N)/2$.
\end{proposition}
\begin{proof}
  Follows from a standard computation.
\end{proof}

\begin{proof}[Proof of Theorem~\ref{thm:extension2}] Let $N$ be the
  least common denominator of the parameters $a,b,c,d$. Let $j \in H$
  and let $\iota_j$ be the involution defined in
  Proposition~\ref{prop:involutions}. The new part of $\C/\iota_j$ has
  genus $\phi(N)/2$. Abusing notation, denote by $\zeta_N$ the map on
  $\C$ given by $\zeta_N(x,y)=(x,\zeta_N y)$. It follows easily from
  its definition that
  \begin{equation}
    \label{eq:commuting-relation}
    \iota_j \circ \zeta_N =
    \begin{cases}
      \zeta_N^{-j} \circ \iota_j & \text{ if }j \neq -1,\\
       \zeta_N^{-1} \circ \iota_j & \text{ if }j=-1.
    \end{cases}
  \end{equation}
  Then over the field $F^{\langle j\rangle}$ there is an action of the
  ring $\Z[\zeta_N + \zeta_N^{-j}]$ (respectively
  $\Z[\zeta_N+\zeta_N^{-1}]$) on $H^1(C/\iota,F)$. Looking at the
  $\chi_0$ isotypical component, we get a $2$-dimensional
  representation with coefficient field $F^{\langle j\rangle}$. If $H$
  is cyclic, this concludes the proof; when $H$ is not cyclic, take
  the quotient by two different involutions attached to elements on
  $H$ (getting a representation over $F^H$).
\end{proof}

The irreducibility assumption $\irr$ in the theorem is not a strong one. 

\begin{lemma}
  \label{lemma:non-irr}
  Let $(a,b), (c,d)$ be generic parameters such that $a+b$ and $c+d$
  are integers. If $\irr$ does not hold then $H = \{\pm 1\}$.
\end{lemma}
\begin{proof}
  By Lemma~\ref{lemma:equiv-irr} the values
  $N:=\lcm\{\den(a),\den(c)\}$ and $N':=\lcm\{\den(a+c),\den(a-c)\}$
  are different so $N=2N'$. This implies that
  $v_2(\den(a))=v_2(\den(c))$. If $j \in H$ and $j \neq \pm1$ then
  $ja =a \pmod \Z$ and $jc = -c \pmod \Z$ or vice-versa. But then
  $v_2(j-1) = v_2(j+1) = v_2(\den(a))$, a contradiction.
\end{proof}

A result similar to Theorem~\ref{thm:extension2} can be proved under
the following hypothesis.
\begin{theorem}
  \label{thm:H=2}
  Let $(a,b),(c,d)$ be generic parameters such that $\irr$
  holds. Suppose furthermore that $H=\langle j \rangle$ where the
  action of multiplication by $j$ is not trivial on both sets $\{a,b\}$
  and $\{c,d\}$. Then $\hgm$ descends to $F^H$ and its coefficient
  field is contained in $F^H$.
\end{theorem}
\begin{proof}
Consider the \emph{twisted} Euler curve
  \begin{equation}
    \label{eq:euler-twisted}
    \widetilde{\C}:y^N=(-1)^{(b-d)N}x^A(1-x)^B(1-xz)^Cz^D,
  \end{equation}
  where the exponents are as in \eqref{eq:ABCD-parameters}. 
Then
  \[
    \hgm = H^1_{\chi_0}(\widetilde{\C},F)^{\text{new}} \otimes \JacMot((-a,-b,c,d),(c-b,d-a))^{-1}.
  \]
  The action of multiplication by $j$ fixes both the Jacobi motive and
  the twisted Euler curve. Then, as in the previous case, we can define
  an involution $\iota_j$ and the proof follows mutatis mutandis that
  of Theorem~\ref{thm:extension2}.
\end{proof}

When the group $H$ is cyclic, we can prove a weaker general version of
the result.
\begin{theorem}
  \label{thm:extension1}
  Let $(a,b),(c,d)$ be generic parameters such that $|H| \le 2$. Then
  for all $\z-spec \in K$, $\z-spec \neq 0,1$, the Galois
  representation of $\hgms$ extends to $\Gal_{K}$ and its
  coefficient field is contained in a quadratic extension of $K$.
\end{theorem}

Let us start with some auxiliary needed results.

\begin{theorem}
  \label{thm:irred-spec}
  Let $(a,b),(c,d)$ be generic rational parameters. For $\z-spec$ outside a
  thin set of $K$, the Galois representation $\rho_{\hgm,\id{p}}$ of
  $\Gal_F$ is absolutely irreducible.
\end{theorem}

\begin{proof}
  By Theorem~\ref{thm:motive-function-field} the Galois representation
  $\rho_{\hgm,\id{p}}$ extends the monodromy representation $\rho$, an
  absolutely irreducible representation (by \cite[Proposition
  3.3]{BH}). After the choice of a stable lattice we can assume that
  our representation $\rho_{\hgm,\id{p}}$ takes values in
  $\GL_2(\Om_{\id{p}})$. Then there exists a positive integer $n$ such
  that the reduction
 \[
   \rho_{\hgm,\id{p},n}: \Gal(\overline{\Q(z)}/F(z)) \to
   \GL_2(\Om_{\id{p}}/\id{p}^n)
  \]
  is irreducible. The field fixed by its kernel is a finite extension
  $L(z)/F(z)$. Then Hilbert's irreducibility theorem (see for example
  \S 3.4 of \cite{MR2363329}) implies that for $\z-spec$ outside a thin
  set, $\Gal(L(\z-spec)/F) = \Gal(L(z)/F(z))$, so the representation
  $\rho_{\hgms,\id{p}}$ is irreducible.
\end{proof}

\begin{example}
  Let $F=\Q(\zeta_3)$, where $\zeta_3$ is a third root of unity.  The
  motive $\HGM((1/3,2/3),(1,1)|z)$ matches (as discovered in
  \cite{Cohen}) the motive attached to the rational
  elliptic curve with equation
  \[
    E_z: y^2+xy+\frac{z}{27}y=x^3.
  \]
  For $z=9/8$ the elliptic curve has complex multiplication by
  $\Z[\zeta_3]$ over $F=\Q(\sqrt{-3})$ (as it is easy to verify). The
  map: $x=u^2x'+r$, $y=u^3y'+u^2sx'+t$, with
  \[
    u = \zeta_3, \quad r=-\frac{\zeta_3+2}{12},\quad
    s=\frac{\zeta_3-1}{2},\quad t=\frac{\zeta_3+2}{24}.
  \]
  is an endomorphism of the curve defined over $F$. It follows that
  the Galois representation attached to
$$
\HGM((1/3,2/3),(1,1)|9/8)
$$
is reducible.
\label{example:reducible2}
\end{example}

\begin{example}[Shimura's Example~\ref{example:shimura} continued]
\label{example:shimura-2}
 Take $\z-spec=1/2$, corresponding to the curve
  \[
    \C:y^5=x(1-x)(1-x/2).
  \]
  The map $\iota(x,y)=(2-x,-y)$ is an involution of $\C$. The quotient
  curve can be computed using~\cite{MR1484478}.
\begin{verbatim}
A<x,y>:=AffineSpace(Rationals(),2);
C:=Curve(A,y^5-x*(1-x)*(1-x/2));
G:=AutomorphismGroup(C);
CG,prj := CurveQuotient(G);
\end{verbatim}
It is given by the hyperelliptic model
\[
 \C_1:y^2=x(x^5-8).
\]
This implies that the corresponding hypergeometric motive at
$\z-spec=1/2$ is reducible.  To obtain an isogeny complement to the
Jacobian of $\C_1$ in the Jacobian of $\C$ we use the code developed
in~\cite{MR3904148} and we find the hyperelliptic curve
\[
  \C_2:y^2=x(x^5-16).
\]
The two curves $\C_1$ and $\C_2$ are actually isomorphic over the
field $\Q(\sqrt[5]{2})$, an isomorphism being
$(x,y) \to (\sqrt[5]{2}x,\sqrt[5]{8}y)$. Their Jacobians are not
isogenous over $F=\Q(\zeta_5)$; this can be verified by computing
their Frobenius polynomial at $11$
\begin{verbatim}
P<x> := PolynomialRing(RationalField());
C := HyperellipticCurve(x*(x^5-8));
D := HyperellipticCurve(x*(x^5-16));
LPolynomial(ChangeRing(C,GF(11))); 
> 121*x^4 + 121*x^3 + 51*x^2 + 11*x + 1
LPolynomial(ChangeRing(D,GF(11)));
> 121*x^4 - 99*x^3 + 41*x^2 - 9*x + 1
\end{verbatim}
Both curves have an order $5$ automorphism sending
$(x,y) \to (\zeta_5^2x,\zeta_5y)$ so correspond to abelian surfaces
with complex multiplication by $\Z[\zeta_5]$. Hence, the motive
$\HGM((\frac{1}{5},\frac{4}{5}),(\frac{3}{5},1)|\frac{1}{2})$
corresponds to the sum of two (distinct) Hecke characters.

A result of Bailey (see \cite{MR185155}) gives the following identity 
\[
{}_2F_1\left(a,1-a;c|\frac{1}{2}\right) = \frac{\Gamma(\frac{c}{2})\Gamma(\frac{c+1}{2})}{\Gamma(\frac{c+a}{2})\Gamma(\frac{1+c-a}{2})}.
\]
A similar formula is expected to hold motivically: the left hand side
corresponding to a hypergeometric motive and the right hand side to a
Jacobi motive. Concretely, taking
$$
a_1=1/5, c_1=2/5,  \qquad a_2=6/5, c_2=2/5
$$
we could expect
\begin{equation}
  \label{eq:shim-hgm}
  \HGM\left(\left(\frac{1}{5},\frac{4}{5}\right),\left(\frac{3}{5},1\right)\left|\frac{1}{2}\right.\right) = \JacMot\left(\left(\frac{1}{5},\frac{7}{10}\right),\left(\frac{4}{5},\frac{1}{10}\right)\right) \oplus \JacMot\left(\left(\frac{1}{5},\frac{7}{10}\right),\left(\frac{3}{10},\frac{3}{5}\right)\right).
\end{equation}
We verified that the Euler factors of both sides coincide for primes
up to 500 (it seems like a nice exercise to provide a proof).
\end{example}

\vspace{2pt}

\noindent {\bf Question:} Let $S$ denote the set of values where the
specialization of the Galois representation is reducible. Besides the
proven statement on $S$ being thin, can more be said? Is it finite?

\vspace{2pt}

\begin{lemma}
  \label{lemma:extension}
  Let $L/K$ be a cyclic extension of number fields, and let
  $\rho:\Gal_L \to \GL_d(\overline{\Q_p})$ be an irreducible
  continuous Galois representation. Then the representation $\rho$
  extends to a representation
  $\tilde{\rho}:\Gal_K \to \GL_d(\overline{\Q_p})$ if and only if for
  all $\tau \in \Gal(L/K)$, the representation
  $\rho^\tau:\Gal_L \to \GL_d(\overline{\Q_p})$ defined by
  $\rho^\tau(\sigma):= \rho(\tau \sigma \tau^{-1})$ is isomorphic to
  $\rho$.
\end{lemma}

\begin{proof}
  See \cite[Lemma 3.4]{MR3968906} (even when the result is stated for
  $L/K$ quadratic, the proof works the same for cyclic extensions).
\end{proof}

\begin{remark}
  The result is only valid for cyclic extensions; it is not hard to
  construct examples of abelian extensions $L/K$ with Galois group
  $\Z/2 \times \Z/2$ where the last result does not hold. It is
  however true for abelian extensions that a twist of our representation does
  extend.
\label{rem:irreducibility}
\end{remark}

\begin{proof}[Proof Theorem~\ref{thm:extension1}]
  Start assuming that $\rho_{\hgms,\id{p}}$ is absolutely irreducible
  (by Theorem~\ref{thm:irred-spec} and Remark~\ref{rem:irreducibility}
  this occurs for $\z-spec$ outside a thin set). Since $H$ is
  cyclic, Lemma~\ref{lemma:extension} implies that it is enough to
  prove that for all $\sigma \in H$, the representation
  $\rho_{\hgms,\id{p}}^\sigma$ is isomorphic to $\rho_{\hgms,\id{p}}$,
  or equivalently (since $\rho_{\hgms,\id{p}}$ is irreducible) that
  for all prime ideals $\id{q}$ of $F$ on a density one set both
  representations have the same trace at $\Frob_{\id{q}}$. By
  Theorem~\ref{thm:Trace-match} and
  Proposition~\ref{prop:props-hyperg-sum}, if
  $\sigma(\zeta_N)=\zeta_N^j$, it amounts to verify that
  \[
    H_{\id{q}}((a,b),(c,d)|\z-spec) = H_{\id{q}}((ja,jb),(jc,jd)|\z-spec),
  \]
  which follows from the definition of $H$.

  If the representation is reducible, then there exists Hecke
  characters $\chi_1,\chi_2$ of $\Gal_F$ such that (up to semisimplification)
  \[
\rho_{\hgms,\id{p}}^{\text{ss}} \simeq \chi_{1,\id{p}} \oplus \chi_{2,\id{p}}.
\]
Let $\sigma$ be the non-trivial element of $\Gal(F/K)$. Either
${}^\sigma \chi_1 = \chi_1$ or ${}^\sigma \chi_1 = \chi_2$. In the
first case, both characters $\chi_i$, $i=1,2$, extend to $\Gal_K$, so
the extension is defined as the sum of the two extended characters. In
the second case, the extension is defined as the induced
representation $\Ind_{\Gal_F}^{\Gal_K}\chi_1$ (which is isomorphic to
the induction of $\chi_2$).
\end{proof}

\begin{corollary}
  \label{coro:extension-motive}
  If $a+b$ and $c+d$ are integers, then for $\z-spec \in K$,
  $\z-spec \neq 0,1$, the Galois representation attached to $\hgms$
  extends to $\Gal_K$.
\end{corollary}
\begin{proof}
  If $\irr$ holds, the result follow from
  Theorem~\ref{thm:extension2}. Otherwise, Lemma~\ref{lemma:non-irr}
  implies that $H$ is cyclic, so the result follows from
  Theorem~\ref{thm:extension1}.
\end{proof}

\section{Hypergeometric Motives over totally real fields}
\label{section:HGM-tot-real}
Our understanding of $L$-series coming from Galois representations is
still very limited. There are a few instances where some general
results are known, including odd $2$-dimensional representations of
weight two over totally real fields. In this case, it is known (see
for example \cite[Theorem 1.1.1]{Snowden}) that the $L$-series extends
meromorphically to the whole complex plane and satisfies a functional
equation (it is furthermore expected that the extension is
holomorphic). In this section we study hypergeometric motives defined
over totally real fields.

According to Conjecture~\ref{conjecture:field-of-defi} the motive
$\hgm$ is defined over a totally real field precisely when $-1 \in H$; i.e.
when
\[
\{a,b\} = \{-a,-b\}, \qquad \{c,d\} = \{-c,-d\}.
  \]
  Equivalently, the motive is defined over a totally real field when
  one of the following (non-disjoint) cases occurs:
  \begin{itemize}
  \item $a = -b$ and $c=-d$, or equivalently, $a+b \in \Z$ and $c+d \in \Z$,
    
  \item $a = -b$ (respectively $c=-d$), $c=-c$ and $d=-d$ (respectively
    $a=-a$ and $b=-b$).
    
  \item $a=-a$, $b=-b$, $c=-c$ and $d=-d$.
  \end{itemize}
  The last case (for generic parameters) corresponds to the motive with parameters
  $(1/2,1/2),(1,1)$ or $(1,1)(1/2,1/2)$. It is a rational motive
  corresponding to Legendre's family of elliptic curves
  \[
    E_z:y^2=x(1-x)(1-zx).
  \]
  In this case modularity for specializations $z\neq 0,1$ in $\Q$ is
  known by the work of Wiles and others (see~\cite{Wiles}
  and~\cite{MR1839918}). We consider the remaining cases.  For a
  rational number $a$, the relation $a \equiv -a \pmod \Z$ implies
  that either $a\in \Z$ or $a\in \frac{1}{2}+\Z$, so we can separate (up to a
  quadratic twist obtained by adding $1/2$ to all parameters) totally
  real motives into two disjoint cases (that we study separately):
  \begin{enumerate}
    
  \item Motives with parameters $(a,b), (c,d)$ with $a+b \in \Z$ and $c + d \in \Z$,
    
  \item Motives with parameters $(a,-a), (1/2,1)$ or $(1/2,1), (c,-c)$.
  \end{enumerate}

  \subsection{Case (1)} We start by computing the Hodge polynomials as
  explained in \S\ref{section:hodge} (using
  Theorem~\ref{thm:zig-zag-rank2}). If two of the parameters are integers, 
\begin{figure}[h]
\begin{center}
\begin{tikzpicture}[scale=0.65]
\draw[step=1.0,black, very thin] (0,0) grid (4,2);
\draw[green, dashed] (0,2) -- (4,2);
\draw[very thick] (0,2) -- (1,1);
\draw[very thick] (1,1) -- (2,0);
\draw[very thick] (2,0) -- (3,1);
\draw[very thick] (3,1) -- (4,2);
\draw[fill=red] (2,0) circle (.1);
\draw[fill=red] (3,1) circle (.1);
\draw[fill=blue] (0,2) circle (.1);
\draw[fill=blue] (1,1) circle (.1);
\end{tikzpicture}
\end{center}
\caption{\label{fig:totreal-1} Zigzag outcome when $a,b \in \ZZ$ or $c,d\in \ZZ$}
\end{figure}
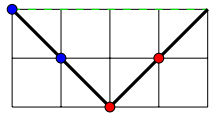
then for all embeddings $\sigma:K \hookrightarrow \CC$, the zigzag
procedure gives the diagram in Figure~\ref{fig:totreal-1} and
the Hodge polynomial equals $x+y$. Then for any $\z-spec \in K \setminus\{0,1\}$, the Galois
representation of $\Gal_K$ attached to $\hgms$ has Hodge-Tate weights
$\{0,1\}$.

When no parameter is an integer, the zig-zag procedure gives the
diagram in Figure~\ref{fig:totreal-2}.
\begin{figure}[h]
\begin{subfigure}{.3\textwidth}
\begin{center}
\begin{tikzpicture}[scale=0.65]
\draw[step=1.0,black, very thin] (0,0) grid (4,2);
\draw[green, dashed] (0,1) -- (4,1);
\draw[red] (0,2)--(0,2) node[above]{$a$};
\draw[red] (3,2)--(3,2) node[above]{$-a$};
\draw[blue] (1,0)--(1,0) node[below]{$c$};
\draw[blue] (2,0)--(2,0) node[below]{$-c$};
\draw[very thick] (0,1) -- (1,2);
\draw[very thick] (1,2) -- (2,1);
\draw[very thick] (2,1) -- (3,0);
\draw[very thick] (3,0) -- (4,1);
\draw[fill=red] (0,1) circle (.1);
\draw[fill=red] (3,0) circle (.1);
\draw[fill=blue] (1,2) circle (.1);
\draw[fill=blue] (2,1) circle (.1);
\end{tikzpicture}
\end{center}
\end{subfigure}\hspace{0em}
\begin{subfigure}{.3\textwidth}
\begin{center}
\begin{tikzpicture}[scale=0.65]
\draw[step=1.0,black, very thin] (0,0) grid (4,2);
\draw[green, dashed] (0,1) -- (4,1);
\draw[blue] (0,2)--(0,2) node[above]{$c$};
\draw[blue] (3,2)--(3,2) node[above]{$-c$};
\draw[red] (1,0)--(1,0) node[below]{$a$};
\draw[red] (2,0)--(2,0) node[below]{$-a$};
\draw[very thick] (0,1) -- (1,0);
\draw[very thick] (1,0) -- (2,1);
\draw[very thick] (2,1) -- (3,2);
\draw[very thick] (3,2) -- (4,1);
\draw[fill=red] (1,0) circle (.1);
\draw[fill=red] (2,1) circle (.1);
\draw[fill=blue] (0,1) circle (.1);
\draw[fill=blue] (3,2) circle (.1);
\end{tikzpicture}
\end{center}
\end{subfigure}\hspace{0em}
\caption{\label{fig:totreal-2} The zigzag procedure in case $(1)$ for $a<c$
  and $a>c$}
\end{figure}
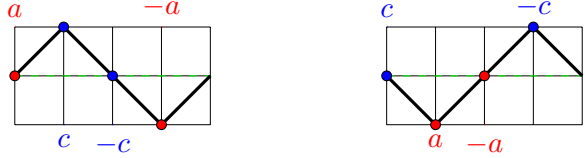
In this case, the Hodge polynomial equals $x^{-1}+y^{-1}$. Then
for any $\z-spec \in K \setminus\{0,1\}$ the Galois representation of
$\Gal_K$ attached to $\hgms$ has Hodge-Tate weights $\{-1,0\}$.

In both cases, the $L$-series attached to $\hgms$ extends
meromorphically to the whole complex plane and satisfies a functional
equation (by \cite[Theorem 1.1.1]{Snowden}).  As already mentioned, it
is expected that the $L$-series is actually holomorphic, and comes (up
to a Tate twist) from a parallel weight $2$ Hilbert modular form defined
over $K$. In some cases one can use modularity lifting theorems and
congruences to prove modularity of all specializations of the motive
(as done in \cite{GP} for the parameters
$(\frac{1}{2p},-\frac{1}{2p}),(1,1)$ for example). In other cases
(like the motive corresponding to the parameters
$(3/8,-3/8),(1/8,-1/8)$ described in the introduction) one can prove
modularity of the motive for some particular specializations (like
$\z-spec=3$) using the Faltings-Serre method.

\subsection{Case (2)} 
The zig-zag procedure gives (for all embeddings
$\sigma:K \hookrightarrow \CC$) the diagram shown in
Figure~\ref{fig:totreal-3}. Its Hodge polynomial is $2$, so it
should correspond to a parallel weight $1$ modular form.
\begin{figure}[h]
\begin{center}
\begin{tikzpicture}[scale=0.65]
\draw[step=1.0,black, very thin] (0,0) grid (4,1);
\draw[green, dashed] (0,1) -- (4,1);
\draw[red] (1,0)--(1,0) node[below]{$a$};
\draw[red] (3,0)--(3,0) node[below]{$-a$};
\draw[blue] (2,1)--(2,1) node[above]{$\frac{1}{2}$};
\draw[blue] (0,1)--(0,1) node[above]{$0$};
\draw[very thick] (0,1) -- (1,0);
\draw[very thick] (1,0) -- (2,1);
\draw[very thick] (2,1) -- (3,0);
\draw[very thick] (3,0) -- (4,1);
\draw[fill=red] (1,0) circle (.1);
\draw[fill=red] (3,0) circle (.1);
\draw[fill=blue] (0,1) circle (.1);
\draw[fill=blue] (2,1) circle (.1);
\end{tikzpicture}
\end{center}
\caption{\label{fig:totreal-3} The zigzag procedure in case $(2)$}
\end{figure}
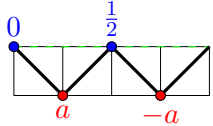
 This is indeed the case, since such motives have
finite monodromy as studied by Schwarz (in \cite{MR1579568}),
corresponding to dihedral projective image. A detailed description
will be given in a subsequent article by the last two authors.
  
In both cases, we can prove a stronger version of
Theorem~\ref{thm:irred-spec}.

\begin{theorem}
  \label{thm:real-irred}
  Let $(a,b),(c,d)$ be rational generic parameters, and
  let $\z-spec \in K$, $\z-spec \neq 0,1$. Suppose that $a+b$ and $c+d$ are
  integers. Then the Galois representation of the motive $\hgms$
  extended to $\Gal_K$ is irreducible.
\end{theorem}

\begin{proof}
  Recall that any Hecke character of the id\`ele group of a totally
  real field is (up to a finite order character) a multiple of the
  norm character (the reason is that if $d=[K:\Q]$, the compatibility
  relation at the units of $K$ impose $d-1$ relations among the $d$
  possible characters at the archimedean places). Suppose then that
  the representation is reducible, namely
  \begin{equation}
    \label{eq:red-dec}
    \rho_\lambda =
    \begin{pmatrix}
      \chi_\ell^a \varepsilon_1 & * \\
      0 & \chi_\ell^b \varepsilon_2
    \end{pmatrix},
  \end{equation}
  for some $a,b \in \ZZ$ and $\varepsilon_1,\varepsilon_2$ finite
  order characters of the id\`ele group of $K$, where $\chi_\ell$
  denotes the cyclotomic character (the precise values of $a$ and $b$
  can be given using the zig-zag procedure as previously explained).
  The representation $\rho_\lambda$ is also reducible while restricted to $\Gal_F$
  and the same holds for any twist, so the $2$-dimensional
  representation coming from
  $H_{\text{\'et}}^{1,\zeta_N^i}(\widehat{\C},\Q_\ell)$ also looks
  like~\eqref{eq:red-dec} (since by Lemma~\ref{lemma:Tate-twist}, the
  Jacobi motive is a power of the cyclotomic character).

  The Hodge-Tate weights of the Galois representation attached to the
  new part of Euler's curve are $\{0,\ldots,0,1,\ldots,1\}$, each one
  with multiplicity $\frac{\phi(N)}{2}$. From the decomposition 
  \[
      H^1_{\text{\'et}}(\Cnsp,F_\lambda)^{\rm new} := 
\bigoplus_{\chi}H^1_{\text{\'et},\chi}(\Cnsp,F_\lambda),  
  \]
  it follows that for each $i$ prime to $N$ the representation of
  $H^1_{\text{\'et},\chi}(\Cnsp,F_\lambda)$ has a decomposition
  like~\eqref{eq:red-dec} for some characters
  $\varepsilon_1^{(\chi)}, \varepsilon_2^{(\chi)}$ (but the exponents
  $a$, $b$, are independent of $\chi$, since they are Galois conjugate of
  each other). In particular, $a=0$ and $b=1$ or vice-versa.

  Let $\id{q}$ be a prime ideal of $F$ not dividing $N \ell$ such that
  $v_{\id{q}}(\z-spec(\z-spec-1))=0$ and also 
  $\varepsilon_1^{(\chi)}(\id{q})=\varepsilon_2^{(\chi)}(\id{q})=1$ for all
  order $N$ character $\chi$ (there are infinitely many such prime ideals by
  Chebotarev density theorem). Then the trace of $\Frob_{\id{q}}$
  acting on the new part of the Jacobian equals $(\normid(q)+1)\phi(N)$.

  On the other hand, Weil's conjectures imply that such trace cannot
  be larger than $2 \phi(N) \sqrt{\normid(q)}$, which is a
  contradiction, since $x+1 > 2\sqrt{x}$ if $x>1$.
\end{proof}

\section{Congruences}
\label{section:congruences}
Congruences between Galois representation have many applications in
number theory. In Langlands' program they play a crucial role in the
proof of results like base change, Serre's conjectures, potential
modularity, the Shimura-Taniyama conjecture, etc. In the study of
Diophantine equations, they also play a crucial role, like in Wiles'
proof of Fermat's last theorem. Hypergeometric motives have the
property that they satisfy many congruences.

\begin{definition}
  Let $\ell$ be a rational prime. Define on $\Q/\Z$ the relation
  defined by the condition that the denominator of $a-b$ is an
  $\ell$-th power. Denote it by $a \sim_\ell b$.
\end{definition}
The following result is easy to verify.
\begin{lemma}
  The relation $\sim_\ell$ is an equivalence relation.
\end{lemma}
Extend the definition to $(\Q/\Z)^n$ component-wise.

\begin{theorem}
  \label{thm:congruences}
  Let $\ell$ be a prime number, and let $(a,b),(c,d)$ and
  $(a',b'),(c',d')$ be two pairs of generic rational parameters. Let
  $F$ be the composite of the fields of definition of the motives
  attached to both parameters. Then if $(a,b) \sim_\ell (a',b')$ and
  $(c,d) \sim_\ell (c',d')$, the base change of both motives to $F$
  specialized at $\z-spec \in F$, $\z-spec \neq 0,1$, are congruent
  modulo $\id{l}$, for $\id{l}$ any prime ideal of their coefficient
  field dividing $\ell$.
  \end{theorem}

\begin{proof}
  By Theorem~\ref{thm:Trace-match} the trace of the hypergeometric
  motive at a prime $\id{p}$ of good reduction matches a finite
  hypergeometric sum. From its very definition (equation
  \eqref{eq:hpergeom-BCM}) the finite hypergeometric sum is a sum and
  product/quotients of Gauss sums of the form $g(m+ \alpha(q-1),\id{p})$, where
  $\alpha$ is one of the parameters (and $q=p^r$ is the order of
  the residue field $\Om_F/\id{p}$).  To prove the statement it is
  enough to verify that if $\alpha \sim_\ell \beta$ then
  $g(m+ \alpha(q-1),\id{p}) \equiv g(m+ \beta(q-1),\id{p}) \pmod{\id{l}}$. Recall
  that if $\omega$ denotes a generator of the multiplicative group
  $\FF_q^\times$, then
  \begin{equation}
    \label{eq:congruence}
g(m+\alpha(q-1)) =\sum_{x \in \Om_F/\id{p}} \omega^m(x)\omega^{\alpha(q-1)}(x)\psi(x),
\end{equation}
where $\psi$ is an additive character of $\FF_q^\times$.
The last sum is an element of $\Z[\zeta_p,\zeta_{q-1}]$.  Let $\id{l}$
be a prime ideal in this ring dividing $\ell$. Then
\[
\frac{\omega^{\alpha(q-1)}(x)}{\omega^{\beta(q-1)}(x)}=\omega^{(\alpha-\beta)(q-1)}(x) = \omega^{(q-1)/\ell^r}(x),
\]
for some non-negative integer $r$. In particular, it is a root of
unity whose order divides $\ell^r$, so congruent to $1$ modulo
$\id{l}$. This implies that $g(m+\alpha(q-1)) \equiv g(m+\beta(q-1)) \pmod{\id{l}}$
as claimed.
\end{proof}
Some explicit examples of lowering/raising the level using congruences
between hypergeometric motives were given at the introduction.

\appendix

\section{Tame inertia of an Hypergeometric Motive}
\label{appendix:Elisa}
\begin{center} by Elisa Lorenzo Garc\'ia and Ariel Pacetti\end{center}
\smallskip

Let $\mathcal{O}$ be a Henselian discrete valuation ring, let $\pi$ be
a local uniformizer, let $K$ be its field of fractions and let $k$ its
residue field. For the present applications, have in mind the case in
which $\mathcal{O}=k(z)$ is the coordinate ring of $\mathbb{P}^1_k$.
Keeping the article's notation, let $(a,b)$, $(c,d)$ be a pair of
generic parameters and let $N$ be their least common denominator. Set
\begin{equation}
  \label{eq:exponents}
  A=(d-b)N, \qquad B=(b-c)N,\qquad C=(a-d)N,\qquad D=dN.
\end{equation}
Define Euler's curve $\bfC$ over $\Om$ by
\begin{equation}
  \label{eq:curve-equation}
  \bfC: y^N = x^A(1-x)^B (1-\pi x)^C\pi^D.
\end{equation}
Assume for the rest of the appendix that the following two conditions hold:
\begin{itemize}
\item The curve $\bfC$ is irreducible (i.e. $\irr$ holds).
\item The residual characteristic of $K$ ($\car(k)$) does not divide $N$.
  
\item $K$ contains the $N$-th roots of unity.
\end{itemize}
Fix $\mu$ a root of unity in $K$ (a generator of $\mubb_N$) and
$\zeta$ a primitive $N$-th root of unity in $\CC$.  
Let $\Yst$ denote the semistable model of $\bfC$ and let $\Ab$ be the abelian
variety $\text{Pic}^0(\Yst)$. Let $\Ab_0$ denote its special fiber. It contains a
toric part $T_0$ and an abelian part $\B_0$ so that
\[
\Ab_0^{0}/T_0 \simeq \B_0
\]
(see \cite{Neron}, Example 8, page 246 and also \cite{Grothendieck}
Expose IX, equation (2.1.2)). Let $n = \dim \Ab$, $\mu = \dim T_0$ and
$a = \dim \B_0$ so that $n = \mu + a$. 

Let $\ell \equiv 1 \pmod N$ be a prime number different from the
characteristic of $K$ (so $\QQ_\ell$ contains the $N$-th roots of
unity). Let $T_\ell(\Ab)$ denote the Tate module of the generic fiber of
$\Ab$. Since $\Yst$ is semistable, the action of the inertia group $I_K$
on $T_\ell(\Ab)$ is unipotent with echelon rank $2$, i.e. there exists a
subspace $U$ such that $I_K$ acts trivially on $U$ and on
$T_{\ell}(\Ab)/U$ (see \cite[Proposition 3.5]{Grothendieck}, page 350).

Let $\Sp_2$ denote the $\QQ_\ell[I_K]$-module of dimension $2$
where the action of the inertia group $I_K$ factors through the
maximal pro-$\ell$-group $I_K(\ell)$, and a generator $g$ acts like
the matrix
$\left(\begin{smallmatrix} 1 & 1 \\ 0 & 1\end{smallmatrix}\right)$.
Let $\Upsilon = (V,E)$ denote the dual graph of of $\overline{\Yst}$
(the special fiber of $\Yst$). Then $\mu$ (the rank of the toric part
$T_0$) equals the rank of $\text{H}^1(\Upsilon,\ZZ)$ and there is an
isomorphism of $\QQ_\ell[I_K]$-modules
\begin{equation}
  \label{eq:rep-splitting}
  T_\ell(\Ab) \simeq \left(\left(\text{H}^1(\Upsilon,\ZZ)\otimes_{\ZZ}\QQ_\ell\right) \otimes \Sp_2\right) \oplus \bigoplus_{\tilde{Y} \in V} T_\ell(\text{Pic}^0(\tilde{Y})),
\end{equation}
where the sum runs over the irreducible components of
$\overline{\Yst}$. To describe the action of $I_K$ on
$T_\ell(\Ab)$ it suffices to compute the semistable model of $\bfC$ and
its dual graph.

The $\irr$ assumption and the generic condition on the parameters imply that
\begin{enumerate}
\item[(i)] $\gcd(N,A,B,C)=1$,
    
\item[(ii)] $N \nmid A$, $N \nmid B$, $N \nmid C$ and $N \nmid A+B+C$.
\end{enumerate}
A simple change of variables transforms
equation~(\ref{eq:curve-equation}) into
\begin{equation}
  \label{eq:curve2}
   \pi^{A+B-D} y^N= x^A(\pi-x)^B(1-x)^C.
\end{equation}
Up to a twist (coming from the change of variables
$y \to \pi^{(A+B-D)/N}y$), it is enough to study the curve
\begin{equation}
  \label{eq:curve3}
  \widetilde{\bfC}: y^N= x^A(\pi-x)^B(1-x)^C.
\end{equation}
The defining polynomial has now roots $0,1,\pi$, so its cluster
picture looks like
$ {\clusterpicture \Root {1} {first} {r1}; \Root {} {r1} {r2}; \Root
  {3} {r2} {r3}; \ClusterLDName c1[][][] = (r1)(r2); \ClusterLDName
  c4[][][] = (c1)(r3); \endclusterpicture} $ (see \cite{DDMM19}). Then
the decorated graph of $(\mathcal{X},\mathcal{D})$ equals
  
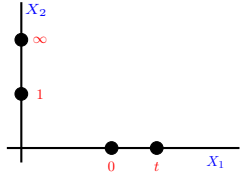
\begin{figure}[H]
\begin{tikzpicture}[xscale=1.2,yscale=1.2,
  l1/.style={shorten >=-1.3em,shorten <=-0.5em,thick}]

  \draw[l1] (1,0.00)--(3.00,0.00) node[scale=0.8,above left=-0.17em] {} node[scale=0.5,blue,below right=-0.4pt,yshift=-.1cm] {$X_1$};
\filldraw[black] (2,0) circle (2pt) node[scale=0.5,red,yshift=-.5cm] {$0$};
\filldraw[black] (2.5,0) circle (2pt) node[scale=0.5,red,yshift=-0.5cm] {$t$};

\draw[l1] (1,0.00)--node[right=-3pt] {} (1,1.2) node[xshift=.2cm, yshift=.4cm,font=\tiny, scale=0.8, blue] {$X_2$};
\filldraw[black] (1,0.6) circle (2pt) node[scale=0.5,red,xshift=0.5cm] {$1$};
\filldraw[black] (1,1.2) circle (2pt) node[scale=0.5,red,xshift=0.5cm] {$\infty$};
\end{tikzpicture}
\end{figure}
Define the following quantities:
\begin{itemize}
\item $d_1:=\gcd(N,C,A+B)$, $d_2:=\gcd(N,A,B)$.
    
\item $e:=\gcd(N,A+B)$, $t:=\frac{e}{d_1d_2}$.
\end{itemize}
It is clear that $d_1\mid e$ and $d_2 \mid e$. Condition (i) implies
that $d_1$ and $d_2$ are prime to each other, hence $t \in \ZZ$. If
$d \mid N$, we denote by $\zeta_d$ the root $d$-th root of unity
$\zeta^{N/d}$.

The special fiber $\overline{\Yst}$ has $d_1$ components over
$\overline{X_1}$ (setting $x_1=x$, $y_1=y$) given by
\begin{equation}
  \label{eq:C1-comp2}
  \bfC_1^{(i)}: y_1^{\frac{N}{d_1}}=\zeta_{2d_1}^{B+2i} x_1^{\frac{A+B}{d_1}}(1-x_1)^{\frac{C}{d_1}}, \qquad i=0,\dots,d_1-1,
\end{equation}
corresponding to the irreducible components of the curve
\[
  \bfC_1: y_1^N = (-1)^B x_1^{A+B}(1-x_1)^C.
  \]
  Over $\overline{X_2}$ the special fiber has $d_2$ components,
  setting $x_2=\pi x$ and $y=\pi^{\frac{A+B}{N}}y_2$, they are given by the
  equation
\begin{equation}
  \label{eq:C1-comp1}
  \bfC_2^{(j)}: y_2^{\frac{N}{d_2}} = \zeta_{d_2}^i x_2^{\frac{A}{d_2}}(1-x_2)^{\frac{B}{d_2}}, \qquad j=0,\ldots,d_2-1,
\end{equation}
corresponding to the irreducible components of the curve
\[
\bfC_2: y_2^N=x_2^A(1-x_2)^B.
  \]
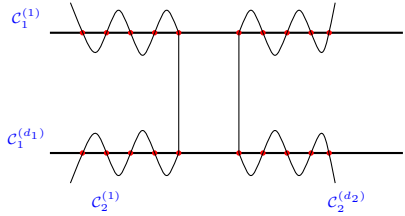
\begin{figure}[h]
\begin{tikzpicture}[xscale=.8,yscale=.8,
  l1/.style={shorten >=-1.3em,shorten <=-0.5em,thick}]
  \draw[l1] (0,1)--node[right=-3pt] {} (5,1) node[xshift=-4.5cm, yshift=.2cm,font=\tiny, scale=0.8, blue] {$\bfC_1^{(d_1)}$};
  \draw[l1] (0,3)--node[right=-3pt] {} (5,3) node[xshift=-4.5cm, yshift=.2cm,font=\tiny, scale=0.8, blue] {$\bfC_1^{(1)}$};
  \node at (0.7,.5) [font=\tiny, scale=0.8, blue,below] {$\bfC_2^{(1)}$};
  \node at (4.7,.5) [font=\tiny, scale=0.8, blue,below] {$\bfC_2^{(d_2)}$};
  \filldraw[red] (.3,3) circle (1pt) node[scale=0.5,red,yshift=-1em] {};
  \filldraw[red] (.7,3) circle (1pt) node[scale=0.5,red,yshift=-1em] {};
  \filldraw[red] (1.1,3) circle (1pt) node[scale=0.5,red,yshift=-1em] {};
  \filldraw[red] (1.5,3) circle (1pt) node[scale=0.5,red,yshift=-1em] {};
  \filldraw[red] (1.9,3) circle (1pt) node[scale=0.5,red,yshift=-1em] {};
  \filldraw[red] (.3,1) circle (1pt) node[scale=0.5,red,yshift=-1em] {};
  \filldraw[red] (.7,1) circle (1pt) node[scale=0.5,red,yshift=-1em] {};
  \filldraw[red] (1.1,1) circle (1pt) node[scale=0.5,red,yshift=-1em] {};
  \filldraw[red] (1.5,1) circle (1pt) node[scale=0.5,red,yshift=-1em] {};
  \filldraw[red] (1.9,1) circle (1pt) node[scale=0.5,red,yshift=-1em] {};

  \filldraw[red] (2.9,3) circle (1pt) node[scale=0.5,red,yshift=-1em] {};
  \filldraw[red] (3.3,3) circle (1pt) node[scale=0.5,red,yshift=-1em] {};
  \filldraw[red] (3.7,3) circle (1pt) node[scale=0.5,red,yshift=-1em] {};
  \filldraw[red] (4.1,3) circle (1pt) node[scale=0.5,red,yshift=-1em] {};
  \filldraw[red] (4.4,3) circle (1pt) node[scale=0.5,red,yshift=-1em] {};
  \filldraw[red] (2.9,1) circle (1pt) node[scale=0.5,red,yshift=-1em] {};
  \filldraw[red] (3.3,1) circle (1pt) node[scale=0.5,red,yshift=-1em] {};
  \filldraw[red] (3.7,1) circle (1pt) node[scale=0.5,red,yshift=-1em] {};
  \filldraw[red] (4.1,1) circle (1pt) node[scale=0.5,red,yshift=-1em] {};
  \filldraw[red] (4.4,1) circle (1pt) node[scale=0.5,red,yshift=-1em] {};

  \draw (0.1,3.5) .. controls (0.5,2.5) and (0.5,2.5).. (0.7,3) ;
  \draw (0.7,3) .. controls (0.9,3.5) and (0.9,3.5).. (1.1,3) ;
  \draw (1.1,3) .. controls (1.3,2.5) and (1.3,2.5).. (1.5,3) ;
  \draw (1.5,3) .. controls (1.7,3.5) and (1.7,3.5).. (1.9,3) ;
  \draw (1.9,3) .. controls (1.9,2) and (1.9,2) .. (1.9,1);
  \draw (0.1,0.5) .. controls (0.5,1.5) and (0.5,1.5).. (0.7,1) ;
  \draw (0.7,1) .. controls (0.9,0.5) and (0.9,0.5).. (1.1,1) ;
  \draw (1.1,1) .. controls (1.3,1.5) and (1.3,1.5).. (1.5,1) ;
  \draw (1.5,1) .. controls (1.7,0.5) and (1.7,0.5).. (1.9,1) ;

  \draw (2.9,3) .. controls (2.9,2) and (2.9,2) .. (2.9,1);
  \draw (2.9,3) .. controls (3.1,3.5) and (3.1,3.5).. (3.3,3) ;
  \draw (3.3,3) .. controls (3.5,2.5) and (3.5,2.5).. (3.7,3) ;
  \draw (3.7,3) .. controls (3.9,3.5) and (3.9,3.5).. (4.1,3) ;  
  \draw (4.1,3) .. controls (4.3,2.5) and (4.3,2.5).. (4.5,3.5) ;

  \draw (2.9,1) .. controls (3.1,0.5) and (3.1,0.5).. (3.3,1) ;
  \draw (3.3,1) .. controls (3.5,1.5) and (3.5,1.5).. (3.7,1) ;
  \draw (3.7,1) .. controls (3.9,0.5) and (3.9,0.5).. (4.1,1) ;  
  \draw (4.1,1) .. controls (4.3,1.5) and (4.3,1.5).. (4.5,.5) ;

\end{tikzpicture}
\caption{Intersection of Special Fiber of $\Yst$}
\label{pic:case1}
\end{figure}
The components $\C_1^{(i)}$ and $\C_2^{(j)}$ intersect at $e$
different points as in Figure~\ref{pic:case1}.

Denote by $\widetilde{\Ab}$ the Jacobian of $\widetilde{\bfC}$
and by $\widetilde{\Ab}^{\text{new}}$ its new part.  Let $\eta$ denote the character
of $\Gal_{K}$ corresponding to the extension $K(\sqrt[N]{t})/K$, and
let $\theta$ be the character corresponding to the (at most quadratic)
extension $K(\sqrt[N]{-1})/K$.
\begin{theorem}
  \label{thm:hgm}
  Let $\ell$ be a prime number different from the characteristic of
  $K$. 
  \begin{itemize}
  \item If $N \nmid A+B$, set $r_i = \frac{\varphi(N)}{\varphi(N/d_i)}$ for $i=1,2$. Then
    \[
      T_{\ell}(\widetilde{\Ab}^{\text{new}}) \simeq T_{\ell}(\Ab_1)^{r_1} \oplus
      T_{\ell}(\Ab_2)^{r_2}\otimes \eta^{-(A+B)},
    \]
    where $\Ab_1$ and $\Ab_2$ are abelian varieties with complex multiplication over $K$.
  \item Otherwise, as $\QQ_\ell[I_K]$-modules
    \[
      T_{\ell}(\widetilde{\Ab}^{\text{new}}) \simeq \bigoplus_{i=1}^{\varphi(N)}
      \Sp_2\otimes \theta ^B.
      \]
  \end{itemize}

\end{theorem}
\begin{proof}
  Start assuming that $N \nmid A+B$, so the curves $\C_1^{(i)}$ and
  $\C_2^{(j)}$ have positive genus. Let $\mu$ be a generator of
  $\mubb_N$.
  The action of $\mu$ gives
  an isomorphism between points of $\C_1^{(i)}$ and of $\C_1^{(i+1)}$
  (for $i$ an element of $\ZZ/d_1$). Then
  \[
  T_{\ell}(\Jac(\C_1^{1})) \simeq T_{\ell}(\Jac(\C_1^{i})) \qquad \text{ for all
    $i$}.
\]
The same holds for $\C_2^{(j)}$. Let $\Ab_1$ (respectively $\Ab_2$) be
the new part of $\Jac(\C_1^{(1)})$ (resp. $\Jac(\C_2^{(1)})$);
$\Ab_1$ consists of the eigenspaces for the action of
$\mubb_{\frac{N}{d_1}}$ whose eigenvalues are primitive $N/d_1$-th
roots of unity.

Let $v$ be an element of $\Jac(\C_1^{(1)})$ on the
$\zeta^{d_1}$-eigenspace for the action of $\mu^{d_1}$. Let
$i \in (\ZZ/N)^\times$ be an element congruent to $1$ modulo $N/d_1$
and consider the element in
$\Jac(\C_1) \simeq \bigoplus_j \Jac(\C_1^{(j)})$ defined by
  \[
    w_i = \zeta^{i(d_1-1)}v + \mu \cdot \zeta^{i(d_1-2)}v + \cdots +
    \mu^{d_1-2}\cdot \zeta^iv+\mu^{d_1-1}\cdot v
  \]
  It is clear from its definition that $\mu \cdot w_i = \zeta^i w_i$
  (since $\mu^{d_1}\cdot v = \zeta^{d_1}v$ and
  $i \equiv 1 \pmod{\frac{N}{d_1}}$), so $w_i$ belongs to the
  $\zeta^i$-eigenspace for the action of $\mu$ on
  $\Jac(\C_1)$. Furthermore, the map $v \to w_i$ gives an isomorphism
  between the $\zeta^{d_1}$-eigenspace for the action of $\mu^{d_1}$
  on $\Ab_1$ and the $\zeta^i$-eigenspace for the action of $\mu$ on
  $\Jac(\bfC_1)$, so as Galois modules
  \[
 T_\ell(\Jac(\bfC_1)^{\text{new}}) \simeq T_\ell(\Ab_1)^{r_1},
\]
where $r_1 = \frac{\varphi(N)}{\varphi(N/d_1)}$ (the number of possible
values of $i$).

  A similar result holds for the curve $\C_2$ and its components, but
  note that $\C_2^{(j)}$ is not a special fiber of $\widetilde{\bfC}$
  but a twist of it by $\eta^{-(A+B)}$ is (due to the change of variables
  $y=\pi^{\frac{A+B}{N}}y_2$).

  Condition (ii) (and the hypothesis $N \nmid A+B$) implies that the
  curve $\C_1^{(i)}$ (respectively $\C_2^{(j)}$) has genus
  $\varphi(N/d_1)/2$ (respectively $\varphi(N/d_2)/2$) and has an
  action of $\mubb_{\frac{N}{d_1}}$ (respectively
  $\mubb_{\frac{N}{d_2}}$), so its Jacobian has complex multiplication
  (see \cite{MR422164}).

   \vspace{2pt}

   Suppose then that $N \mid A+B$. Then $e=N$ and the components
   $\C_1^{(i)}$ and $\C_2^{(j)}$ have all genus
   zero. By~(\ref{eq:rep-splitting}) the Galois representation is a
   twist of the representation $\Sp_2$ by a character coming from the
   action of the inertia group $I_K$ on the graph cohomology. The
   edges of the graph $\text{H}^1(\Upsilon,\ZZ)$ correspond to the
   components $\C_1^{(i)}$ and $\C_2^{(j)}$, which are defined over
   $K$, so the inertia group $I_K$ acts trivially on them. The
   intersection points of $\C_1^{(i)}$ and $\C_2^{(j)}$ correspond to
   the desingularizations points at $(0,0)$ of the first curve and at
   $\infty$ of the second. By \S3.1.1 and \S3.1.2 of \cite{MR2005278},
   the intersection points of the desingularized curves are defined
   over the extension $K(\sqrt[N]{(-1)^B})$, giving the extra twist by
   $\theta^B$.
 \end{proof}
\subsection{The hypergeometric motive}
Keeping the previous notation, let $\Ab^{\zeta}$ denote the submotive of
$\Ab$ given by the $\zeta$-eigenspace for the action of $\mu$. Then (as done in \ref{defi:HGM})
 \[
 \HGM((a,b),(c,d)|z)= \Ab^{\zeta}(-1) \otimes \JacMot((-a,-b,c,d),(c-b,d-a))^{-1} \twist_K^{(d-b)N}.
 \]
 The motive $\JacMot((-a,-b,c,d),(c-b,d-a))^{-1}$ is a Hecke character
 (independent of the variable $z$), so it is unramified while
 restricted to $\Gal(\overline{\QQ(z)}/\overline{\QQ}(z))$ (the
 same holds for $\twist_K$).

 Let $\ell \equiv 1 \pmod N$ be a prime and let
 \[
   \rho_\ell: \Gal(\overline{\QQ(z)}/\overline{\QQ}(z)) \to
   \GL_2(\QQ_\ell)
 \]
 denote the $2$-dimensional Galois representation attached to
 $\HGM((a,b),(c,d)|z)$.  Recall the definition (as given in
 Proposition~\ref{prop:monodromy-eigenvalues}) of the monodromy matrix
 $M_0$ around $0$ for the parameters $(a,b),(c,d)$
  \[
    M_0 =
    \begin{cases}
      \begin{pmatrix}
        e^{-c} & 0\\
        0& e^{-d}
      \end{pmatrix} & \text{ if } c-d \not \in \ZZ,\\
      \begin{pmatrix}
        e^{-c} & 1\\
        0 &e^{-c}
      \end{pmatrix} & \text{ if }c-d \in \ZZ,
    \end{cases}
\]
where $e^a:= \exp(2\pi i a)$. 
\begin{theorem}
The image of the inertia group at $z=0$ under $\rho_{\ell}$ is
spanned (up to $\GL_2$-conjugation) by the
matrix $M_0$. Furthermore, if $s \in \QQ[z]$ and we consider Euler's
curve with parameter $s$, then the same holds, but replacing $M_0$ by
$M_0^{v(s)}$, where $v(s)$ denote the valuation of $s$ at $0$.
\label{thm:HGM-0}
\end{theorem}

\begin{proof}
  Let $\Om = \QQ[[z]]$, the completion of $\QQ[z]$ at the prime ideal
  $(z)$. It is a Henselian complete valuation ring with local
  uniformizer $z$, so we can apply the results of the previous section
  to the curve $\widetilde{\bfC}$. To relate the curve $\bfC$ to
  $\widetilde{\bfC}$ we need to twist by $\eta^{A+B-D} = \eta^{-cN}$.
  Since $A+B = (c-d)N$, if $c-d \not \in \ZZ$ we are in the first case
  of Theorem~\ref{thm:hgm}, so inertia acts on the
  $\zeta$-eigenspace like the sum of the two characters
  \[
\eta^{-c} \oplus \eta^{-d}.
\]
When $c-d \in \ZZ$, we are in the second case of Theorem~\ref{thm:hgm}, but
note that the character
$\theta$ is trivial on $\Gal(\overline{\QQ(z)}/\overline{\QQ}(z))$.
\end{proof}

\begin{corollary}
  \label{coro:specialized-inertia}
  Let $L$ be a number field containing $F$ and let $\z-spec \in L$,
  $\z-spec \neq 0,1$. Let $\id{p}$ be a prime ideal of $L$ not
  dividing $N$. Then the image of the inertial group $I_{\id{p}}$ on $\hgms$ is spanned (up to conjugation) by $M_0^{v_{\id{p}}(\z-spec)}$.
\end{corollary}
\begin{proof}
  Mimics the previous theorem, noting that the Hecke character
  $\JacMot((-a,-b,c,d),(c-b,d-a))^{-1}$ and the quadratic character
  $\twist_N$ are only ramified at primes dividing $N$.
\end{proof}

An analogous result holds for the inertia group at $1$ and
$\infty$. Define the matrices
  \[
    M_\infty =
    \begin{cases}
      \begin{pmatrix}
        e^{a} & 0\\
        0& e^{b}
      \end{pmatrix} & \text{ if } a-b \not \in \ZZ,\\
      \begin{pmatrix}
        e^{a} & 1\\
        0 &e^{a}
      \end{pmatrix} & \text{ if }a-b \in \ZZ.
    \end{cases} \qquad
    M_1 =
    \begin{cases}
      \begin{pmatrix}
        e^{\delta} & 0\\
        0& 1
      \end{pmatrix} & \text{ if } \delta:=a+b-c-d \not \in \ZZ,\\
      \begin{pmatrix}
        1 & 1\\
        0 &1
      \end{pmatrix} & \text{ otherwise}.
    \end{cases}
  \]

\begin{theorem}
The image of the inertia group at $z=\infty$ (respectively at $z=1$) under $\rho_{\ell}$ is
spanned (up to $\GL_2$-conjugation) by the
matrix $M_\infty$ (respectively $M_1$).
\label{thm:HGM-1-infty}
\end{theorem}

\begin{proof}
  Mimics the last theorem. For the first statement, the local
  parameter is $1/z$. The change of variables $x'=zx$, $y'=(1/z)^{b}y$
  (on the equation for $\bfC$) gives the component
  \[
    \C_1:y'^N=x'^A(1-x')^C.
  \]
  Similarly, the second component is obtained by setting $y'=(1/z)^{a}y$, with equation
  \[
    \C_2: y'^N = (-1)^C x^{A+B}(1-x)^B.
  \]
  The result then follows from the same arguments as in the last
  theorem.  For the second statement, consider the two components of
  $\widetilde{\bfC}$ given by
  \[
   \C_1: y^N = x^A(1-x)^{B+C}, \qquad \C_2: y'^N=x'^B (1-x')^C,
\]
where the second curve is obtained by looking at the special fiber of
the curve obtained via the change of variables
$x'=\frac{(x-1)}{(z-1)}$, $y'=(1-z)^{-\frac{B+C}{N}}y$. The result
follows from the fact that $B+C = \delta$ and that the twist needed to
relate $\bfC$ and $\widetilde{\bfC}$ is unramified at $z-1$.
\end{proof}

\section{Stable models and the trace of Frobenius}
\label{section:frob-trace}
The goal of the present
appendix is to give a formula for the trace of a Frobenius element
$\Frob_{\id{p}}$ when the motive has good reduction at $\id{p}$ but
the equation used to define Euler's curve is singular at
$\id{p}$. There are three different cases, depending on whether the
specialization $\z-spec$ reduces to $0$, $1$ or $\infty$ modulo
$\id{p}$. More concretely, let $(a,b),(c,d)$ be generic rational
parameters and let $N$ be their least common denominator. Let
$F=\Q(\zeta_N)$ and let $\id{p}$ be a prime ideal of $F$ not dividing
$N$. Consider the following three cases:
\begin{enumerate}
\item The monodromy matrix $M_0$ has finite order $r_0$,
  $v_{\id{p}}(\z-spec)$ is positive and divisible by $r_0$.
  
\item The monodromy matrix $M_1$ has finite order $r_1$,
  $v_{\id{p}}(\z-spec-1)$ is positive and divisible by $r_1$.
\item The monodromy matrix $M_\infty$ has finite order $r_\infty$,
  $v_{\id{p}}(\z-spec)$ is negative and divisible by $r_\infty$.
\end{enumerate}
The following three theorems describe the value of the trace of
$\Frob_{\id{p}}$ in each case (assuming an extra technical condition
that holds in most instances).
\begin{theorem}
  \label{thm:trace-p-mid-z}
  Let $(a,b), (c,d)$ be rational generic parameters and let $N$ be
  their least common denominator. Let $\z-spec\in \QQ$. Suppose that
  condition (1) is satisfied.  Write
  $\z-spec = p^{v_{\id{p}}(\z-spec)}\tilde{\z-spec}$.  Suppose that the following
  two extra hypotheses hold:
  \[
    \gcd(N,(d-b)N,(b-c)N)=1, \quad \text{and} \quad \gcd(N,(d-c)N,(a-d)N)=1.
  \]
  Then the trace of $\Frob_{\id{p}}$
  acting on $\hgms$ equals
  \begin{multline}
    -\gent(-1)^{(d-b)(\normid{p}-1)}\JacMot((-a,-b,c,d),(c-b,d-a))(\id{p})^{-1}\cdot \\
    \left(\gent(\tilde{\z-spec})^{dN}J(\gent^{(d-b)N},\gent^{(b-c)N})+\gent(-1)^{(b-c)N}\gent(\tilde{\z-spec})^{cN}J(\gent^{(d-c)N},\gent^{(a-d)N})\right),    
  \end{multline}
 where $J(\chi_1,\chi_2)$ denotes the usual Jacobi sum.
\end{theorem}
\begin{proof}
  The proof follows standard arguments (see \cite{wewers}, \cite{Tim}
  and \cite{MR4625990}). Let $A = (d-b)N$, $B=(b-c)N$, $C=(a-d)N$ and
  $D=dN$. Then Euler's curve is defined by the equation
  \begin{equation}
    \label{eq:equation-C}
\C:y^N=x^A(1-x)^B(1-\z-specx)^Cp^{v_{\id{p}}(\z-spec)D}\tilde{\z-spec}^D.    
  \end{equation}
  The order of $M_0$ is $\lcm\{\den(c),\den(d)\}$, hence the
  hypothesis of being in case (1) implies that both $v_{\id{p}}(\z-spec)c$
  and $v_{\id{p}}(\z-spec)d$ are integers. The change of variables
  $y \to yp^{-v_{\id{p}}(\z-spec)d}$ is defined over $F$. The new model
  reduces modulo $\id{p}$ to the curve
  \[
  \C_1: y^N = x^A(1-x)^B\tilde{\z-spec}^D.
\]
The first extra hypothesis implies that $\C_1$ is irreducible. Since
the parameters are generic, $N \nmid A$ and $N\nmid B$. Furthermore,
$N \nmid A+B=N(d-c)$ because the matrix $M_0$ has finite order. Then
$\dim(\Jac(\C_1)^{\text{new}})=\frac{\phi(N)}{2}>0$, so $\C_1$ is a
component of the stable model of $\C$ (see \cite{wewers} for more details).

To find the other component, consider the change of variables
$x \to xp^{-v_{\id{p}}(\z-spec)}$, which transforms \eqref{eq:equation-C} into
\[
y^Np^{v_{\id{p}}(\z-spec)(A+B-D)}=x^A(p^{v_{\id{p}}(\z-spec)}-x)^B(1-\tilde{\z-spec}x)^C\tilde{\z-spec}^D.
\]

From its definition, $A+B-D = -cN$ so after the change of variables
$y \to yp^{v_{\id{p}}(\z-spec)c}$ we get an equation whose reduction
modulo $\id{p}$ equals
\[
\C_2:y^N=(-1)^Bx^{A+B}(1-\tilde{\z-spec}x)^C\tilde{\z-spec}^D.
\]
The second extra hypothesis implies that $\C_2$ is irreducible. The
genericity condition on the parameters implies that $N \nmid C$,
$N \nmid A+B$ and $N\nmid A+B+C=N(a-c)$ so
$\dim(\Jac(\C_2)^{\text{new}})=\frac{\phi(N)}{2}$ and $\C_2$ is
another component of the stable reduction of $\C$. Then the Galois
representation of the decomposition group at $\id{p}$ acting on
$H^1_{\text{\'et},\chi}(\Cnsp,\Q_\ell)$ is the direct sum of the
contribution from $\C_1$ and that from $\C_2$ (both being abelian
varieties with complex multiplication by $\Z[\zeta_N]$).

To compute the trace of $\Frob_{\id{p}}$ acting on $\C_1$
(respectively $\C_2$), let $\gent$ be as in
\eqref{eq:gen-defi}. Then we can apply
Theorem~\ref{thm:trace-equality} to $\C_1$ to get that the trace of
$\Frob_{\id{p}}$ acting on $H_{et}^{1, \zeta_N}(\Cnsp_1, \QQ_{\ell})$ equals
\[
  -\sum_{x \in
    \F_q}\gent(x^A(1-x)^B\tilde{\z-spec}^D)=-\gent(\tilde{\z-spec})^DJ(\gent^A,\gent^B).
\]
The same computation for $\C_2$ gives that the trace of
$\Frob_{\id{p}}$ on the (new part of its) \'etale cohomology equals
\[
-\gent(-1)^B\gent(\tilde{\z-spec})^{D-A-B}J(\gent^{A+B},\gent^C).
\]
The result follows from the definition of $A,B,C$ and $D$ together
with the relation between the hypergeometric motive and Euler's curve
given in~(\ref{eq:motive-definition}).
\end{proof}

With the same technique one can prove the following results for the other two cases.

\begin{theorem}
  \label{thm:trace-p-mid-z-1}
  Let $(a,b), (c,d)$ be rational generic parameters and let $N$ be
  their least common denominator. Let $\z-spec\in \QQ$. Suppose that
  condition (2) is satisfied. Let $v = v_{\id{p}}(\z-spec-1)$ and write
  $\z-spec=1+p^v\tilde{\z-spec}$. Suppose that the following two extra
  hypotheses hold:
  \[
    \gcd(N,(d-b)N,(a+b-c-d)N)=1,\quad \text{and}\quad \gcd(N,(b-c)N,(a-d)N)=1.
  \]
  Then the trace of $\Frob_{\id{p}}$
  acting on $\hgms$ equals
  \begin{multline}
    -\gent(-1)^{(d-b)(\normid{p}-1)}\JacMot((-a,-b,c,d),(c-b,d-a))(\id{p})^{-1}\cdot \\
   \left(J(\gent^{(d-b)N},\gent^{(a+b-c-d)N})+\gent(-1)^{(a-d)N}\gent(\tilde{\z-spec})^{(a-b)N}J(\gent^{(d-b)N},\gent^{(a-d)N})\right).
  \end{multline}
\end{theorem}

\begin{theorem}
  \label{thm:trace-p-mid-den-z}
  Let $(a,b), (c,d)$ be rational generic parameters and let $N$ be
  their least common denominator. Let $\z-spec\in \QQ$. Suppose that
  condition (3) is satisfied.  Write
  $\z-spec = p^{v_{\id{p}}(\z-spec)}\tilde{\z-spec}$.  Suppose that the following
  two extra hypotheses hold:
  \[
    \gcd(N,(a-b)N,(b-c)N)=1,\quad \text{and}\quad \gcd(N,(d-b)N,(a-d)N)=1.
  \]
  Then the trace of $\Frob_{\id{p}}$
  acting on $\hgms$ equals
  \begin{multline}
    -\gent(-1)^{(d-b)(\normid{p}-1)}\JacMot((-a,-b,c,d),(c-b,d-a))(\id{p})^{-1}\cdot \\
   \left(\omega(\tilde{\z-spec})^{bN}J(\omega^{(d-b)N},\omega^{(a-d)N})+\omega(-1)^{(a-d)N}\omega(\tilde{\z-spec})^{aN}J(\omega^{(a-b)N},\omega^{(b-c)N})\right).
  \end{multline}
\end{theorem}
\bibliographystyle{plain}
\bibliography{biblio2}
\end{document}